\documentclass[12pt]{article}
 \usepackage[margin=0.912in]{geometry} 

\usepackage{mathrsfs} 
\usepackage{kpfonts,comment}

\usepackage{kz_style}

\newcommand{\issue}[1]{{\color{red}#1}}

\newcommand{\remind}[1]{{\color{blue}#1}}

\title{\Large 
Policy Optimization for $\cH_2$ Linear Control 
with $\cH_\infty$ Robustness Guarantee:
Implicit Regularization and Global Convergence}

\begin{document}
\author{Kaiqing Zhang \and    Bin Hu \and  Tamer Ba\c{s}ar\thanks{The authors are with the Department of Electrical and Computer Engineering \&  Coordinated Science Laboratory, University of Illinois at Urbana-Champaign. Email: \{kzhang66,~binhu7,~basar1\}@illinois.edu.}}
\date{{\normalsize Oct. 17, 2019\qquad Revised: Feb., 2021}}
%}
%\normalsize
\maketitle

\vspace{-7pt}
\begin{abstract}

Policy optimization (PO) is a key ingredient for modern reinforcement learning (RL). 
For control design, certain \emph{constraints} are usually enforced on the policies to  optimize, accounting for either the stability, robustness, or safety concerns on the system.  
Hence, PO is by nature a \emph{constrained (nonconvex) optimization} in most cases, whose global convergence is challenging to analyze in general. 
More importantly, some constraints that are safety-critical, e.g., the closed-loop stability, or the $\cH_\infty$-norm constraint that guarantees the system  robustness, can  be difficult to enforce on the controller being learned as the PO methods proceed. Recently, policy gradient   methods have been shown to converge to the global optimum of  linear quadratic regulator (LQR), a classical optimal control problem, without regularizing/projecting the control iterates onto the stabilizing set  \citep{fazel2018global,bu2019LQR}, the (implicit) feasible set of the problem.  
This striking result is built upon the  property that the cost function is \emph{coercive}, ensuring that the iterates remain feasible {and strictly separated from the infeasible set as the cost decreases}. In this paper, we study  the convergence theory of PO for $\cH_2$ linear control with $\cH_\infty$-norm \emph{robustness guarantee}, for both discrete- and continuous-time settings.   This general framework  includes \emph{risk-sensitive} linear control as a special case. One significant new feature of this problem is the \emph{lack of coercivity}, i.e., the cost may have \emph{finite} value around the boundary of the robustness constraint set,  breaking the existing analysis for LQR. 
Interestingly, among the three proposed PO methods,  two of them enjoy the \emph{implicit regularization} property, i.e., the iterates preserve the $\cH_\infty$ robustness  constraint automatically, as if they are regularized by the algorithms. Furthermore, despite the nonconvexity  of the problem, we show that these algorithms converge to the \emph{globally optimal} policies with \emph{globally sublinear} rates, {avoiding all suboptimal stationary points/local minima,} and with \emph{locally (super-)linear} rates under certain conditions.  
To the best of our knowledge, our work offers   the first results on {the implicit regularization property and global convergence of} PO methods for robust/risk-sensitive control. 
\end{abstract} 

\begin{comment}
\noindent\remind{TO DO LIST:
\begin{itemize}  
%	\item Add discussions of ``model-free'' and LQ games
	\item Add simulations
\end{itemize}
}

\remind{One possible flow:
\begin{itemize}
	\item We start with LEQR as an \emph{motivating example}, to study \emph{risk-sensitive} optimal control/optimal control with some robustness or uncertainty concern. We will showcase all our results on LEQR here (also answer Doyle' question).
	\item We then make connection of this problem to mixed-design, introducing a bigger picture: this LEQR is one example of them (one of the objectives in  Mustafa's paper \citep{mustafa1991lqg}). Also, both continuous and discrete-time settings are introduced. 
	\item Under this ``mixed-design perspective'', we prove the nice properties of algorithms for both settings, and establish global convergence.
	\item In order to do model-free, we point out the challenge in sampling, and connect it to the game. Then gives a pseudo-code for model-free algorithms. \issue{Note that we should mention that the technique here (perturbation theory) regarding the feasibility of ``next-step'' should be applicable to the maxmin game in \citep{zhang2019policy} as well.}
\end{itemize} 
}
\end{comment}

\section{Introduction}\label{sec:intro}

%{\color{red} I haven't added any references.

%This paragraph needs to be rewritten.
%Recently, reinforcement learning has achieved impressive successes on various tasks. (Kaiqing, you should take care of this paragraph.)
%A popular class of algorithms are policy based. This motivates the recent trend in studying linear control problems as benchmark for RL.
%}

Recent years have witnessed  tremendous  success of reinforcement learning (RL) in various sequential decision-making applications \citep{silver2016mastering,OpenAI_dota,alphastarblog} and continuous  control tasks   \citep{lillicrap2015continuous,schulman2015high,levine2016end,recht2019tour}. Interestingly, most  successes hinge on the algorithmic framework of \emph{policy optimization} (PO), umbrellaing  policy gradient (PG) methods \citep{sutton2000policy,kakade2002natural}, actor-critic methods \citep{konda2000actor,bhatnagar2009natural}, trust-region \citep{schulman2015trust} and proximal PO \citep{schulman2017proximal} methods, etc. 
This inspires an increasing  interest  in studying the convergence theory, especially global convergence to optimal policies, of PO methods; see recent progresses in  both classical RL contexts \citep{bhandari2019global,zhang2019global,wang2019neural,agarwal2019optimality,shani2019adaptive}, and continuous control benchmarks \citep{fazel2018global,bu2019LQR,malik2019derivative,tu2018gap,zhang2019policy,matni2019self}.

%\issue{XXXXXX add more control RL examples XXXXXX} 

Indeed, PO provides a general framework for control design.\footnote{Hereafter, we will mostly adhere to the terminologies and notational convention in the control literature, which are equivalent to, and can be easily translated to those in the RL literature, e.g., {cost} v.s. {reward}, {control} v.s. {action}, etc.} Consider a general control design problem for  the following  discrete-time  nonlinear dynamical system  
\[ 
x_{k+1}=f(x_k, u_k, w_k),
\]
where $x_k$ is the state, $u_k$ is the control input,  and $w_k$ is the process noise. Formally, PO is  a constrained optimization problem $\min_{K\in \mathcal{K}} \cJ(K)$, where the decision variable $K$ is determined by the controller parameterization, the cost function $\cJ(K)$ is a pre-specified control performance measure, and the feasible set $\mathcal{K}$ carries the information of the constraints on the controller $K$.  
These concepts are briefly reviewed as follows. 

\begin{itemize}
\item Optimization variable $K$: The control input $u_k$ is typically determined by a feedback law $K$ which is also termed as  a \emph{controller}.  In the simplest case where a LTI state-feedback controller is used, $K$ is parameterized as a static matrix and $u_k$  is given as $u_k=-K x_k$. Then this matrix $K$ becomes the decision variable of the PO problem. For the so-called linear output feedback case where the state $x_k$ is not directly measured, the controller can be either a memoryless mapping or an LTI dynamical system.  Hence,  $K$ can be parameterized by either a static matrix \citep{rautert1997computational} or some state/input/output matrices $(A_K, B_K, C_K, D_K)$ \citep{apkarian2008mixed}. It is also possible to deploy nonlinear controllers and parameterize $K$ as either polynomials, kernels, or deep neural networks \citep{topcu2008local, levine2016end}.   
\item Objective function $\cJ(K)$:  $\cJ(K)$ is specified to assess the performance of a given controller $K$. The cost function design is more of an art than a science. Popular choices of such cost functions include $\mathcal{H}_2$ or $\mathcal{H}_\infty$-norm (or some related upper bounds) for the resultant feedback systems \citep{zhou1996robust, skogestad2007multivariable, dullerud2013course}. For standard RL models that are based on 
 Markov decision processes (MDPs), the cost $\cJ(K)$ usually has an additive structure over time. For instance, in the  classical linear quadratic regulator (LQR) or state-feedback LQ Gaussian (LQG)  problems, the cost is $\cJ(K):=\sum_{t=0}^{\infty} \EE[x_t^\top Q x_t+u_t^\top R u_t ]$, which also has an $\cH_2$-norm interpretation \citep{zhou1996robust}. Nevertheless, $\cJ(K)$ does not necessarily have an additive structure. We will further discuss the specification of $\cJ(K)$ in \S\ref{sec:formulation}. 
\item Feasible set $\mathcal{K}$:  Constraints on the decision variable $K$ are posed to account for either the stability, robustness, or safety concerns on the system. A common, though sometimes implicit, example in continuous control tasks is the stability constraint, i.e. $K$ is required to stabilize the closed-loop dynamics \citep{makila1987computational,bu2019LQR}. 
 There are also other constraints related to robustness or safety concerns in control design \citep{skogestad2007multivariable, dullerud2013course,apkarian2008mixed}.
The constraints will naturally confine the policy search to a feasible set $\mathcal{K}$. The cost function $\cJ(K)$ is either $\infty$ or just undefined for $K\notin \mathcal{K}$.
\end{itemize}

%It  seems that a projection operator for the set $\mathcal{K}$ is required when applying policy gradient methods to solve the above constrained optimization problem  $\min_{K\in \mathcal{K}} \cJ(K)$. Unfortunately, efficient solvers for such projection are unavailable for many applications (especially in the context of learning-enable control).

To ensure the feasibility of $K$ on the fly as the PO methods proceed, projection of the iterates onto the set $\cK$ seems to be the first natural approach that comes to mind.   However, such a projection  may not be computationally efficient or even tractable. For example, projection onto the stability constraint in LQR problems can hardly be computed, as the set $\cK$ therein is well known to be nonconvex \citep{fazel2018global,bu2019topological}.
Fortunately,    such a projection is not needed to preserve the feasibility of the iterates in PG-based methods, as recently reported by  \cite{fazel2018global,bu2019LQR}. In particular, \cite{bu2019LQR} has identified that the cost of LQR has the \emph{coercive} property, such that it diverges to infinity as the controller $K$ approaches the boundary of the feasible set $\cK$. In other words, the cost of LQR serves as a barrier function on $\cK$. 
This way, the level set of the cost becomes compact, and the decrease of the cost ensures the next iterate to stay inside the level set, which further implies the stay inside $\cK$.    
This desired property is further illustrated in Figure \ref{fig:illust_hardness}(a),  where $K$ and $K'$ are two consecutive iterates, and the level set $\{\tilde{K}\given \cJ(\tilde{K})\leq \cJ(K)\}$ is always   separated from the set  $\cK^c$ by  some distance $\delta>0$. Hence, as long as $\norm{K-K'}<\delta$, the next iterate $K'$ still stays in the set $\mathcal{K}$. More importantly,  such a separation distance $\delta$ can be re-used for the next iterate, as the  next level set is at least $\delta$ away from $\cK^c$. By induction, this allows the existence of a constant stepsize that can guarantee the controllers' stability along the iterations. It is worth emphasizing that such a property is \emph{algorithm-agnostic}, in the sense that it is dictated by the cost, and independent of the algorithms adopted, as long as they  follow any descent directions of the cost.

Besides the stability constraint, 
another  commonly used one in the control  literature is the so-called \emph{$\cH_\infty$ constraints}. This type of constraints plays a fundamental role in 
robust control  \citep{zhou1996robust, skogestad2007multivariable, dullerud2013course,apkarian2008mixed} and risk-sensitive control \citep{whittle1990risk,glover1988state}. 
Based on the well-known small gain theorem \citep{zames1966input,zhou1996robust}, such constraints can be used to guarantee robust stability/performance of the closed-loop systems when model uncertainty is at presence. 
Compared with  LQR  under the stability constraint, control synthesis under the 
$\cH_\infty$ constraint leads to a fundamentally different optimization  landscape, over which the behaviors of PO methods have not been fully investigated yet. 
 In this paper, we take an initial step towards understanding the theoretical aspects of policy-based RL methods on robust/risk-sensitive control problems. 
%and the behaviors of policy gradient methods subject to such constraints are largely unknown. There lacks a thorough understanding of the impacts of $\cH_\infty$ constraints on policy optimization, and this is an important open issue for robust/risk sensitive RL.

Specifically, we establish a convergence theory for PO methods  on $\cH_2$ linear control problems with $\cH_\infty$ constraints, referred to as \emph{mixed $\cH_2/\cH_\infty$ state-feedback control design} in the robust control literature \citep{glover1988state,khargonekar1991mixed,kaminer1993mixed,mustafa1990minimum,mustafa1991lqg,mustafa1989relations,apkarian2008mixed}. As the name suggests, the goal of mixed design is to find a robust stabilizing controller that
minimizes an upper bound for the $\cH_2$-norm,  under the restriction that the $\cH_\infty$-norm on a certain input-output channel is less than a pre-specified value.  
% and the more general mixed $\mathcal{H}_2$/$\mathcal{H}_\infty$ state-feedback design .
% XXXX
%The basic idea of the mixed $\cH_2/\cH_\infty$ control is to  
% find a robust stabilizing controller that
%minimizes an upper bound for the $\cH_2$-norm under the restriction that the $\cH_\infty$-norm on a certain input-output channel is less than a pre-specified value. 
The $\mathcal{H}_\infty$ constraint is explicitly posed here to guarantee the robustness of the closed-loop system to some extent. 
This general framework also  includes  risk-sensitive linear control, modeled as 
linear exponential quadratic Gaussian (LEQG)  \citep{jacobson1973optimal, whittle1990risk} problems as a special case, when a certain upper bound of $\cH_2$-norm is used.  
Detailed formulation for such $\cH_2$ linear control with $\cH_\infty$ constraint is provided in \S\ref{sec:formulation0}. 
% while the cost function here is an upper bound on the $\cH_2$-norm and quantifies the average performance of the controller.  
%Essentially, the policy optimization for such linear control problems can be formulated as $\min_{K\in \mathcal{K}} \cJ(K)$ where $K$ parameterizes the linear state-feedback controller, $\cJ$ is some performance cost, and $\mathcal{K}$ is the set of stabilizing controllers satisfying a certain $\cH_\infty$ constraint. 
%In Section \ref{sec:formulation0}, we will present a detailed PO formulation for such linear control with $\cH_\infty$ constraints.
%The linear control with $\cH_\infty$ robustness guarantee 
%The mixed $\cH_2/\cH_\infty$ design problem 
%provides a useful benchmark for robust/risk-sensitive RL in general, since it  examines  how well PO methods handle the $\mathcal{H}_\infty$ constraint. 
In contrast to LQR, two challenges exist in the analysis of PO methods for mixed design problems.
First, by definition of  $\mathcal{H}_\infty$-norm \citep{zhou1996robust}, the constraint is defined in the frequency domain, and is hard to impose, for instance, by directly projecting the iterates in that domain, especially in the context of RL when the system model is unknown. Note that preserving the constraint bound of $\cH_\infty$-norm as the controller updates is critical in practice, since violation of it can cause catastrophic consequences on the system.   
Second, more importantly, the coercive property of LQR fails to hold for  mixed design problems, as illustrated in Figure \ref{fig:illust_hardness}(b) (and formally established later). Particularly, the cost value, though undefined outside the set $\cK$, remains \emph{finite} around the boundary of $\cK$. Hence, the decrease of cost from $K$ to $K'$ cannot guarantee that the iterate does not travel towards, and even beyond the feasibility boundary. With no strict separation between the cost level set and $\cK^c$, there may not exist a constant stepsize that induces global convergence to the optimal policy.

These two challenges naturally raise the question: does there exist any computationally tractable  PO method, which preserves the robustness constraint  along the iterations, and enjoys (hopefully global) convergence guarantees? We provide a positive answer to this  question in the present work. Our key contribution is three-fold: First, we  study  the landscape of mixed $\cH_2/\cH_\infty$  design problems for both discrete- and continuous-time settings, and propose three policy-gradient based methods, inspired by those for LQR  \citep{fazel2018global,bu2019LQR}. Second, we prove that two of them (the Gauss-Newton method and the natural PG method) enjoy the \emph{implicit regularization} property,  such that the iterates  are automatically biased to satisfy the required $\cH_\infty$ constraint. Third, we establish the global convergence of those two PO methods to the \emph{globally optimal} policy  with \emph{globally sublinear} and \emph{locally (super-)linear} rates under certain conditions, despite the nonconvexity of the problem.   In particular, the two policy search directions  always lead to convergence to the global optimum, without getting stuck at any spurious stationary point/local optima. 
Along the way, we also derive new results on   linear risk-sensitive control, i.e., LEQG problems, and discuss the connection of mixed design to zero-sum LQ dynamic games, for designing model-free versions  of the algorithms. We expect our work to help   pave the way for  rigorous understanding of PO methods for general optimal control with $\cH_\infty$ robustness guarantees.

%Unlike the stability constraint, the $\mathcal{H}_\infty$ constraint leads to more complicated optimization landscape, i.e.,  
% the cost values on the boundary of the $\cH_\infty$ feasible set $\mathcal{K}$ can be finite such that the coercivity property used in the LQR case does not hold. \issue{XXXXXXX} 
%\issue{XXXXXXX}
%However, this is not the case for the mixed $\cH_2/\cH_\infty$ control. Notice that the cost for the mixed $\cH_2/\cH_\infty$ control and the LEQR is not defined outside $\mathcal{K}$, and it is crucial to ensure the entire trajectories of the policy optimization methods to stay in $\mathcal{K}$. As shown in Figure \ref{fig:illust_hardness}(b), the cost is not coercive any more, and the cost value on the boundary of $\mathcal{K}$ can be finite. There is no strict separation between the level set and the set $\mathcal{K}^c$. Hence it is possible for a policy optimization method to travel towards the boundary of $\mathcal{K}$ and even go outside $\mathcal{K}$ while decreasing the cost value. Can policy optimization methods (without projection steps) be guaranteed to stay in $\mathcal{K}$ and eventually find the global optimal solution for linear control with $\cH_\infty$ constraints? 
%This is the key question that we want to answer in this paper.

\begin{figure*}[!t]
	\centering
	\begin{tabular}{cc}
		\hskip0pt\includegraphics[width=0.4\textwidth]{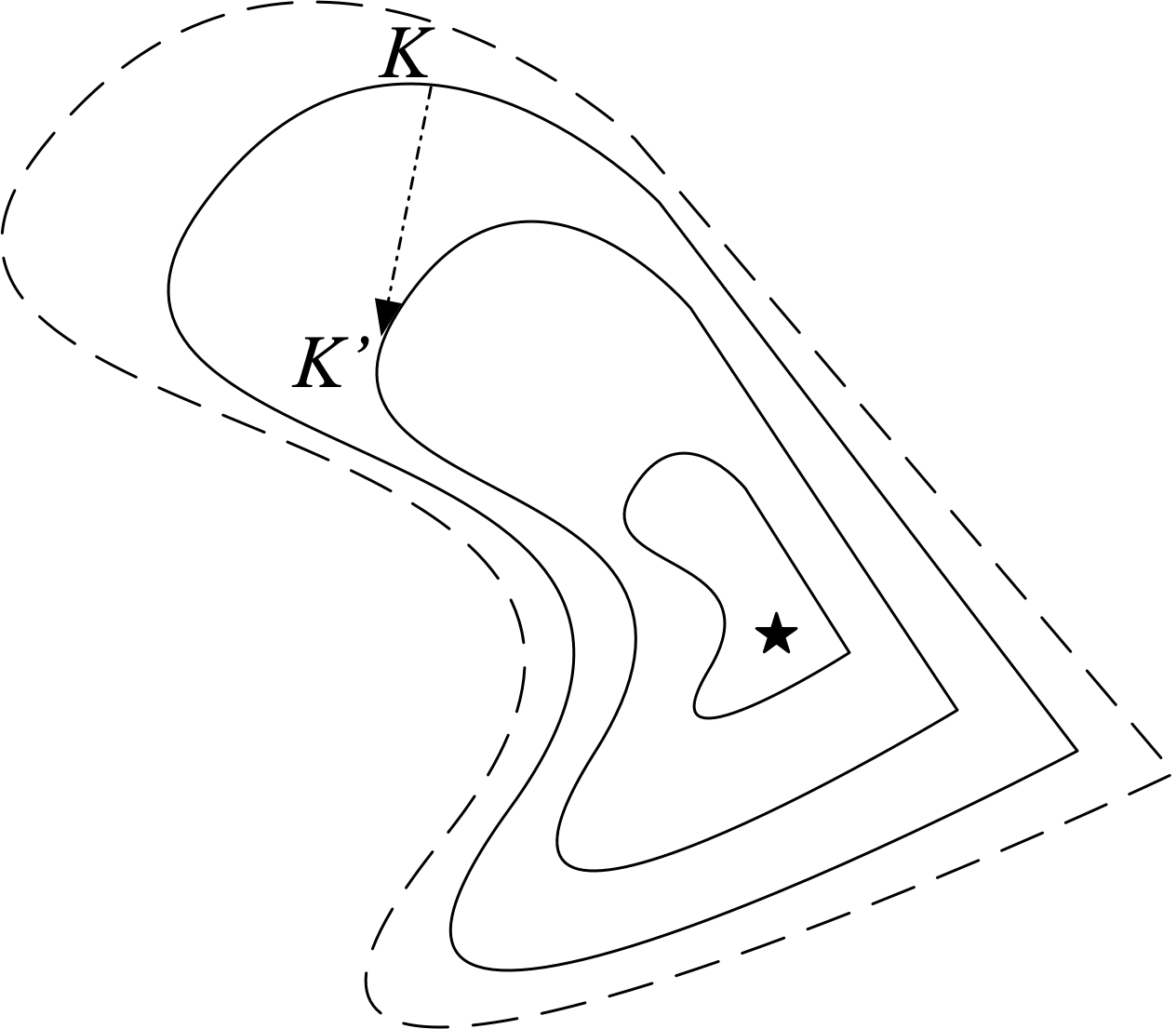}
		&
		\hskip16pt\includegraphics[width=0.4\textwidth]{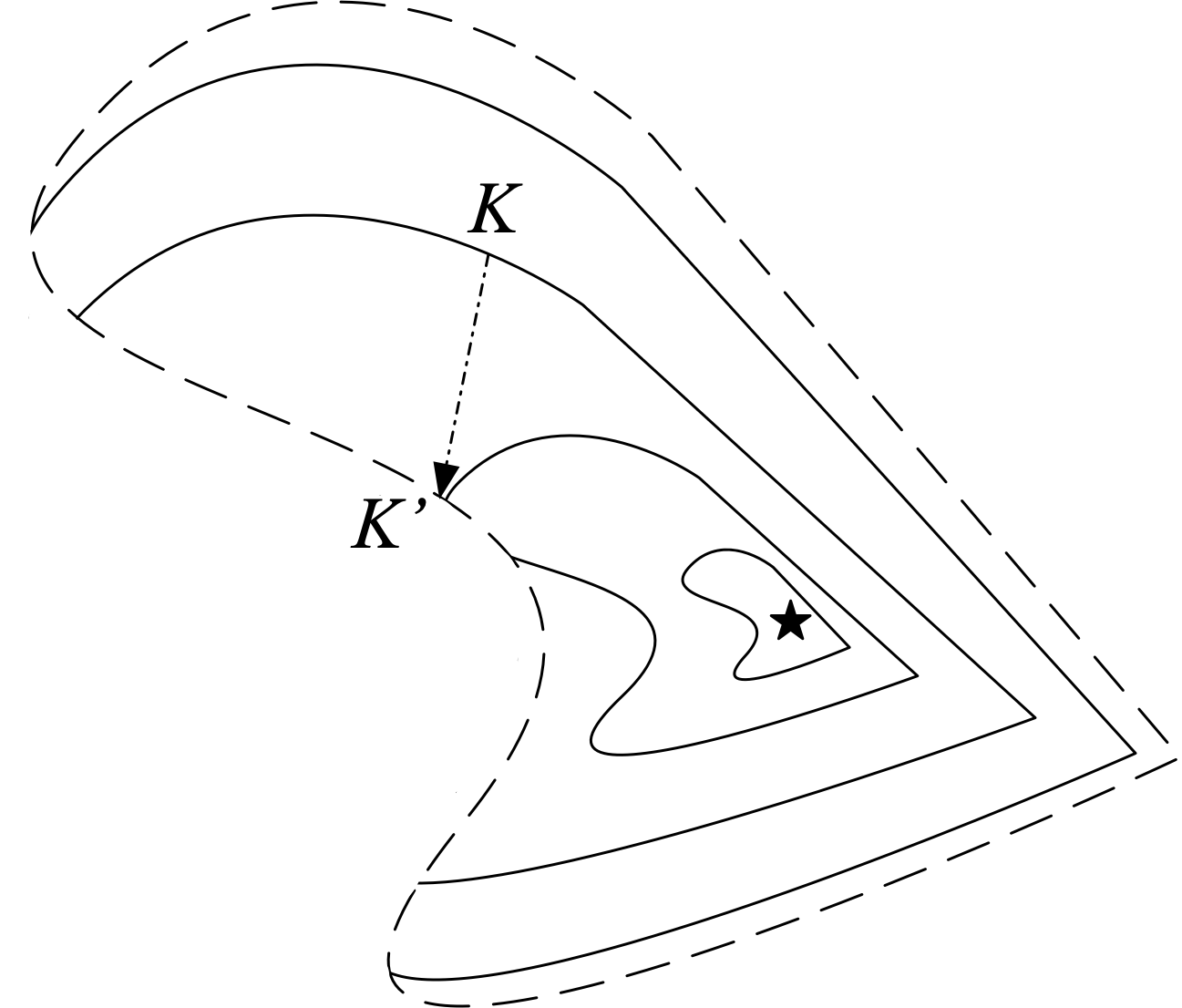}
		\\
		\hskip-8pt(a) Landscape of LQR & \hskip 50pt(b) Landscape of Mixed $\cH_2/\cH_\infty$  Control
	\end{tabular}
	\caption{Comparison of the landscapes of LQR and mixed $\cH_2/\cH_\infty$ control design that illustrates the hardness of showing convergence of the latter. The dashed lines  represent the boundaries of the constraint sets $\mathcal{K}$. For (a) LQR, $\mathcal{K}$ is the set of all linear stabilizing state-feedback controllers; for (b) mixed $\cH_2/\cH_\infty$ control,  $\mathcal{K}$ is set of all linear stabilizing state-feedback controllers satisfying an extra $\cH_\infty$ constraint on some input-output channel. The solid lines represent the contour lines of the cost $
	\cJ(K)$. $K$ and $K'$ denote the control gain of two consecutive iterates; {\scriptsize $\bigstar$} denotes the global optimizer.}
	\label{fig:illust_hardness}
\end{figure*}

%The key of our contribution is that we prove that two proposed PO methods motivated by  \citep{fazel2018global,bu2019LQR} (the Gauss-Newton method and the natural policy gradient method) have an \emph{implicit regularization} property,  such that the iterates of these two methods are automatically biased towards the points satisfying the required $\cH_\infty$ constraint. 
%Building upon this property, we further prove that these two methods are guaranteed to converge to the global optimal controller,  despite  the nonconvexity of the problem.

Note that the concept of {(implicit) regularization} has been adopted in many recent works on nonconvex optimization, including training  neural networks \citep{allen2018learning,kubo2019implicit},  phase retrieval \citep{chen2015solving,ma2017implicit}, matrix completion \citep{chen2015fast,zheng2016convergence}, and blind deconvolution \citep{li2019rapid}, referring to any scheme that biases the search direction of the optimization algorithms. The term \emph{implicit} emphasizes that the  algorithms without regularization may behave as if they are regularized. This property 
  has been advocated as an important feature of  gradient-based methods for solving  aforementioned nonconvex problems.    We emphasize that it is a feature of both the \emph{problem} and the \emph{algorithm}, i.e., it holds for  certain algorithms that solve certain problems.  This is precisely  the case in the present work. The specific search  directions of the Gauss-Newton and the natural PG methods bias the iterates  towards the set of the stabilizing controllers satisfying  the $\cH_\infty$ constraint, although no explicit  regularization, e.g., projection, is adopted, which contrasts to that the stability-preserving  of PO methods for LQR problems  is algorithm-agnostic.   
% In addition, we also discuss the model-free implementations of policy optimization methods and their connections to linear quadratic games.
To the best of our knowledge, our work appears to be the first studying the implicit regularization properties of PO methods for \emph{learning-based control} in general.  
%, motivated by the  linear control problem with $\cH_\infty$ robustness guarantee. 
%It is our hope that this serves as an important first step towards rigorous understanding of PO for robust/risk-sensitive control in general. 

%\issue{ISSUE: Can we write the ``contributions'' more explicitly, in a separate paragraph?}

%\paragraph{Related Work.} A lot of work has been done on LQR. Policy optimization is an old idea. Global convergence for LQR is new. For implicit regularization, a lot of work in the machine learning field. A lot of work on LEQR (risk-sensitive linear control) and $\mathcal{H}_2$/$\mathcal{H}_\infty$ mixed design.

%\paragraph{Future work.}
%
%\begin{itemize}
%\item Implicit bias of other optimization methods
%\item Model-based v.s. model free for robust control
%\item Optimization landscape of $\mathcal{H}_\infty$ control
%\end{itemize}

\paragraph{Related Work.}~\\ 

\vspace{-8pt}
\noindent \textbf{Mixed $\cH_2/\cH_\infty$   \& Risk-Sensitive Control.} 
The history of mixed $\cH_2/\cH_\infty$ control design dates back to the seminal works \cite{bernstein1989lqg,khargonekar1991mixed} for continuous-time settings, built upon the Riccati equation approach and the convex optimization approach, respectively. Such a formula can be viewed as a surrogate/sub-problem of the more challenging $\cH_\infty$-control problem,  where the goal is to find the optimal controller that minimizes the $\cH_\infty$-norm \citep{doyle1988state}. These formulation and approaches were then investigated for  discrete-time systems  in \cite{mustafa1991lqg,kaminer1993mixed}. 
A non-smooth constrained optimization perspective for solving mixed design problems was adopted in \cite{apkarian2008mixed}, with proximity control algorithm designed to handle the constraints explicitly,  and convergence guarantees to stationary-point controllers. 
Numerically, there also exist other packages for multi-objective $\cH_2/\cH_\infty$  control \citep{gumussoy2009multiobjective,arzelier2011h2} that are based on non-smooth nonconvex optimization. However, in spite of  achieving impressive numerical performance, these methods have no theoretical guarantees for either the global convergence or the $\cH_\infty$-norm constraint violation. It is also not clear yet how these methods can be made \emph{model-free}. 
%A non-smooth constrained optimization perspective for the mixed $\cH_2/\cH_\infty$ design problem has been presented in  \cite{apkarian2008mixed}, and the proposed proximity control algorithm was designed to handle the constraints explicitly. 
On the other hand, risk-sensitive control with exponential performance measure was  originally proposed by \cite{jacobson1973optimal} for the linear quadratic case, and then generalized in \cite{whittle1981risk,fleming1997risk,borkar2002risk,jaskiewicz2007average}.
{Under certain conditions on the cost, noise, and risk factor, convex optimization perspectives on   linear risk sensitive control have been reported in \cite{dvijotham2014universal,dvijotham2014convex}}. Recently, first-order optimization methods have also been applied to finite-horizon risk sensitive nonlinear control, but the control inputs (instead of the policy) are treated as decision variables \citep{roulet2019convergence}. Convergence to stationary points was shown therein for the iterative LEQG algorithm.  
 Interestingly, there is a relationship between mixed design and risk-sensitive control, as established in \cite{glover1988state,whittle1990risk}. These two classes of problems can also be unified with maximum-entropy  $\cH_\infty$ control \citep{glover1988state,mustafa1989relations} and zero-sum  dynamic games \citep{jacobson1973optimal,bacsar1995h}.  
 {Besides these direct controller/policy search methods, general mixed $\cH_2/\cH_\infty$ control can also be tackled via Youla-parameterization based approaches \citep{boyd1988new,scherer1995multiobjective,chen1995linear,hindi1998multiobjective,rotstein1998exact}, which lead to convex programming problems that can be solved numerically. However,  actual implementation of these approaches   either require a finite-horizon \emph{truncation} of system impulse responses, which  loses optimality guarantees \citep{boyd1988new}, or require solution of (a large enough sequence of  \citep{chen1995linear,rotstein1998exact}) semi-definite programs   or linear matrix inequalities with lifted dimensions \citep{scherer1995multiobjective,chen1995linear,rotstein1998exact,hindi1998multiobjective}, which may not be computationally efficient for  large-scale dynamical systems.  More importantly, it is not clear yet how to implement these approaches in the data-driven regime, without identifying the model. In contrast, the  direct search methods can easily be made {model-free}, see e.g., our recent attempt \cite{zhang2021derivative} for robust control design.}  
% \issue{Many of the above papers covered the general output feedback case while our paper focuses on the first-order optimization perspective for the state feedback case.  

\vspace{7pt}
\noindent \textbf{Constrained MDP \& Safe RL.}
The mixed design formulation is also pertinent to constrained dynamic control  problems,   usually modeled as constrained 
MDPs (CMDPs) \citep{altman1999constrained}.  However,  the constraint in CMDPs is generally imposed on either the expected long-term cost \citep{borkar2005actor,achiam2017constrained,chow2018lyapunov}, which shares the additive-in-time structure as the objective, or some risk-related constraint \citep{di2012policy,chow2015risk,chow2017risk}. Those constraints are in contrast to the $\cH_\infty$ robustness constraint considered here.  
Under the CMDP model,   various \emph{safe RL} algorithms, especially PO-based ones, have been developed \citep{borkar2005actor,geibel2005risk,di2012policy,chow2015risk,achiam2017constrained,chow2018lyapunov,yu2019convergent}. It is worth mentioning that except \cite{achiam2017constrained,chow2018lyapunov}, other   algorithms cannot guarantee the constraint to be satisfied during the learning  iterations, as opposed to our  on-the-fly implicit regularization property. 
Recently, safety constraint has also been incorporated into the LQR model for \emph{model-based} learning control \citep{dean2019safely}. Several recent   model-based safe RL algorithms include \cite{garcia2012safe,aswani2013provably,akametalu2014reachability,berkenkamp2017safe}.

%\cite{ding2019aggressive} considers distributed structure constraints, and study the convergence using first-order methods.
%
%Global convergence of PG methods can also be established in  \cite{gravell2019learning}, with analysis similar to the deterministic case in \cite{fazel2018global}.  
%Moreover,
\vspace{7pt}
\noindent \textbf{PO for LQR.}
PO for LQR stemmed from the adaptive policy iteration algorithm in \cite{bradtke1994adaptive}. Lately, studying the global convergence of   policy-gradient based methods for  LQR \citep{fazel2018global,tu2018gap,malik2019derivative,bu2019LQR,gravell2019learning, mohammadi2019global, mohammadi2019convergence, jansch2020convergence, venkataraman2019recovering} has   drawn increasing   attention. 
Specifically, \cite{fazel2018global} first identified the landscape of PO for LQR problems that stationary point implies global optimum, which motivated the development of first-order methods for solving this nonconvex problem. A more comprehensive landscape characterization was then reinforced in \cite{bu2019LQR}, where the \emph{coercive} property of LQR cost was explicitly mentioned.    
Based on these, \cite{malik2019derivative} advocated two-point zeroth-order methods  to improve the sample complexity of model-free PG; \cite{yang2019global} proposed actor-critic algorithms with non-asymptotic   convergence guarantees; \cite{tu2018gap} compared the  asymptotic behavior of model-based and model-free PG methods, with focus on \emph{finite-horizon} LQR problems. 
Recently, the continuous-time setup has been considered in \cite{mohammadi2019global, mohammadi2019convergence}, and the extensions to Markov jump linear quadratic control have been presented in \cite{jansch2020convergence}.
Moreover,
\cite{gravell2019learning} considered  LQR  with multiplicative noises, in order to improve the controller's robustness.   The robustness issue for the output feedback case
has been further discussed in \cite{venkataraman2019recovering}. 

\vspace{7pt}
\noindent \textbf{Robust RL.}
Robustness with respect to the model uncertainty/misspecification has long been a significant concern in RL. Indeed, the early attempt for robust RL  was based on the notion of $\cH_\infty$ robustness considered here  \citep{morimoto2005robust}, where the uncertainty was modeled as the control of an adversarial agent playing against the nominal controller. This game-theoretic perspective, as we will also discuss in \S\ref{subsec:connection_to_games}, enabled the development of actor-critic based algorithms therein, though without theoretical analysis. Such an idea has recently been carried forward in the empirical work \cite{pinto2017robust}, which proposed PO methods alternating between the two agents. Another line of work follows the \emph{robust MDP} framework \citep{nilim2005robust,iyengar2005robust}, with RL algorithms developed in \cite{lim2013reinforcement,lim2019kernel,chen2019action,mankowitz2019robust}. However, these algorithms apply to only tabular/small-scale MDPs (not continuous control tasks) and/or do not belong to PO methods that guarantee robustness during learning. 
More recently, linear control design against adversarial disturbances has also been placed in the online learning context \citep{cohen2018online,agarwal2019online,agarwal2019logarithmic} to achieve nearly-optimal regret,  where  either the dynamics or the cost functions are adversarially changing.  Model-based  methods also exist for  continuous control tasks \citep{berkenkamp2015safe,dean2019robust}.

\paragraph{Notation.} For two  matrices $A$ and $B$ of proper dimensions, we use $\tr{(AB)}$  to denote the trace of $AB$. For any  $X\in\RR^{d\times m}$, we use $\vect(X)\in\RR^{dm},~\rho(X),~\|X\|,~\|X\|_F,~\sigma_{\max}(X),~\sigma_{\min}(X)$ to denote the vectorization, the spectral radius, the operator norm, the Frobenius norm, the largest and smallest singular values of   $X$,  respectively.  
If $X\in\RR^{m\times m}$ is  square, we use $X>0$ (resp. $X\geq 0$) to denote that $X$ is positive definite  (resp. nonnegative definite), i.e., for any nonzero vector $v\in\RR^n$, $v^\top {(X^\top+X)}v>0$ (resp. $v^\top {(X^\top+X)}v\geq 0$). 
%If $X$ is additionally symmetric, we use $X>0$, $X\geq 0$ to denote that $X$ is positive definite and nonnegative definite, respectively.
 For a symmetric matrix $X$, we use $\lambda_{\max}(X)$ and $\lambda_{\min}(X)$ to denote, respectively,  its largest and smallest eigenvalues. 
We use $\otimes$ to denote the Kronecker product. We use $I$ and $0$ to denote the identity matrix and all-zero matrix of proper dimensions. We use $\cN(\mu,\Sigma)$ to denote the Gaussian distribution with mean $\mu$ and covariance matrix $\Sigma$.   
 For any  integer $m>0$, we use $[m]$ to denote the set of integers $\{1,\cdots,m\}$. 
%s For $f,g\geq 0$, 
% we use $f=o(g)$ to represent that 
%% $f$ is negligible compared to $g$  
% as $g\to 0$, i.e., $\lim\limits_{g\to 0} f/g=0$.  
For any complex number $c\in\mathfrak{C}$, we use $\Re c$ to denote the real part of $c$. We use 
 $G:=\mleft[
\begin{array}{c|c}
  A & B \\
  \hline
  C& D
\end{array}
\mright]$ to denote the input-output transfer function of the following state-space   linear dynamical systems:
\begin{flalign*}
 &{\rm \textbf{Discrete-Time:}}~~~~~~\qquad\quad\qquad x_{t+1}=   A x_t + B u_t,\qquad\qquad~ z_t= C x_t+ D u_t,&\\
 &{\rm \textbf{Continuous-Time:}}~~~~~\qquad \qquad\quad\dot{x}=Ax+Bu,\qquad\qquad\quad z=Cx+D u, &
\end{flalign*}
%The transfer function $\cT_{G}$ of $G$ is thus defined 
which can also be written as $G(z)=C(zI-A)^{-1}B+D$ and $G(s)=C(sI-A)^{-1}B+D$ for discrete- and continuous-time systems, respectively. The  $\cH_\infty$-norm $\|G\|_{\infty}$ is then defined as  
\begin{flalign}
 &{\rm \textbf{Discrete-Time:}}~~~~~~~\qquad\qquad\qquad \|G\|_\infty:=\sup_{\theta\in[0,2\pi)}~\lambda_{\max}^{1/2}[G(e^{-j\theta})^\top G(e^{j\theta})],\label{equ:def_dis_Hinf_norm}&\\
 &{\rm \textbf{Continuous-Time:}}~~~~~\qquad \qquad\quad\|G\|_\infty:=\sup_{\omega}~\sigma_{\max}[G(j\omega)]. \label{equ:def_cont_Hinf_norm} &
\end{flalign}
%\#
%  \|\cT_{G}(z)\|_\infty:=\sup_{\theta}~\lambda_{\max}^{1/2}[\cT(K)^\top (e^{-j\theta})\cT(K)(e^{j\theta})]\label{equ:def_dis_Hinf_norm}\\
%  \|\cT_{G}(s)\|_\infty:=\sup_{\omega}~\sigma_{\max}[\cT(K)(j\omega)]\label{equ:def_cont_Hinf_norm}.
%\#

%\remind{To Be Changed/Added After Meeting:
%\begin{itemize}
%	\item \newlytyped{Compress Sec. 2, 3, into one section with three subsections: motivating example (emphasize the Hinf bound more, since this now becomes the ``motivating'' part); formulation; bounded real lemma (or change the name to sth regarding ``imposability'').} 
%	\item \newlytyped{Take out continuous-time part and make it in the appendix.}
%	\item \newlytyped{Landscape and algorithms, only discrete-time. Talk about all three algorithms, including Gradient.}
%	\item \newlytyped{Theory: emphasize the challenging, and give a proof sketch on ``implicit regularization''}
%	\item Future directions: prima-dual for LQ-game, using this implicit regularization idea; minimization of Hinf norm directly, though non-smooth.
%\end{itemize} 
%}
 
\section{Preliminaries}\label{sec:formulation0}
We  first provide some preliminary results on $\cH_2$ linear control with $\cH_\infty$ robustness guarantees. Throughout this section, and the following sections in the main text, we will focus on systems in discrete time.  
Counterparts of these results for continuous-time systems are included in Appendix \S\ref{sec:aux_cont_res}.

\subsection{Motivating Example: LEQG}\label{sec:mot_example}

We start with an example of \emph{risk-sensitive} control, the \emph{infinite-horizon state-feedback linear exponential quadratic Gaussian}  problem\footnote{Unless otherwise noted, we will just refer to this problem as LEQG hereafter.}   \citep{jacobson1973optimal},  which  is motivating in that:  
%In this section, we introduce an  example of optimal control synthesis  with robustness guarantees that is motivating, in the sense that: 
i) it  
%To better motivate the consideration of robustness, we start with a motivating example 
%on the discrete-time linear exponential quadratic regulator problem \cite{jacobson1973optimal}, which 
%that 
is closely related to the   well-known linear optimal control  problems, e.g., LQR and state-feedback LQG; ii) it illustrates the idea of mixed control   design, especially introducing the $\cH_\infty$-norm constraint, though implicit, that  guarantees robustness. The latter manifests the challenge in  the convergence analysis  of  PO methods for this problem.  
% challenging. 

%\issue{XXXXXX 2019.09.26}

%\subsubsection{Formulation}  
%Specifically, 
%one standard approach to  account for the uncertainty in optimal control synthesis is by designing \emph{risk-sensitive} control.  
%As a baseline, the so-termed linear exponential quadratic regulator  problem \cite{jacobson1973optimal} has been advocated in the linear-quadratic setting. 
%\issue{2019.09.28.} 
Specifically, at time $t\geq 0$, the agent takes an action $u_t\in\RR^d$ at  state $x_t\in\RR^m$, which leads the system to a new state $x_{t+1}$ by a linear dynamical  system 
\$
x_{t+1}=   A x_t + B u_t+w_t, \qquad x_0 \sim \cN(\bm{0},X_0),\quad w_t\sim\cN(\bm{0},W),
\$
where $A$ and $B$ are matrices of proper dimensions, $x_0\in\RR^{m}$ and $w_t\in\RR^{m},\forall t\geq 0$ are independent zero-mean Gaussian random variables with positive-definite covariance matrices $X_0$ and $W$, respectively. 
%Moreover, t
The one-stage cost of applying  control $u$ at state $x$  is  given  by $
c(x, u) =  x^\top Q x +u^\top R u$, 
where $Q$ and $R$ are positive-definite matrices. 
%Now, consider the following cost $\cJ$ for  risk-sensitive control design:
Then, the long-term cost to minimize is 
\#\label{equ:def_obj} 
\cJ:=\limsup_{T\to\infty}~~\frac{1}{T}\frac{2}{\beta}\log\EE\exp\bigg[\frac{\beta}{2} \sum_{t=0}^{T-1}c(x_t, u_t) \bigg],
\#
where $\beta$ is the parameter that describes the intensity of risk-sensitivity, and the expectation  is taken over the randomness of both $x_0$ and $w_t$ for all $t\geq 0$. The intuition behind the objective \eqref{equ:def_obj} is that by Taylor series  expansion around $\beta=0$, 
\$
\cJ\approx \limsup_{T\to\infty}~~\frac{1}{T}\bigg\{\EE\bigg[\sum_{t=0}^{T-1}c(x_t, u_t) \bigg]+\frac{\beta}{4}\Var\bigg[\sum_{t=0}^{T-1}c(x_t, u_t) \bigg]\bigg\}+O(\beta^2). 
\$ 
Hence, if $\beta>0$, the control is \emph{risk-averse} since minimization also places positive weight on the variance, in addition to the expectation, of the cost; in contrast, if $\beta<0$, the control is referred to as \emph{risk-seeking}, which encourages the variance to be large. As $\beta\to 0$, the objective  \eqref{equ:def_obj} reduces to the \emph{risk-neutral}  objective of  LQR/state-feedback LQG. Usually LEQG problems consider the case of $\beta>0$. In this sense, LEQG can be viewed as a generalization of LQR/state-feedback LQG  problems.

The goal of LEQG	 is to find the optimal control policy $\mu_t:(\RR^m\times \RR^d)^{t}\times\RR^m\to \RR^d$, which in general is a mapping  from the history of state-action pairs till time $t$ and current state $x_t$, to the action $u_t$ in $\RR^d$, that  minimizes the cost in \eqref{equ:def_obj}. 
By assuming that such an  optimal policy exists, the $\limsup$ in \eqref{equ:def_obj} can be  replaced by $\lim$. Moreover, we can show, see a formal statement    in Lemma \ref{lemma:optimal} in \S\ref{sec:aux_res}, that the optimal control  has a desired property of being {\it memoryless} and \emph{stationary}, i.e.,  \emph{linear time-invariant} (LTI), and current \emph{state-feedback}, i.e., $
\mu_t(x_{0:t},u_{0:t-1})=\mu(x_t)=-Kx_t,
$ 
for some $K\in\RR^{d\times m}$. 
Hence, it suffices to optimize over the control gain $K$, without loss of optimality, i.e.,     
\#\label{eq:obj_K}
\min_{K}\quad \cJ(K):=\lim_{T\to\infty}~~\frac{1}{T}\frac{2}{\beta}\log\EE\exp\bigg[\frac{\beta}{2} \sum_{t=0}^{T-1}c(x_t, -Kx_t) \bigg]. 
\#

\subsubsection{Cost Closed-Form} 
%\vspace{6pt}

To solve  \eqref{eq:obj_K} with PO methods, it is necessary to establish the closed-form of the objective with respect to  $K$. 
To this end, we  introduce the following algebraic  Riccati equation 
\#\label{equ:def_mod_Bellman_ori} 
P_K=Q+K^\top RK+(A-BK)^\top\big[P_K- P_K(-1/\beta\cdot  W^{-1}+ P_K)^{-1}P_K\big](A-BK), 
\#
for given control gain $K$. 
If $\beta\rightarrow 0$, \eqref{equ:def_mod_Bellman_ori} reduces to the Lyapunov equation of policy evaluation for given $K$ in LQR problems. 
For notational simplicity, we also define $\tP_K$ as
\#
%P_K&=Q+K^\top RK+(A-BK)^\top \tP_K(A-BK),\qquad\quad 
\tP_K =P_K+\beta P_K(W^{-1}-\beta P_K)^{-1}P_K.\label{equ:def_tPK}
\#
%\emph{modified Bellman equation}\footnote{It is \emph{modified} since if $\beta=0$,  \eqref{equ:def_tPK} reduces to the Bellman equation of LQR for given $K$.} corresponding to $K$:
%\#
%P_K&=Q+K^\top RK+(A-BK)^\top \tP_K(A-BK),\qquad\quad 
%\tP_K =P_K+\beta P_K(W^{-1}-\beta P_K)^{-1}P_K,\label{equ:def_tPK}
%\#	
%which can equivalently written as a Riccati equation
%\#\label{equ:def_mod_Bellman_ori} 
%P_K=Q+K^\top RK+(A-BK)^\top\big[P_K- P_K(-1/\beta\cdot  W^{-1}+ P_K)^{-1}P_K\big](A-BK). 
%\#
Then, the objective $\cJ(K)$ can be expressed by the solution to \eqref{equ:def_mod_Bellman_ori}, $P_K$, as follows.

\begin{lemma}\label{lemma:LEQR_obj_form_for_K}
	For any stabilizing LTI state-feedback controller  $u_t=-Kx_t$, such that  the Riccati equation  \eqref{equ:def_mod_Bellman_ori}  admits a solution $P_K\geq 0$ that: i) is stabilizing, i.e., $\rho\big((A-BK)^\top(I-\beta P_KW)^{-1}\big)<1$, and  ii) satisfies $W^{-1}-\beta P_K>0$, 
%	that induces a finite objective value, suppose the \emph{modified Bellman equations} defined as
%	\#
%P_K&=Q+K^\top RK+(A-BK)^\top \tP_K(A-BK)\label{equ:def_PK}\\
%\tP_K&=P_K+\beta P_K(W^{-1}-\beta P_K)^{-1}P_K,\label{equ:def_tPK}
%\#	
%admits a stabilizing fixed-point solution such that: i) $P_K\geq 0$; ii) $W^{-1}-\beta P_K>0$; iii) $(A-BK)^\top(I-\beta P_KW)^{-1}$ is stable. 
%Then,  	 
$\cJ(K)$ has the form of 
	\#\label{equ:obj_logdet_form}
\cJ(K)=-\frac{1}{\beta}\log\det (I-\beta P_KW).
\#
%where $P_K$ is the solution to the modified Bellman equation that satisfies the second condition in Proposition \ref{prop:discrete_equiva_set_cK}. 
\end{lemma}

%The proof of Lemma \ref{lemma:LEQR_obj_form_for_K} is deferred to \S\ref{sec:aux_res}.  
% for the first time. 
Note that when $\beta\to 0$, the objective \eqref{equ:obj_logdet_form} reduces to $\tr (P_KW)$, the cost function  for LQG  problems. 

\begin{remark}[New Results on LEQG]
	To the best of our knowledge,  our results on that the optimal controller is LTI  state-feedback in Lemma \ref{lemma:optimal}, and on the form of the objective $\cJ(K)$ in 	Lemma \ref{lemma:LEQR_obj_form_for_K}, though expected, have not been rigorously established for  LEQG  problems in the literature.   For completeness,  we  present a self-contained proof in \S\ref{sec:aux_res}.  
	Interestingly, the former argument has been hypothesized in Section $3$ of \cite{glover1988state}; while the  form of $\cJ(K)$ in \eqref{equ:obj_logdet_form}  connects to the performance criterion   for more general optimal control problems  with robustness guarantees, as to be shown shortly.
\end{remark}

%{\color{red} BinComment: I am a little bit worried that Glover and other experts will think the above ``new" result is already known. So I change the tone a little bit here as follows.

%For completeness,  we will present a self-contained proof for the fact the optimal controller is LTI  state-feedback in Lemma \ref{lemma:optimal}, and for the form of the objective $\cJ(K)$ in 	Lemma \ref{lemma:LEQR_obj_form_for_K}. Although these results are as expected \cite[Section $3$]{glover1988state}, we are not able to find their proofs in existing literatures.  It is also worth mentioning that the  form of $\cJ(K)$ in \eqref{equ:obj_logdet_form}  connects to the performance criterion   for more general optimal control problems  with robustness guarantees, as to be shown shortly.
%}

%\subsubsection
%\vspace{8pt}
%\noindent\textbf
\subsubsection{Implicit Constraint on $\cH_\infty$-Norm} 
%\vspace{6pt}

Seemingly, 
\eqref{eq:obj_K}  is an unconstrained optimization over $K$. However, as identified by  \cite{glover1988state}, there is an implicit constraint set for this problem, which corresponds to the lower-level set of the $\cH_{\infty}$-norm of the closed-loop transfer function under the linear stabilizing controller $u=-Kx$. We reiterate the result as follows.

\begin{lemma}[\cite{glover1988state}]\label{lemma:LEQR_feasible_set}
	Consider the LEQG problem in  \eqref{eq:obj_K} that finds the optimal stationary state-feedback control gain $K$, and a closed-loop transfer function from the noise $\{w_t\}$ to the output, $\cT(K)$, as
	\$
\renewcommand\arraystretch{1.3}
\cT(K):=\mleft[
\begin{array}{c|c}
  A-BK & W^{1/2} \\
  \hline
  (Q+K^\top R K)^{1/2}& 0
\end{array}
\mright]. 
\$
Then, the feasible set of  $\cJ(K)$ is the intersection of the set of linear stabilizing feedback controllers and the $1/\sqrt{\beta}$-lower-level set of the $\cH_{\infty}$-norm of $\cT(K)$, i.e., $\big\{K\biggiven \rho(A-BK)<1, \,\,\mbox{and}\,\,\|\cT(K)\|_{\infty}<1/\sqrt{\beta}\big\}$. 
\end{lemma}
\begin{proof}
	The result follows  by applying the  results in Section $3$ in \cite{glover1988state},  writing out the transfer function, and 
	replacing the $\theta$ therein by the $-\beta$ here. 
\end{proof}

We note that  the feasible set for LEQG  in Lemma \ref{lemma:LEQR_feasible_set} may not necessarily be bounded. This   feasible set, though quite concise to characterize, is hard to enforce directly onto the control gain $K$, since it is a frequency-domain characterization using the $\cH_\infty$-norm. 
To develop PO algorithms for  finding $K$, the time-domain characterization in Lemma \ref{lemma:LEQR_obj_form_for_K} is more useful. Interestingly, as we will show shortly, 
the conditions that lead to the form of $\cJ(K)$ in Lemma \ref{lemma:LEQR_obj_form_for_K} are indeed equivalent to the feasible set given by 
$\cH_\infty$-norm constraint in Lemma \ref{lemma:LEQR_feasible_set};   see Remark \ref{remark:necess_lemma_LEQR_obj}. 

In fact,  this reformulation of LEQG as a constrained optimization problem,  belongs to a general class problems, \emph{mixed $\cH_{2}/\cH_{\infty}$ control design} with state-feedback.

%\section{Bigger Picture: Mixed $\cH_2/\cH_{\infty}$ Control Synthesis}\label{sec:formulation}

\subsection{Bigger Picture: Mixed $\cH_2/\cH_{\infty}$ Control Synthesis}\label{sec:formulation}

%The LEQR example  introduces a bigger picture for optimal control with robustness guarantees: mixed $\cH_{2}/\cH_{\infty}$ control design. 
%, which  includes  risk-sensitive control  as an example.  
%For notational simplicity, we focus on  discrete-time settings in the main text, and defer the  continuous-time results, to Appendix  \S\ref{sec:aux_cont_res}.    
Consider the following discrete-time linear dynamical system with a single input-output channel 
\#
x_{t+1}=Ax_t+Bu_t+Dw_t,\quad z_t=Cx_t+E u_t,\label{equ:def_mixed_discret}
\#
where $x_t\in\RR^{m},u_t\in\RR^d$ denote the states and controls, respectively,   $w_t\in\RR^n$ is the disturbance, $z_t\in\RR^l$ is the controlled output, and $A,B,C,D,E$ are  matrices of proper dimensions. 
Consider the \emph{admissible} control policy $\mu_t$ to be a mapping from the history of state-action pairs till time $t$ and the current state $x_t$ to action $u_t$. It has been shown in \cite{kaminer1993mixed} that, \emph{LTI} state-feedback controller (without memory)  suffices to achieve the optimal performance of mixed $\cH_2/\cH_\infty$ design under this \emph{state-feedback} information structure\footnote{For discrete-time settings, if both the (exogenous) disturbance $w_t$ and the state $x_t$ are available, i.e., under the \emph{full-information} feedback  case, LTI controllers may not be optimal \citep{kaminer1993mixed}. Interestingly, for continuous-time settings, LTI controllers are indeed optimal \citep{khargonekar1991mixed}.  }. As a consequence, it suffices to consider only stationary, current state-feedback  controller parametrized as $u_t=-Kx_t$. 

\begin{remark}
%\remark
[Justification of LTI Control for LEQG]\label{remark:LTI_control}
	As to be shown shortly, LEQG is a special case of mixed $\cH_2/\cH_\infty$  design. Hence, the result we derived in Lemma \ref{lemma:optimal}, i.e., the optimal controller of LEQG is indeed LTI, is consistent with this earlier result on mixed  design from \cite{khargonekar1991mixed,kaminer1993mixed}. 
%\vspace{6pt}
\end{remark}

In accordance with this parametrization, the  transfer function from the disturbance $w_t$   to the output $z_t$  can be represented as 
\#\label{equ:mixed_design_transfer}
\renewcommand\arraystretch{1.3}
\mleft[
\begin{array}{c|c}
  A-BK & D \\
  \hline
  C-EK& 0
\end{array}
\mright]. 
\#   
In common with  \cite{glover1988state,khargonekar1991mixed,bacsar1995h}, we make the following  assumption on the matrices $A,B,C,D$ and $E$. 

\begin{assumption}\label{assum:coeff_matrices}
	The matrices $A,B,C,D,E$ in \eqref{equ:mixed_design_transfer} satisfy  
	 $E^\top [C ~~ E]=[0~~ R]$ for some  $R>0$. 
%	\begin{itemize}
%		\item $(A,B)$ and $(A,D)$ are both  stabilizable, $(A,C)$ is detectable;
%		\item $E^\top [C \quad E]=[\bm{0}\quad R]$ for some positive definite matrix $R>0$. 
%	\end{itemize}
\end{assumption}

Assumption \ref{assum:coeff_matrices} is fairly standard, which 
%. The second condition, especially, 
%which  
clarifies the exposition substantially by normalising the control weighting and eliminating cross-weightings between control signal and state \citep{bacsar1995h}. 
 Hence, the transfer  function in  \eqref{equ:mixed_design_transfer}  has the  equivalent form\footnote{Strictly speaking, the transfer functions for \eqref{equ:mixed_design_transfer} and \eqref{equ:mixed_design_transfer2} are equivalent in the sense that  the values of $\cT^{\sim}(K) \cT(K)$ are the same for all the points on the unit circle.} of
% \footnote{By a slight abuse of notation, we use   $\cT(K)$ again as the transfer function here, in contrast to \eqref{equ:dynamic_sys} for the discrete-time setting.} 
\#\label{equ:mixed_design_transfer2}
\renewcommand\arraystretch{1.3}
\cT(K):=\mleft[
\begin{array}{c|c}
  A-BK & D \\
  \hline
  (C^\top C+K^\top R K)^{1/2}& 0
\end{array}
\mright]. 
\#
%Note that the transfer function $\cT(K)$ can also be written as  $\cT(K)(z)=(C^\top C+K^\top RK)^{1/2}[zI-(A-BK)]^{-1}D$.  
Hence, robustness of the designed controller  can be guaranteed by the  constraint on the $\cH_\infty$-norm, i.e., $\|\cT(K)\|_{\infty}<\gamma$ for some $\gamma>0$. The intuition behind the constraint, which follows from small gain theorem \citep{zames1966input}, is that the constraint on $\|\cT(K)\|_{\infty}$  implies that the closed-loop system is \emph{robustly stable} in that any stable transfer function $\Delta$ satisfying $\|\Delta\|_{\ell_2\to\ell_2}<1/\gamma$ may be connected from $z_t$ back to $w_t$ without destablizing the system. For more background on $\cH_\infty$ control, see \cite{bacsar1995h,zhou1996robust}. For notational convenience, we define the feasible set of mixed $\cH_{2}/\cH_{\infty}$ control design  as
\#\label{equ:define_cK} 
\cK:=\big\{K\biggiven \rho(A-BK)<1, \,\,\mbox{and}\,\, \|\cT(K)\|_{\infty}<{\gamma}\big\}.
\#
We note that the set $\cK$ may be unbounded.

In addition to the constraint, the objective of mixed $\cH_{2}/\cH_{\infty}$  design  is usually an upper bound of the $\cH_2$ norm of the closed-loop system. By a slight abuse of notation, let $\cJ(K)$ be the cost function of mixed design. Then the common forms of $\cJ(K)$ include \cite{mustafa1989relations,mustafa1991lqg} 
\#
  \cJ(K)&=\tr(P_KDD^\top),\label{equ:form_J1}\\
%  \quad\text{or}\quad 
  \cJ(K)&=-\gamma^2\log\det (I-\gamma^{-2}P_KDD^\top),\label{equ:form_J2}\\
%  \quad \text{or}\quad 
\cJ(K)&=\tr\big[D^\top P_K(I-\gamma^{-2}DD^\top P_K)^{-1}D\big],\label{equ:form_J3}
\#
where $P_K$ is the solution  to the following Riccati equation 
%\small
\#
(A-BK)^\top \tP_K (A-BK)+C^\top C+K^\top RK-P_K=0,\label{equ:discret_riccati}
\#
\normalsize
with $\tP_K$ defined as 
\#\label{equ:def_tP_K}
\tP_K:=P_K+P_KD(\gamma^2I-D^\top P_K D)^{-1}D^\top P_K. 
\#
%for notational convenience. 

\begin{remark}[LEQG as a Special Case of Mixed-Design]\label{remark:special_case}
	By  Lemma \ref{lemma:LEQR_feasible_set}, replacing $\beta$, $W$, and $Q$ in LEQG by $\gamma^{-2}$,  $DD^\top$ and $C^\top C$, respectively, yields the formulation of  mixed $\cH_{2}/\cH_{\infty}$ design. 
	In particular, the closed-form cost of LEQG that we derived for the first time, see Lemma \ref{lemma:LEQR_obj_form_for_K}, 
	is identical to the cost in \eqref{equ:form_J2}; and the implicit constraint of LEQG in Lemma \ref{lemma:LEQR_feasible_set} is exactly the $\cH_\infty$-norm constraint in \eqref{equ:define_cK}. 
	Thus,  LEQG is a mixed-design problem with $D=W^{1/2}$ and $\cJ(K)$ being \eqref{equ:form_J2}. Note that  $DD^\top =W>0$ for LEQG.
\end{remark}

%By   Remark \ref{remark:special_case},  
%we note that \eqref{equ:discret_riccati} is the same as the Riccati equation \eqref{equ:def_mod_Bellman_ori}, with certain change of variables. 

%which was also referred to as {modified Bellman equation}  for LEQR. Thus, to be consistent, we also refer to   \eqref{equ:discret_riccati} as the \emph{modified Bellman equation} for mixed $\cH_2/\cH_\infty$ design. 

All three objectives in \eqref{equ:form_J1}-\eqref{equ:form_J3} are upper bounds of the $\cH_2$-norm \citep{mustafa1989relations,mustafa1991lqg}. In particular, cost \eqref{equ:form_J1} has been adopted in \cite{bernstein1989lqg,haddad1991mixed}, which resembles the standard $\cH_2$ control/LQG control objective, but with $P_K$ satisfying a Riccati equation instead of a Lyapunov equation.  
Cost \eqref{equ:form_J2} is closely related to maximum entropy $\cH_\infty$-control, see the detailed relationship between the two in \cite{mustafa1990minimum}. In addition, cost \eqref{equ:form_J3} can also be connected to the cost of LQG using a different Riccati equation \cite[Remark $2.7$]{mustafa1991lqg}. 
As $\gamma\to \infty$, the costs in all \eqref{equ:form_J1}-\eqref{equ:form_J3}  reduce to the cost for LQG, i.e., $\cH_2$ control design problems. 

In sum, the mixed $\cH_2/\cH_\infty$ control design can be formulated as
\#\label{equ:def_mixed_formulation}
\min_K \quad \cJ(K),\qquad s.t.\quad K\in\cK,
\#
with $\cJ(K)$ and $\cK$ defined in \eqref{equ:form_J1}-\eqref{equ:form_J3} and \eqref{equ:define_cK}, respectively.

%\issue{2019.08.21. }
%
%\issue{In the introduction of the ``bigger picture'' on mixed design, for both continuous \cite{khargonekar1991mixed} and discrete \cite{kaminer1993mixed} time, if we consider the optimal control design in the  ``state-feedback'' case, then the ``dynamic full information controller'' has the same optimal performance as just ``stationary state-feedback controllers''. If we also have access to the noise, then the ``dynamic full information controller'' performance is still achievable by ``stationary state-feedback controllers'' for ``continuous-time'', but, interestingly, not for the ``discrete-time'' \cite{kaminer1993mixed}. Anyways, it suffices to only consider ``stationary state-feedback controllers'' here, since we have no access to the noise/disturbance/exogenous input. Hence, we may want to connect to this discrete-time result in \cite{kaminer1993mixed}, to justify that the results for LEQR is ``expected'', since we only have state, no noise observations (write this as a remark).} 
%\\

%For notational convenience, we define the feasible set of mixed $\cH_{2}/\cH_{\infty}$ control design  as
%\#\label{equ:define_cK}
%\cK:=\big\{K\biggiven \|\cT(K)\|_{\infty}<{\gamma}\big\}. 
%\# 

%\issue{WE INTRODUCE THE NONCONVEXITY of $\cK$ HERE TOO} 

%XXXXXXXXX

%\issue{2019.08.23}

\subsection{Bounded Real Lemma}

Though the constraint \eqref{equ:define_cK} is concise, it is hard to enforce over $K$ in policy optimization, since the constraint is defined  in the frequency domain. 
Interestingly, by using a significant  result in robust control theory, i.e., \emph{Bounded Real Lemma} \citep[Chapter $1$]{bacsar1995h}, \citep{zhou1996robust,rantzer1996kalman}, constraint \eqref{equ:define_cK} can be related to the solution of a Riccati equation and a Riccati inequality. 
We formally introduce the result  
as follows, whose proof is deferred to \S\ref{sec:proof_lemma_bounded_real_lemma}.

\begin{lemma}[Discrete-Time Bounded Real Lemma]\label{lemma:discrete_bounded_real_lemma}
Consider a discrete-time  transfer function $\cT(K)$ defined in \eqref{equ:mixed_design_transfer2}, suppose $K$ is stabilizing, i.e., $\rho(A-BK)<1$,  then the  following conditions are equivalent: 
%consider a discrete-time dynamical system with transfer function $\cT(K)$ defined as
%\#\label{equ:dynamic_sys}
%\renewcommand\arraystretch{1.3}
%\cT(K):=\mleft[
%\begin{array}{c|c}
%  A-BK & W^{1/2} \\
%  \hline
%  (Q+K^\top R K)^{1/2}& \bm{0} 
%\end{array}
%\mright]. 
%\#
\begin{itemize}
		\item $\|\cT(K)\|_{\infty}<{\gamma}$, which, due to $\rho(A-BK)<1$, further implies  $K\in\cK$ with $\cK$ defined in \eqref{equ:define_cK}. 
%		The control gain $K$ lies in $\cK$ defined in \eqref{equ:define_cK}, i.e., $\rho(A-BK)<1$ and $\|\cT(K)\|_{\infty}<{\gamma}$; 
		\item The Riccati equation \eqref{equ:discret_riccati} 
%		following Riccati equation, also referred to as \emph{modified Bellman equation}\#\label{equ:bnded_real_lemma_riccati}
%  P_K=\tilde A_K^\top P_K\tilde A_K+\tilde A_K^\top P_KD(\gamma^2I-D^\top P_K D)^{-1}D^\top P_K\tilde A_K+C^\top C+K^\top RK\\
%    P_K=(A-BK)^\top \tP_K (A-BK)+C^\top C+K^\top RK
%  \#
		 admits  a unique stabilizing  solution $P_K\geq 0$ such that: i) $I-\gamma^{-2}D^\top P_K D>0$; ii) $(I-\gamma^{-2}  P_KDD^\top)^{-\top}(A-BK)$ is stable; 
%		  where $\tP_K:=P_K+P_KD(\gamma^2I-D^\top P_K D)^{-1}D^\top P_K$; 
		\item There exists some $P> 0$, such that
		\#\label{equ:discrete_equiva_set_cK_cond}
I-\gamma^{-2}D^\top P D>0,\quad (A-BK)^\top \tilde{P}(A-BK)-P+C^\top C+K^\top R K<0,
\#
where $\tP:=P+PD(\gamma^2I-D^\top P D)^{-1}D^\top P$. 
%		\item The $\cH_\infty$-norm of $\cT(K)$ satisfies that $\|\cT(K)\|_{\infty}<1/\sqrt{\gamma}$. 
	\end{itemize}
\end{lemma}

%\begin{lemma}[Continuous-Time Bounded Real Lemma]\label{lemma:cont_bounded_real_lemma}
%Consider a continuous-time dynamical system with transfer function $\cT(K)$, then the  following conditions are equivalent: 
%%consider a discrete-time dynamical system with transfer function $\cT(K)$ defined as
%%\#\label{equ:dynamic_sys}
%%\renewcommand\arraystretch{1.3}
%%\cT(K):=\mleft[
%%\begin{array}{c|c}
%%  A-BK & W^{1/2} \\
%%  \hline
%%  (Q+K^\top R K)^{1/2}& \bm{0} 
%%\end{array}
%%\mright]. 
%%\#
%XXXXXX
%\begin{itemize}
%		\item The control gain $K$ lies in $\cK$ defined in \eqref{equ:define_cK}, i.e., $\|\cT(K)\|_{\infty}<1/\sqrt{\gamma}$; 
%		\item The modified Bellman equations  \eqref{equ:def_PK}-\eqref{equ:def_tPK} admit  a unique stabilizing  solution $P_K> 0$ such that: i) $W^{-1}-\gamma P_K>0$; ii) $(A-BK)^\top(I-\gamma P_KW)^{-1}$ is stable; 
%		\item There exists some $P>0$, such that
%		\#\label{equ:discrete_equiva_set_cK_cond}
%W^{-1}-\gamma P>0,\quad (A-BK)^\top \tilde{P}(A-BK)-P+Q+K^\top R K<0,
%\#
%where $\tilde{P}:=P+\gamma P(W^{-1}-\gamma P)^{-1}P$. 
%%		\item The $\cH_\infty$-norm of $\cT(K)$ satisfies that $\|\cT(K)\|_{\infty}<1/\sqrt{\gamma}$. 
%	\end{itemize}
%\end{lemma}
%%}

The three equivalent conditions in Lemma    \ref{lemma:discrete_bounded_real_lemma}  
 will be frequently used in the ensuing analysis. 
%By Lemma  \ref{lemma:discrete_bounded_real_lemma}, we have the immediate corollary for discrete-time LEQR problems introduced in \S\ref{sec:mot_example}.  
%\issue{WE CAN JUST SAY THE FOLLOWING AS A REMARK.}
%
%\issue{XXXXX}
Note that the unique stabilizing  solution to \eqref{equ:discret_riccati} for any $K\in\cK$, is also \emph{minimal}, if the pair   $(A-BK,D)$ is stabilizable, see 
%\cite[Corollary $13.13$, page $339$] {zhou1996robust} and 
\cite[Theorem $3.1$]{ran1988existence}.   This holds since $K\in\cK$ is indeed stabilizing. 
%if $DD^\top >0$, which is indeed the case for   LEQG problems where $D=W^{1/2}>0$. In this case, 
Thus, the optimal control that minimizes \eqref{equ:form_J1}-\eqref{equ:form_J3}, which are all monotonically increasing  with respect to $P_K$, only involves the stabilizing solution $P_K$. Hence, it suffices to consider only stabilizing solution $P_K$ of the Riccati  equation \eqref{equ:discret_riccati} for LEQG.

\begin{remark}[Necessity of Lemma \ref{lemma:LEQR_obj_form_for_K}]\label{remark:necess_lemma_LEQR_obj}
	By Lemma \ref{lemma:discrete_bounded_real_lemma} and Remark \ref{remark:special_case},  the conditions in Lemma \ref{lemma:LEQR_obj_form_for_K}   are   equivalent to the $\cH_\infty$-norm constraint in \eqref{equ:define_cK} for LEQG. This implies that these conditions are not only \emph{sufficient} for the form of $\cJ(K)$ in  Lemma \ref{lemma:LEQR_obj_form_for_K} to hold, but also \emph{necessary}. In other words, any feasible $K\in\cK$ should lead to the form of $\cJ(K)$ in \eqref{equ:obj_logdet_form}. 
\end{remark}

Next, we 
%study the optimization landscape, and 
develop policy optimization algorithms for solving the mixed $\cH_2/\cH_\infty$ control problem 
%optimal control problem with robustness guarantee 
in  \eqref{equ:def_mixed_formulation}.

\section{Landscape and 
%Policy Optimization 
Algorithms}\label{sec:algorithm}

In this section, we investigate  the optimization landscape of mixed $\cH_2/\cH_\infty$ control design, and develop policy optimization algorithms with convergence  guarantees.   In particular, we study both discrete- and continuous-time  settings  focusing on  two representative example costs $\cJ(K)$ from \eqref{equ:form_J2} and \eqref{equ:form_J1}, respectively.\footnote{Although only two example settings are studied in detail, the techniques developed can also be applied to other combinations of settings, e.g., cost \eqref{equ:form_J1} in discrete-time settings.} The first  combination of settings also by chance solves the discrete-time LEQG problems  introduced in \S\ref{sec:mot_example}. The second combination for continuous-time settings is discussed  in \S\ref{sec:aux_cont_res}. 

\subsection{Optimization Landscape}

We start by showing  that, regardless of the cost  $\cJ(K)$, the mixed-design problem in \eqref{equ:def_mixed_formulation} is a \emph{nonconvex} optimization problem.

 \begin{lemma}[Nonconvexity of Discrete-Time Mixed $\cH_{2}/\cH_{\infty}$ Design]\label{lemma:nonconvex_Hinf_norm_set}
 	The discrete-time mixed $\cH_{2}/\cH_{\infty}$ design  problem \eqref{equ:def_mixed_formulation}   is nonconvex. 
% 	, for both continuous- and discrete-time settings. 
 \end{lemma}

The proof of  Lemma \ref{lemma:nonconvex_Hinf_norm_set} is deferred to \S\ref{sec:proof_lemma:nonconvex}. In particular, we show by an easily-constructed example that the convex combination of two control gains $K$ and $K'$ in $\cK$ may no longer lie in $\cK$. As a result, this nonconvexity poses challenges in solving \eqref{equ:def_mixed_formulation} using standard policy gradient-based approaches. Note that similar    nonconvexity of the constraint set also exists in LQR problems \citep{fazel2018global,bu2019LQR}, and has been recognized as one of the main challenges to address. 
Still, the landscape of LQR  has some desired property of being \emph{coercive} \citep[Lemma $3.7$]{bu2019LQR}, which played a significant role in the analysis of PO methods for LQR.  However, we establish in the following lemma that such a coercivity  does not hold for mixed design problems.

\begin{lemma}[No Coercivity of Discrete-Time  Mixed $\cH_{2}/\cH_{\infty}$ Design]\label{lemma:mixed_design_no_coercivity}
	The cost functions \eqref{equ:form_J1}-\eqref{equ:form_J3} for discrete-time mixed $\cH_{2}/\cH_{\infty}$ design are not coercive. Particularly, as $K\to \partial \cK$, where $\partial \cK$ is the boundary of the constraint set $\cK$, the cost $\cJ(K)$ does not necessarily  approach  infinity. 
\end{lemma} 

The proof of Lemma \ref{lemma:mixed_design_no_coercivity} is provided in \S\ref{sec:proof_lemma_mixed_design_no_coercivity}. The key of the argument is that for given $K\in\cK$, the \emph{policy evaluation} equation for mixed design problems is a Riccati equation, see \eqref{equ:discret_riccati} (a quadratic equation of $P_K$ in $1$-dimensional case); while for LQR problems, the policy evaluation equation is a Lyapunov equation, which is essentially linear. Hence, some additional  condition on $K$ is required for the existence of the solution, which can be \emph{restricter} than the conditions on $K$ and $P_K$ that makes the cost $\cJ(K)$ finite. 
%from Bounded Real Lemma, namely,  $P_K\geq 0$, $I-\gamma^{-2}D^\top P_K D>0$, and $(I-\gamma^{-2}  P_KDD^\top)^{-\top}(A-BK)$ is stable. 
In this case, the existence condition of the solution characterizes the boundary of $\cK$, which leads to a well-defined $P_K$, and thus a 
finite value of the cost $\cJ(K)$, even when $K$ approaches the boundary $\partial \cK$. 

The lack of coercivity turns out to be the greatest challenge when  analyzing the stability/feasibility of PO methods for mixed control design,  in contrast to LQR problems.  Detailed discussion on this is provided in \S\ref{subsec:implicit_reg}.  The illustration in Figure \ref{fig:illust_hardness} in \S\ref{sec:intro}  of the landscape of mixed design problems was actually based on Lemmas \ref{lemma:nonconvex_Hinf_norm_set} and  \ref{lemma:mixed_design_no_coercivity}. 
%, which is independent of the choice of $\cJ(K)$.  Now we proceed to  study the landscape property that depends  on $\cJ(K)$. 
%Let $\cJ(K)$ have the form of  \eqref{equ:form_J2}, which is also the cost of risk-sensitive control problem, see Lemma \ref{lemma:LEQR_obj_form_for_K}.   
We then show   the differentiability of $\cJ(K)$ at each $K$ 
%for all three objectives \eqref{equ:form_J1}-\eqref{equ:form_J3}   
within the feasible set $\cK$, and provide the closed-form of the policy gradient. 
Here we focus on the most complicated objective defined in \eqref{equ:form_J2} among the three in \eqref{equ:form_J1}-\eqref{equ:form_J3}, due to its direct connection to the risk-sensitive control problem; see Remark \ref{remark:special_case}. 
We note that the proof can be used directly to establish similar results for the  other two objectives, namely,  \eqref{equ:form_J1} and \eqref{equ:form_J3}, too.

\begin{lemma}\label{lemma:differentiability_policy_grad}
The cost $\cJ(K)$ defined in \eqref{equ:form_J2} is differentiable 
%, and thus continuous, 
in $K$ for any $K\in\cK$, and the policy gradient has the following form:
\$ 
\nabla \cJ(K)=2\big[(R+B^\top \tP_K B)K-B^\top \tP_K A\big]\Delta_K,
\$	
where $\Delta_K\in\RR^{m\times m}$ is a matrix given by
\small 
\#\label{equ:def_Delta}
\Delta_K&:=\sum_{t=0}^\infty \big[(I-\gamma^{-2} P_KDD^\top)^{-\top}(A-BK)\big]^tD(I-\gamma^{-2}D^\top P_K D)^{-1}D^\top\big[(A-BK)^\top(I-\gamma^{-2} P_KDD^\top)^{-1}\big]^t,
\#
\normalsize
and $\tP_K$ is defined in \eqref{equ:def_tP_K}. 
\end{lemma}

The proof of Lemma \ref{lemma:differentiability_policy_grad} is provided in \S\ref{sec:proof_lemma_differentiability_policy_grad}. 
Note that Lemma \ref{lemma:differentiability_policy_grad} 
%,  proved in \S\ref{sec:proof_lemma_policy_grad},  
also implies some property on the landscape of $\cJ(K)$. Specifically, if $\Delta_K>0$ is full-rank, then   $\nabla \cJ(K)=0$ admits a unique solution $K=(R+B^\top \tP_{K} B)^{-1}B^\top \tP_{K} A$, which corresponds to the unique global optimum. Otherwise, if $\Delta_K\geq 0$ is not full-rank, there can be multiple stationary points. Yet, the global optimum is still of the same form. We formally establish this in the following proposition,  which is proved in   \S\ref{proof:coro_opt_control_form}. 

%characterizes the desired  landscape  of the LEQR problems. In particular, as long as the matrix $\Delta_K$ is full-rank, the stationary-point of $\cJ(K)$ will lead to the optimal control  in \eqref{equ:def_tP}. However, unlike  the LQR setting in \cite{fazel2018global}, where the full-rankness of the  correlation of the initial states suffices to guarantee such a good landscape (cf. Corollary $4$ in    \cite{fazel2018global}), the requirement of $\Delta_K$ to be full-rank here depends on $K$. By Proposition  \ref{prop:discrete_equiva_set_cK}, this can be guaranteed if $K$ always lies in $\cK$. 
%By Lemma \ref{lemma:policy_grad_ct}, we can further obtain the following corollary that gives the formula of the optimal controller for $\cH_2/\cH_\infty$ mixed control design under certain conditions. 

\begin{proposition}\label{coro:opt_control_form_discrete}
	Suppose that the discrete-time mixed $\cH_2/\cH_\infty$ design admits a global optimal solution  $K^*\in\cK$; then,  one such  solution  has the form of $K^*=(R+B^\top \tP_{K^*} B)^{-1}B^\top \tP_{K^*} A$. Additionally, if the pair $\big((I-\gamma^{-2}  P_KDD^\top)^{-\top}(A-BK),D\big)$ is controllable at some  stationary point of $\cJ(K)$, such that $\nabla \cJ(K)=0$, then this is the unique stationary point, and corresponds to the unique  global optimizer  $K^*$.  
\end{proposition}

%Under this constraint, \issue{the stationary feedback control gain satisfying the following equation achieves the minimum value of $\cJ(K)$    (actually there is no reference I can find that explicitly gives this formula. One option is that we only say we are looking for some $K$ here. Then, by differentiability, we show that the only stationary point is the one below. Done. Thus, we put it later.). 

%Moreover, this is equivalently to require the following LMI hold, i.e., there exists a $P>0$ such that:
%\#\label{equ:LMI_cond_ct}
%\mleft[
%\begin{array}{cc}
%  (A-BK)^\top P+P(A-BK)+Q+K^\top R K  & P \\
%  P& -\frac{1}{\gamma}W^{-1} 
%\end{array}
%\mright]<0. 
%\# 
%The goal of the problem is to find the optimal $K$ such that 
%} 	

The form of the optimal control gain $K^*$ above echoes back that of the solution to \emph{finite-horizon} LEQG  \citep{jacobson1973optimal}, \citep[Chapter $7$]{whittle1990risk}. Note that for LEQG problems,  $D=W^{1/2}>0$ implies that the controllability  condition holds automatically. Thus, this $K^*$ corresponds to the \emph{unique}  global optimizer. 
We also remark that the landscape result above can also be shown for the other two objectives  \eqref{equ:form_J1} and \eqref{equ:form_J3}. In fact, the key in proving Proposition \ref{coro:opt_control_form_discrete} is to show that $P_{K^*}$ is \emph{matrix-wise} minimal in the positive semi-definite sense for all $P_K$ with  $K\in\cK$. Note that since the objectives \eqref{equ:form_J1} and \eqref{equ:form_J3} are both monotonically non-decreasing in the eigenvalues of $P_K$, one can verify that the $K^*$ is also the global optimizer. Note that $K^*$ may not be the \emph{unique} global minimizer without the controllability assumption. Finally, following the proof of Lemma \ref{lemma:differentiability_policy_grad}, one can show that the policy gradients for \eqref{equ:form_J1} and \eqref{equ:form_J3} yield an almost identical form as in Lemma \ref{lemma:differentiability_policy_grad}, except the definition of $\Delta_K$. Although the  controllability assumption has been made in the literature  \citep{mustafa1991lqg}, and is also satisfied automatically by LEQG problems,  we will show next that our PO methods can find the global optimum $K^*$ even without this assumption. 	
	
%Note that for LEQG  with $D=W^{1/2}>0$, the controllability  condition in Corollary  \ref{coro:opt_control_form_discrete} is satisfied for any $K\in\cK$. Therefore, the argument that \emph{stationary point implies global optimum} holds for LEQG  without any additional assumption.   
%Indeed, the form of the optimal control gain $K^*$ in the corollary echoes back that of the solution to \emph{finite-horizon} LEQG  \citep{jacobson1973optimal}, \citep[Chapter $7$]{whittle1990risk}.  
%Note that the controllability assumption is standard for mixed design, and has also been made in \cite{mustafa1991lqg}. We also remark that the landscape above can also be shown for the other two objectives  \eqref{equ:form_J1} and \eqref{equ:form_J3}. In fact, following the techniques in proving Lemma \ref{lemma:differentiability_policy_grad}, one can show that the policy gradient yields an almost identical form as in Lemma \ref{lemma:differentiability_policy_grad}, except the definition of $\Delta_K$. Under the controllability assumption, the new $\Delta_K$ is also invertible, leading to the optimality of the stationary point.
%
% In order to find the global optimum under the conditions of Corollary \ref{coro:opt_control_form_discrete}, it suffices to find the first-order stationary point, which can be obtained using first-order policy optimization methods.

%\issue{XXXXXXXXXXX TO BE CHANGED XXXXXXXXXXX}

\subsection{Policy Optimization Algorithms}\label{sec:PO_alg}

Consider three    policy-gradient based methods as follows. For  simplicity, we define 
\#\label{equ:def_mu_Ek}
%\mu := \sigma_{\textrm{min}}(\EE_{x_0\sim \cD } x_0 x_0^\top)\text{~~~~~and~~~~}
E_K:=(R+B^\top \tP_{K} B)K-B^\top \tP_{K} A. 
\#
We also suppress  the iteration index, and  use $K$ and $K'$ to represent the control gain before and after one-step of  the update. 
\begin{flalign}
 {\rm \textbf{Policy Gradient:}}~~~~~~\qquad\quad\qquad K'&=K-\eta \nabla\cJ(K)=K-2\eta E_{K} \Delta_K\label{eq:exact_pg}\\
 {\rm \textbf{Natural Policy Gradient:}}~~~\qquad K'&=K-\eta \nabla\cJ(K)\Delta_K^{-1}=K-2\eta E_{K} &\label{eq:exact_npg}\\
 {\rm \textbf{Gauss-Newton:}}\qquad\qquad\quad\qquad K'&=K-\eta (R+B^\top \tP_{K} B)^{-1}\nabla\cJ(K)\Delta_K^{-1}\notag\\
 \qquad\qquad\qquad\qquad\qquad\qquad\qquad\quad&=K-2\eta (R+B^\top \tP_{K} B)^{-1}E_{K}
 &\label{eq:exact_gn}
\end{flalign}
where $\eta>0$ is the stepsize. The updates are motivated by  and resemble   the policy optimization updates for LQR \citep{fazel2018global,bu2019LQR}, but with $P_K$ therein replaced by $\tP_K$. The natural PG update is related to gradient over a Riemannian manifold; while the Gauss-Newton update is one type of quasi-Newton update, see \cite{bu2019LQR} for further  justifications on the updates. In particular, with $\eta=1/2$, the Gauss-Newton update \eqref{eq:exact_gn} can be viewed as the \emph{policy iteration} update for infinite-horizon mixed $\cH_2/\cH_\infty$ design. Model-free versions of the PG  update \eqref{eq:exact_pg}  can be directly obtained, since the gradient $\nabla\cJ(K)$ can be estimated by sampled data, using for instance zeroth-order methods, as in \cite{fazel2018global,malik2019derivative}. 
A direct model-free implementation of  the natural PG update \eqref{eq:exact_npg} using zeroth-order optimization methods  requires estimating the matrix $\Delta_K$.  It is not clear yet how to estimate it  from the sampled trajectories. Instead, we propose one solution by the connections  between mixed design and  zero-sum linear quadratic games; see \S\ref{sec:discussion} for more details.  Finally,  as in LQR problems, the Gauss-Newton update \eqref{eq:exact_gn} cannot yet be estimated using zeroth-order methods  directly.

\section{Theoretical Results}\label{sec:conv_res}

In this section, we investigate the convergence  of the PO methods proposed in \S\ref{sec:algorithm}.

\subsection{Implicit Regularization}\label{subsec:implicit_reg}

The first key challenge in the convergence analysis for PO methods,  is to ensure that  the iterates remain \emph{feasible} as the algorithms proceed, hopefully without the use of \emph{projection}. 
%In fact, for \emph{online control} design, ensuring  the stability/feasibility of the control iterates has been  recognized as one of the \emph{shades of reinforcement learning} by John N. Tsitsiklis recently. 
This is especially significant in mixed design problems, as the feasibility here means \emph{robust stability}, the  violation of which  can be catastrophic in practical \emph{online} control design.     
We formally define the concept of \emph{implicit regularization} to describe this feature. 

\begin{definition}[Implicit Regularization]\label{def:implicit_reg}
	For mixed $\cH_2/\cH_\infty$ control design problem \eqref{equ:def_mixed_formulation}, suppose an iterative algorithm generates a sequence of control gains $\{K_n\}$. If $K_n\in\cK$ for all $n\geq 0$, this algorithm is called \emph{regularized}; if it is regularized without projection onto $\cK$ for any $n\geq 0$, this algorithm is called \emph{implicitly  regularized}.
\end{definition}

%We refer to this property of \emph{remaining feasible along the iterations} as the  \emph{regularization} property, and remaining so without projection as the \emph{implicit regularization} property. We note that 

\begin{remark}\label{remark:def_implicit_reg}
	The concept of \emph{(implicit)  regularization} has been adopted in many recent studies  on nonconvex optimization, including training  neural networks \citep{allen2018learning,kubo2019implicit},  phase retrieval \citep{chen2015solving,ma2017implicit}, matrix completion \citep{chen2015fast,zheng2016convergence}, and blind deconvolution \citep{li2019rapid}, referring to any scheme that biases the search direction of gradient-based algorithms. Implicit regularization  has been advocated as an important feature of  (stochastic) gradient descent methods for solving these problems, which, as the name suggests,  means that the  algorithms without regularization may behave as if  they are regularized. Note that the term  \emph{regularization} may refer to several different schemes in different problems, e.g., trimming/truncation the gradient, adding a regularization term in the objective, etc. Here we focus on the scheme of \emph{projection}, as summarized in \cite{ma2017implicit}. Also note that implicit regularization is a feature of both the \emph{problem} and the \emph{algorithm}, i.e., it holds for  certain algorithms that solve certain nonconvex problems.   
\end{remark}

One possible way for the iterates to remain feasible is to keep shrinking the stepsize, whenever the next iterate goes outside $\cK$, following for example the Armijo rule \citep{bertsekas1976goldstein}. However, as the cost $\cJ(K)$ is not necessarily smooth (see Lemma \ref{lemma:cost_diff} and its discussion later), it may not converge within a finite number of iterations  \cite[Theorem $3.2$]{hintermuller2010nonlinear}. Another option is to  project the iterate onto $\cK$. Nonetheless, it is challenging to perform projection onto the $\cH_\infty$-norm constraint set directly in the frequency domain.  

%
%For LQR problems, due to the coercivity of the cost that as $K$ approaches the boundary of the stability/feasibility region, i.e., as $\rho(A-BK)\to 1$, the cost approaches infinity, and due to  the fact that the cost is continuous w.r.t.  $K$, we have  that the lower-level set of the cost is compact \cite{makila1987computational} and is contained in the open set  $\{K\in\RR^{d\times m}\given \rho(A-BK)<1\}$.  As a consequence,  decrease of the cost  ensures that the  next iterate lies in the  lower-level set of the cost defined by the previous iterate, which further ensures its stay within the stability region. Hence, there exists a \emph{constant stepsize} such that as long as the initialization control is stabilizing, the iterates along any  path with descending cost remain stabilizing. 
%This way, implicit regularization is obtained by any descent algorithm for free. The stability proofs in \cite{fazel2018global,bu2019LQR} for LQR  are essentially built upon this idea.  
%It is worth mentioning that such an idea   had  been adopted earlier, and is known as the \emph{homotopy proof} \cite{megretski1997system} in the robust control literature.

%\issue{
For LQR problems, due to the coercivity of the cost that as $K$ approaches the boundary of the stability/feasibility region   $\{K\in\RR^{d\times m}\given \rho(A-BK)<1\}$, i.e., as $\rho(A-BK)\to 1$, the cost blows up to   infinity, and due to  the fact that the cost is continuous with respect to  $K$,   the lower-level set of the cost is compact \citep{makila1987computational} and is contained within the stability region.  As a consequence,  there is a strict separation between any lower-level set of the cost and the set $\{K\in\RR^{d\times m}\given \rho(A-BK)\ge 1\}$. Hence, as discussed in the introduction, there exists a \emph{constant stepsize} such that as long as the initialization control is stabilizing, the iterates along the  path  remain stabilizing and keep decreasing the cost. 
Such a property is \emph{algorithm-agnostic} in that it is dictated by the property of the  cost, and independent of the algorithms adopted, as long as they  follow any descent directions of the cost.  The stability proofs in \cite{fazel2018global,bu2019LQR} for LQR  are essentially built upon this idea.

In contrast, for mixed $\cH_2/\cH_\infty$ design problems,   lack of coercivity invalidates the  argument above, as the control approaching the robustness constraint boundary $\partial \cK$ may incur a \emph{finite} cost,  and the descent direction may still drive the iterates out of the feasibility region.  In addition, there may not exist a strict separation between all the lower-level sets of the cost and the complementary  set $\mathcal{K}^c$.
%Hence, the analysis in \cite{fazel2018global,bu2019LQR} does not apply here. 
This difficulty  has been  illustrated in Figure \ref{sec:intro} in the introduction, which   compares the landscapes of the two problems. Interestingly, we show in the following theorem that the natural PG and Gauss-Newton methods in \eqref{eq:exact_npg}-\eqref{eq:exact_gn} enjoy the implicit regularization feature, with certain \emph{constant} stepsize.   
We only highlight the idea of the proof here, and defer the details  to \S\ref{sec:proof_thm_stability_update}. 
%\issue{
The proof includes two main steps. First, we directly use $P_K$ to construct a Lyapunov function for $K'$ to show a non-strict Riccati inequality that guarantees $\|\cT(K')\|_\infty\le{\gamma}$. Second, we further perturb $P_K$ in a specific way to show the strict inequality $\|\cT(K')\|_\infty<{\gamma}$. The perturbation argument is inspired by the proof of the Kalman-Yakubovich-Popov (KYP) Lemma \citep{dullerud2013course}.
%}

%\remind{Hardness: Emphasize (in a remark) why naive gradient descent \eqref{eq:exact_pg}  \emph{may} not work. Current potential reason: no smoothness, so may be no ``uniform'' upper bound on the stepsize for the Armijo rule \cite[Theorem $3.2$]{hintermuller2010nonlinear}, although this is for \emph{unconstrained optimization}, but still, it only converges as $t\to\infty$, and may take forever.}
%
%\remind{Refer back to  Figure \ref{sec:intro} before.}
%
%
%{
%We first establish the following important result to show that with properly chosen stepsize $\eta$, the updates \eqref{eq:exact_npg}-\eqref{eq:exact_gn} preserve the \emph{robustness} as well as the \emph{stability} of the control gain $K$. Note that certifying the stability and robustness is   significant in online control   design, which has been recognized as  one of the \emph{shades of reinforcement learning} by John N. Tsitsiklis recently. We refer to such a result the \emph{implicit regularization} property of the updates \eqref{eq:exact_npg}-\eqref{eq:exact_gn}. 
 
%\subsubsection{Discrete-Time Mixed Design}
 
\begin{theorem}[Implicit Regularization for Discrete-Time Mixed Design]\label{thm:stability_update}
For any control gain $K\in\cK$, i.e., $\rho(A-BK)< 1$ and $\|\cT(K)\|_\infty<{\gamma}$,  with $\|K\|<\infty$, 
	suppose that  the stepsize  $\eta$ satisfies:
	\begin{itemize}
	\item Natural policy gradient \eqref{eq:exact_npg}: $\eta\leq {1}/{(2\|R+B^\top \tP_K B\|)}$,
	\item Gauss-Newton \eqref{eq:exact_gn}: $\eta\leq {1}/{2}$.  
%	\item Gradient descent \eqref{eq:exact_pg}: XXXXX, 
	\end{itemize}	 
%	and   $K$ with $\|K\|<\infty$ lies in $\cK$ defined in \eqref{equ:define_cK}, i.e., 
%	the $\cH_{\infty}$-norm of the transfer function $\cT(K)$ defined in \eqref{equ:dynamic_sys} satisfies 
%	$\|\cT(K)\|_\infty<{\gamma}$ for $\cT(K)$ defined in \eqref{equ:mixed_design_transfer2}.  Equivalently, $K$ satisfies that: i) there exists a solution $P_{K}\geq 0$  to the modified Bellman equation \eqref{equ:discret_riccati};  
%	\eqref{equ:def_PK}-\eqref{equ:def_tPK}; 
%	ii) $I-\gamma^{-2}D^\top P_K D>0$; iii) $\rho\big((I-\gamma^{-2}  P_KDD^\top)^{-\top}(A-BK)\big)<1$. 
	 Then the $K'$ obtained from \eqref{eq:exact_npg}-\eqref{eq:exact_gn} also lies in $\cK$. Equivalently, $K'$ is stabilizing, i.e., $\rho(A-BK')<1$, and satisfies that: i) there exists a solution $P_{K'}\geq 0$  to the Riccati  equation \eqref{equ:discret_riccati};  
%	\eqref{equ:def_PK}-\eqref{equ:def_tPK}; 
	ii) $I-\gamma^{-2}D^\top P_{K'} D>0$; iii) $\rho\big((I-\gamma^{-2}  P_{K'}DD^\top)^{-\top}(A-BK')\big)<1$.  
%	 satisfy 
%	 the conditions i)-iii).  
%	 Equivalently, conditions i)-iii) above also hold for 
%	 the updated $K'$. 
%	 obtained from \eqref{eq:exact_gn}-\eqref{eq:exact_npg} 
%	 also satisfies that: i) there exists a solution $P_{K'}>0$  to the modified Bellman equations \eqref{equ:def_PK}-\eqref{equ:def_tPK}; ii) $W^{-1}-\gamma P_{K'}>0$; iii) $\rho\big((A-BK')^\top(I-\gamma P_{K'}W)^{-1}\big)<1$. 
%	 the $\cH_\infty$-norm of the transfer function $\cT(K)$ defined in \eqref{equ:dynamic_sys} is smaller than $\sqrt{1/\gamma}$. 
\end{theorem}  
%\begin{proof}
\noindent \textit{Proof Sketch.}
The general idea, contrary to the coercivity-based idea that works for any descent direction, is that we focus on the feasibility of $K'$ after an update along \emph{certain directions}: either  \eqref{eq:exact_npg} or \eqref{eq:exact_gn}.  
By Bounded Real Lemma, i.e., Lemma \ref{lemma:discrete_bounded_real_lemma}, the feasibility condition for $K'$, if $K'$ is stabilizing, is equivalent to the existence of $P> 0$ such that the linear matrix inequalities (LMIs) in  \eqref{equ:discrete_equiva_set_cK_cond} hold for $K'$. 
Moreover, it is straightforward to see that such a $P>0$, if exists,  satisfies $(A-BK')^\top P (A-BK')-P<0$, which can be used to show that $K'$ is stabilizing \citep{boyd1994linear}. Thus, it now suffices to find such a $P$. 

To show this, we first study the case with stepsizes being the upper bound in the theorem, i.e., $\eta=1/2$ for Gauss-Newton and $\eta= {1}/{(2\|R+B^\top \tP_K B\|)}$ for natural PG. 
As the solution to the Riccati equation \eqref{equ:discret_riccati} under $K$, $P_K\geq 0$ satisfies $I-\gamma^{-2}D^\top P_K D>0$, the first LMI in \eqref{equ:discrete_equiva_set_cK_cond}.  Hence, it may be possible  to perturb $P_K$ to obtain a $P>0$, such that the equality in \eqref{equ:discret_riccati} becomes a strict inequality of the second LMI in \eqref{equ:discrete_equiva_set_cK_cond}, while preserving the first LMI. Moreover, if $K'$ is not too far away from $K$, such a perturbed $P_K$ should also work for $K'$. Such an observation motivates the use of $P_K$ as the candidate of $P$ for the LMIs in \eqref{equ:discrete_equiva_set_cK_cond} under $K'$. 

Indeed, it can be shown that substituting $P=P_K$ makes the second LMI in \eqref{equ:discrete_equiva_set_cK_cond} under $K'$ \emph{non-strict}, namely, the left-hand side (LHS) $\leq 0$; see  \eqref{equ:before_perturb_GN} in the detailed proof. 
To make it strict, consider the perturbed $P=P_K+\alpha \bar P$ for some $\alpha>0$, where $\bar P>0$ is the solution to some Lyapunov equation
\#\label{equ:def_bar_P_sketch}
 (A-BK)^\top (I-\gamma^{-2} DD^\top  P_K)^{-\top} \bar P(I-\gamma^{-2} DD^\top  P_K)^{-1}(A-BK)-\bar P=-I. 
\#
Such a Lyapunov equation \eqref{equ:def_bar_P_sketch} always admits a solution $\bar P>0$,  since $K\in\cK$ implies that $(I-\gamma^{-2} DD^\top  P_K)^{-1}(A-BK)$ is stable. 
The intuition of choosing \eqref{equ:def_bar_P_sketch} is as follows.  First, the LHS of the second LMI in \eqref{equ:discrete_equiva_set_cK_cond} under $K'$ can be separated as
\#\label{equ:LHS_sep_sketch}
 &(A-BK')^\top \tilde{P}(A-BK')-P+C^\top C+K'^\top R K'\notag\\
 &\quad=\underbrace{[(A-BK')^\top \tilde{P}(A-BK')-(A-BK)^\top \tilde{P}(A-BK)] +K'^\top R K'-K^\top R K}_{\circled{1}}\notag\\
 &\qquad+\underbrace{(A-BK)^\top \tilde{P}(A-BK)-P+C^\top C+K^\top R K}_{\circled{2}}. 
 \#
By some algebra, the first term $\circled{1}$ is of order $o(\alpha)$. Since for small $\alpha$,  
\$
\tP=\tP_K+(I-\gamma^{-2} P_KDD^\top)^{-1}(\alpha \bar P)(I-\gamma^{-2} DD^\top P_K)^{-1}+o(\alpha),
\$
this,  combined with the Riccati equation  \eqref{equ:discret_riccati} and \eqref{equ:def_bar_P_sketch}, makes the second term $\circled{2}=-\alpha I+o(\alpha)$. Hence, there exists small enough $\alpha>0$ such that $\circled{1}+\circled{2}<0$, ensuring that the updated $K'$ is feasible. Lastly, by the linearity of LMIs, any interpolation of $K'$ with a smaller stepsize is also  feasible/robustly stable, thus completing the proof. 
\hfill$\QED$

Note that 
%the argument above does not apply to the vanilla PG update \eqref{eq:exact_pg}, i.e.,
 theoretically, it is not clear yet if vanilla PG enjoys  implicit regularization. In the worst-case, as discussed right after Remark \ref{remark:def_implicit_reg}, vanilla PG may take infinitely many iterations to converge.    
Hence,  hereafter,  we only focus on the global convergence of natural PG method \eqref{eq:exact_npg} and Gauss-Newton method \eqref{eq:exact_gn}, with constant stepsizes.

\subsection{Global Convergence}

The term \emph{global convergence}   here refers to two notions: i) the convergence performance of the algorithms starting from \emph{any feasible initialization} point $K_0\in\cK$; ii) convergence to the \emph{global optimal} policy under certain conditions.  
We formally establish the results for the natural PG \eqref{eq:exact_npg} and Gauss-Newton \eqref{eq:exact_gn} updates in the following theorem.

\begin{theorem}[Global Convergence   for Discrete-Time  Mixed Design]\label{theorem:global_exact_conv} Suppose that $K_0\in\cK$ 
% with $\cK$ defined in \eqref{equ:define_cK} 
 and $\|K_0\|<\infty$.  Then, under the 
%following 
stepsize choices\footnote{In fact, for natural PG \eqref{eq:exact_npg}, it suffices to require the stepsize $\eta\leq{1}/{(2\|R+B^\top \tP_{K_0} B\|)}$ for the initial $K_0$.}  as in Theorem \ref{thm:stability_update},  
%\begin{itemize}
%\item Gauss-Newton \eqref{eq:exact_gn_ct}: $\eta\in[0,1/2]$
%\item Natural policy gradient \eqref{eq:exact_npg_ct}: $[0,1/(2\|R\|)]$,
%\end{itemize}
both   updates  \eqref{eq:exact_npg} and  \eqref{eq:exact_gn}  converge to the global optimum $K^*=(R+B^\top \tP_{K^*} B)^{-1}B^\top \tP_{K^*} A$,  
%stationary points $K$ where $E_{K}=\bm{0}$, 
in the sense that the average of $\{\|E_{K_n}\|_F^2\}$ over iterations converges to zero with $O(1/N)$ rate. 
%In addition, if the pair $\big((I-\gamma^{-2}  P_KDD^\top)^{-\top}(A-BK),D\big)$ is controllable  at the stationary point $K$, then such a convergence is towards the unique \emph{global} optimal policy.  
%
% then under the following stepsize choices 
%%as in Theorem \ref{thm:stability_update},  
%\begin{itemize}
%\item Natural policy gradient \eqref{eq:exact_npg}: $[0,1/(2\|R+B^\top \tP_{K_0} B\|)]$
%\item Gauss-Newton \eqref{eq:exact_gn}: $\eta\in[0,1/2]$,
%\end{itemize}
%both updates \issue{XXXXXX NEED TO CHANGE THE WAY TO SAY IT XXXXXX} \eqref{eq:exact_npg} and \eqref{eq:exact_gn} converge to the optimal control gain $K^*$ with sublinear rate,  in the sense that $\{\|E_{K_n}\|_F^2\}$ converges to zero with $O(1/N)$ rate. 
\end{theorem}

The proof of Theorem \ref{theorem:global_exact_conv} is detailed in \S\ref{sec:proof_theorem:global_exact_conv}. We remark that, the  controllability assumption made in Proposition  \ref{coro:opt_control_form_discrete} is not required for the global convergence here. Remarkably, there might be multiple stationary points such that $\nabla \cJ(K)=0$, while the two specific policy search directions \eqref{eq:exact_npg} and  \eqref{eq:exact_gn} provably avoid the suboptimal local minima, and always converge to the global optimum $K^*$. This can be viewed as another implication of \emph{implicit regularization}, in that  \eqref{eq:exact_npg} and  \eqref{eq:exact_gn} always bias the iterates towards a certain global optimal solution, without getting stuck at spurious local minima. The key reason is that, without using the curvature information in $\Delta_K$, these two PO methods can converge to the specific and optimal  stationary point such that $E_{K}=0$, instead of any arbitrary stationary point.

% and is satisfied automatically for LEQG  problems with $D=W^{1/2}>0$. 
 Moreover,  in contrast to the results for  LQR \citep{fazel2018global},  
only globally 
\emph{sublinear} $O(1/N)$, instead of \emph{linear}, convergence rate can be obtained so far. This  $O(1/N)$    rate of the (iteration average) gradient norm square  matches the  \emph{global} convergence rate of  gradient descent and second order algorithms to  stationary points for  general nonconvex optimization,  either under the smoothness assumption
of the objective \citep{cartis2010complexity,cartis2017worst}, or for a class of non-smooth objectives \citep{khamaru2018convergence}. 

%\issue{XXXXXXXXXXX TO BE CHANGED XXXXXXXXXXX}

%\issue{XXXXXX 2019.10.01 XXXXXX}

%\issue{
%\noindent XXXXXX
%Discussion on the global convergence results. 
%\noindent XXXXXX
%}

\begin{remark}[Robust Initial Controller]\label{remark:initial}
Our global convergence requires  the initial controller to satisfy the $\cH_\infty$-norm robustness constraint, which, as the assumption on the initial controller being stabilizing for LQR  \citep{fazel2018global,bu2019LQR}, is inherent to PO methods with iterative local  search. Complementary to \cite{fazel2018global,bu2019LQR}, our iterates not only improve the performance criterion,  but also preserve the  robustness. 
%We leave the policy-based procedure of finding such an initial controller as a future work. 
\end{remark}

Though sublinear globally, much faster rates, i.e., (super-)linear rates, can be shown locally  around the optimum as below. Proof of the following theorem is deferred to \S\ref{sec:proof_theorem:local_exact_conv}.

%\issue{2019.10.01}
 
\begin{theorem}[Local (Super-)Linear   Convergence for Discrete-Time  Mixed Design]\label{theorem:local_exact_conv} 
Suppose that the  conditions  in Theorem \ref{theorem:global_exact_conv} hold, and additionally  $DD^\top>0$ holds.
Then, under the 
%following 
stepsize choices  as in Theorem \ref{theorem:global_exact_conv}, 
 both    updates  \eqref{eq:exact_npg} and \eqref{eq:exact_gn} converge to the optimal control gain $K^*$ with \emph{locally linear} rate,  in the sense that the objective $\{\cJ(K_n)\}$ defined in \eqref{equ:form_J2} converges to $\cJ(K^*)$ with a linear rate. In addition, if $\eta=1/2$,  the Gauss-Newton update \eqref{eq:exact_gn} converges to   $K^*$ with  a locally \emph{Q-quadratic}  rate. 
% \issue{IN ADDITION, WE WANT TO SHOW SUPER-LINEAR RATE FOR Gauss-Newton!!! CAN WE?? I think it is possible since: i) gradient dominance holds locally; ii) by \eqref{equ:rela_tP_tPK}, the trick for LQR still holds, i.e., $P_K-P_*$ satisfies a Lyapunov equation, locally.}
\end{theorem}

%\issue{XXXX THIS IS NOT CORRECT!! XXXX 2019.10.01}

Key to the locally linear rates is that the property of \emph{gradient dominance} \citep{polyak1963gradient,nesterov2006cubic} holds locally around the optimum for mixed design problems. Such a property   has been shown to hold globally for LQR problems  \citep{fazel2018global}, and also hold locally for zero-sum LQ  games  \citep{zhang2019policy}. 
The Q-quadratic rate  echoes back the  rate of Gauss-Newton with $\eta=1/2$ for LQR problems \citep{hewer1971iterative,bu2019LQR}.  
This globally sublinear and locally (super-)linear convergence resembles  the behavior  of (Quasi)-Newton methods for nonconvex optimization  \citep{nesterov2006cubic,ueda2010convergence}, and policy gradient  methods for zero-sum LQ games \citep{zhang2019policy}.

\begin{remark}[Comparison to  \cite{zhang2019policy}]
Due to the close relationship between mixed design and zero-sum LQ games, see \S \ref{sec:discussion}, one may compare the convergence results and find the rates here (\emph{globally sublinear and locally linear}) not improved over \cite{zhang2019policy}. However, one key difference is that an extra \emph{projection} step is required to guarantee the \emph{stability} of the system in \cite{zhang2019policy}, which is essentially to \emph{regularize} the iterates \emph{explicitly}. More importantly, such a projection can only be calculated under more restrictive assumptions (see Assumption 2.1 therein), which, though cover a class of LQ games,  are not standard in robust control. Here, similar convergence results are established, without projections or non-standard assumptions in robust control, thanks to \emph{implicit regularization}. Moreover, we  have established the local ``superlinear'' rate for the Gauss-Newton update, and whole new set of results for the ``continuous-time'' setup, which were not studied in \cite{zhang2019policy}. 
\end{remark}

%\subsubsection{Continuous-Time Mixed Design}

\section{Proofs of Main Results}

%\issue{XXXXXXX 2019.08.28.  The proof below has been cleared in notations. XXXXXXXX}

In this section, we provide detailed proofs for the main  results of the paper.

\subsection{Proof of Theorem  \ref{thm:stability_update}}\label{sec:proof_thm_stability_update}

%\newlytyped
{
%We first recall the definition of the transfer function $\cT(K)$ defined in \eqref{equ:mixed_design_transfer2}: 
%\$
%\renewcommand\arraystretch{1.3}
%\cT(K):=\mleft[
%\begin{array}{c|c}
%  A-BK & D \\
%  \hline
%  (C^\top C+K^\top R K)^{1/2}& \bm{0} 
%\end{array}
%\mright]. 
%\$

To show that $K'$ lies in $\cK$, we first argue that it suffices to find some $P>0$ such that  
\#
&I-\gamma^{-2}D^\top PD>0,\quad \text{~~and~~}\label{equ:LMI_cond_2}\\
&(A-BK')^\top \tilde{P}(A-BK')-P+C^\top C+(K')^\top R K'<0,\label{equ:LMI_cond_3}
\# 
where $\tP:=P+PD(\gamma^2I-D^\top P D)^{-1}D^\top P$. 
By Schur complement, showing \eqref{equ:LMI_cond_2}-\eqref{equ:LMI_cond_3} is also equivalent to   showing 
\#\label{equ:LMI_cond}
\mleft[
\begin{array}{cc}
  (A-BK')^\top P(A-BK')-P+C^\top C+K'^\top R K'  & (A-BK')^\top PD \\
  D^\top P(A-BK')& -(\gamma^2I-D^\top P D) 
\end{array}
\mright]<0.
\# 
Obviously, if such a $P$ exists, we denote the LHS of  \eqref{equ:LMI_cond_3} by $-M<0$. Thus, \eqref{equ:LMI_cond_2} and \eqref{equ:LMI_cond_3} imply  
\$
&(A-BK')^\top P (A-BK')-P\\
&\quad=-M-C^\top C-(K')^\top R K'-(A-BK')^\top PD(\gamma^2 I-D^\top PD)^{-1}D^\top P(A-BK')\leq -M<0,
\$
which shows that $K'$ is stabilizing, i.e., $\rho(A-BK')<1$ \citep{boyd1994linear}. 
Thus, Lemma \ref{lemma:discrete_bounded_real_lemma} can be applied to $K'$. Then, \eqref{equ:LMI_cond_2} and \eqref{equ:LMI_cond_3} are identical to \eqref{equ:discrete_equiva_set_cK_cond}, which further shows that $\|\cT(K')\|_\infty <\gamma$ and thus shows $K'\in\cK$. Hereafter we will focus on finding such a $P>0$.

We   first show that for   the Gauss-Newton update \eqref{eq:exact_gn} with stepsize $\eta=1/2$, \eqref{equ:LMI_cond_2} and \eqref{equ:LMI_cond_3} hold for some $P>0$.  Specifically, we have 
\#\label{equ:K_prime_GN}
K'=K-(R+B^\top \tilde P_K B)^{-1}E_K=(R+B^\top \tilde P_K B)^{-1}B^\top\tilde P_K A.
\#
Since  $P_{K}\geq 0$   satisfies the conditions i)-iii), and $K$ and $K'$ are close to each other, we can choose   
$P_{K}$ as a candidate of $P$. 
%Thus, \eqref{equ:LMI_cond_2} holds trivially for $P_K$ by condition ii). 
%Recall that modified Bellman equation gives that  
%\#\label{equ:mod_bellman_restate}
%P_K
%%=Q+K^\top RK+(A-BK)^\top \big[P_K+\gamma P_K(W^{-1}-\gamma P_K)^{-1}P_K\big](A-BK)
%=Q+K^\top RK+(A-BK)^\top \tilde P_K(A-BK).
%\#
 Hence, by Riccati equation \eqref{equ:discret_riccati}, the LHS of \eqref{equ:LMI_cond_3} can be written as 
 \#\label{equ:before_perturb_GN}
  &[A-B(R+B^\top \tilde P_K B)^{-1}B^\top\tilde P_K A]^\top \tilde{P}_K[A-B(R+B^\top \tilde P_K B)^{-1}B^\top\tilde P_K A]-P_K+C^\top C\notag\\
  &\quad +[(R+B^\top \tilde P_K B)^{-1}B^\top\tilde P_K A]^\top R [(R+B^\top \tilde P_K B)^{-1}B^\top\tilde P_K A]\notag\\
%  &=A^\top  \tilde P_K A-A^\top  \tilde P_K B(R+B^\top \tilde P_K B)^{-1}B^\top \tilde P_K A-P_K +Q\notag\\
%  &=A^\top  \tilde P_K A-A^\top  \tilde P_K B(R+B^\top \tilde P_K B)^{-1}B^\top \tilde P_K A-K^\top RK
%  -(A-BK)^\top \tilde P_K(A-BK)\notag\\
%  &=-A^\top  \tilde P_K B(R+B^\top \tilde P_K B)^{-1}B^\top \tilde P_K A-K^\top (R+B^\top \tilde P_K B)K+K^\top B^\top \tilde P_K A+A^\top \tilde P_K BK\notag\\
  &=-\big[(R+B^\top \tilde P_K B)^{-1}B^\top \tilde P_KA-K\big]^\top (R+B^\top \tilde P_K B)\big[(R+B^\top \tilde P_K B)^{-1}B^\top \tilde P_KA-K\big]\leq 0, 
 \#
 where we substitute $K'$ from  \eqref{equ:K_prime_GN}, 
% , and  the second equation follows from \eqref{equ:mod_bellman_restate},  
 and the last equation is due to completion of  the squares.  
 
 Now we need to perturb $P_K$ to obtain a $P$, such that \eqref{equ:before_perturb_GN} holds with a \emph{strict} inequality.  
 To this end, we define $\bar P>0$ as the   solution to the Lyapunov equation
 \#\label{equ:def_bar_P}
 (A-BK)^\top (I-\gamma^{-2} DD^\top  P_K)^{-\top} \bar P(I-\gamma^{-2} DD^\top  P_K)^{-1}(A-BK)-\bar P=-I,
 \# 
 and let $P=P_K+\alpha \bar P>0$ for some $\alpha>0$. 
 By Lemma \ref{lemma:discrete_bounded_real_lemma},   $(I-\gamma^{-2} DD^\top  P_K)^{-1}(A-BK)$ is stable, and thus the solution $\bar P>0$  exists. 
 For  \eqref{equ:LMI_cond_2} to hold, we need a small   $\alpha>0$ to  satisfy 
 \#\label{equ:alpha_cond_1}
  \alpha D^\top \bar PD<\gamma^{2}I-D^\top P_KD.
 \#
 Moreover, the LHS of \eqref{equ:LMI_cond_3} now can be written as 
 \#\label{equ:LHS_sep}
 &(A-BK')^\top \tilde{P}(A-BK')-P+C^\top C+K'^\top R K'\notag\\
 &\quad=\underbrace{[(A-BK')^\top \tilde{P}(A-BK')-(A-BK)^\top \tilde{P}(A-BK)] +K'^\top R K'-K^\top R K}_{\circled{1}}\notag\\
 &\qquad+\underbrace{(A-BK)^\top \tilde{P}(A-BK)-P+C^\top C+K^\top R K}_{\circled{2}},
 \#
 where we aim to show that there exists some $\alpha>0$ such that $\circled{1}+\circled{2}<0$. 
 Note that $\tP$ can be written as
 \#\label{equ:rela_tP_tPK} 
 \tP&=[I-\gamma^{-2} (P_K+\alpha \bar P)DD^\top]^{-1}(P_K+\alpha \bar P)\notag\\
 \quad&=\big[(I-\gamma^{-2} P_KDD^\top)^{-1}+(I-\gamma^{-2} P_KDD^\top)^{-1}(\alpha\gamma^{-2}\bar PDD^\top)(I-\gamma^{-2} P_KDD^\top)^{-1}+o(\alpha)\big](P_K+\alpha \bar P)\notag\\
% \quad&=\tP_K+(I-\gamma^{-2} P_KDD^\top)^{-1}(\alpha\gamma^{-2}\bar PDD^\top)(I-\gamma^{-2} P_KDD^\top)^{-1}P_K+\alpha (I-\gamma^{-2} P_KDD^\top)^{-1} \bar P+o(\alpha) \notag\\
 \quad&=\tP_K+(I-\gamma^{-2} P_KDD^\top)^{-1}(\alpha \bar P)(I-\gamma^{-2} DD^\top P_K)^{-1}+o(\alpha),
 \#
 where the first equation follows from  definition,  and the second one uses the fact that
 \$
 (X+Y)^{-1}=X^{-1}-X^{-1}YX^{-1}+o(\|Y\|),
 \$
 for small perturbation $Y$ around the matrix $X$. 
% , and the last  one   follows by direct calculation. 
 Thus, $\circled{1}$ can  be written as 
% \small
 \#\label{equ:upper_bnd_circle_1}
 \circled{1}
% &=[-K'^\top B^\top \tilde P A-A^\top \tP BK'+K'^\top B^\top \tP B K'+K^\top B^\top \tP A+A^\top \tP BK-K^\top B^\top \tP BK] +K'^\top R K'-K^\top R K\notag\\
 &=-K'^\top B^\top \tilde P A-A^\top \tP BK'+K'^\top (R+B^\top \tP B )K'+K^\top B^\top \tP A+A^\top \tP BK-K^\top (R+B^\top \tP B)K \notag\\
% &=-K'^\top B^\top \tilde P A-A^\top \tP BK'+K'^\top (R+B^\top \tP B )K'+A^\top \tP B (R+B^\top \tP B )^{-1}B^\top \tP A\notag\\
% &\quad -\big[(R+B^\top \tilde P B)^{-1}B^\top \tilde PA-K\big]^\top (R+B^\top \tilde P B)\big[(R+B^\top \tilde P B)^{-1}B^\top \tilde P A-K\big]\notag\\
 &\leq -K'^\top B^\top \tilde P A-A^\top \tP BK'+K'^\top (R+B^\top \tP B )K'+A^\top \tP B (R+B^\top \tP B )^{-1}B^\top \tP A \notag\\
 &= \big[(R+B^\top \tilde P B)^{-1}B^\top \tilde PA-K'\big]^\top (R+B^\top \tilde P B)\big[(R+B^\top \tilde P B)^{-1}B^\top \tilde P A-K'\big],
 \#
 \normalsize
 where the inequality  follows by completing squares. 
 By substituting in $K'$ from  \eqref{equ:K_prime_GN}, we further have
 \#\label{equ:circle_1_res}
 \circled{1}&\leq \big[(R+B^\top \tilde P B)^{-1}B^\top \tilde PA-(R+B^\top \tilde P_K B)^{-1}B^\top\tilde P_K A\big]^\top (R+B^\top \tilde P B)\notag\\
 &\qquad\cdot\big[\underbrace{(R+B^\top \tilde P B)^{-1}B^\top \tilde P A}_{\circled{3}}-(R+B^\top \tilde P_K B)^{-1}B^\top\tilde P_K A\big].
 \# 
 Note that by \eqref{equ:rela_tP_tPK}, we have
 \#\label{equ:circle_3_res}
 \circled{3}
% &=\big\{R+B^\top \big[\tP_K+(I-\gamma^{-2} P_KDD^\top)^{-1}(\alpha \bar P)(I-\gamma^{-2} DD^\top P_K)^{-1}+o(\alpha)\big] B\big\}^{-1}B^\top \tilde P A\notag\\
 &=\big[(R+B^\top\tP_K B)^{-1}-(R+B^\top\tP_K B)^{-1}B^\top(I-\gamma^{-2} P_KDD^\top)^{-1}(\alpha \bar P)(I-\gamma^{-2} DD^\top P_K)^{-1}\notag\\
 &\qquad\cdot B(R+B^\top\tP_K B)^{-1}+o(\alpha)\big]  B^\top\big[\tP_K+(I-\gamma^{-2} P_KDD^\top)^{-1}(\alpha \bar P)(I-\gamma^{-2} DD^\top P_K)^{-1}+o(\alpha)\big] A \notag\\
 &=(R+B^\top \tilde P_K B)^{-1}B^\top\tilde P_K A+O(\alpha)+o(\alpha),
 \#
 where $O(\alpha)$ denotes the quantities that have the order of $\alpha$. By plugging \eqref{equ:circle_3_res} into \eqref{equ:circle_1_res}, we obtain that $
 \circled{1}=o(\alpha)$. 
 
 Moreover, by \eqref{equ:rela_tP_tPK}, 
 $\circled{2}$ can be written as 
 \small
 \#\label{equ:circle_2_res}
 \circled{2}&=(A-BK)^\top \big[\tP_K+(I-\gamma^{-2} P_KDD^\top)^{-1}(\alpha \bar P)(I-\gamma^{-2} DD^\top P_K)^{-1}+o(\alpha)\big](A-BK)-P+C^\top C+K^\top R K\notag\\
 &=(A-BK)^\top (I-\gamma^{-2} P_KDD^\top)^{-1}(\alpha \bar P)(I-\gamma^{-2} DD^\top P_K)^{-1}(A-BK)-\alpha\bar P+o(\alpha)=-\alpha I+o(\alpha),
 \# 
 \normalsize 
 where the first equation uses \eqref{equ:rela_tP_tPK}, the second one uses the Riccati equation \eqref{equ:discret_riccati}, and the last one uses \eqref{equ:def_bar_P}. 
  Therefore, for small enough $\alpha>0$ such that $\circled{1}+\circled{2}<0$, and also satisfies \eqref{equ:alpha_cond_1}, there exists some $P>0$ such that  both \eqref{equ:LMI_cond_2} and \eqref{equ:LMI_cond_3} hold for $K'$ obtained with stepsize $\eta=1/2$. On the other hand,  
%  since  such an $\alpha$ (and thus $P$) makes \eqref{equ:LMI_cond_2} hold and $\circled{2}<0$, we know that 
  such a $P$ also makes the LMI  \eqref{equ:LMI_cond} hold for $K$, i.e., 
  \#\label{equ:LMI_cond_K}
\mleft[
\begin{array}{cc}
  (A-BK)^\top P(A-BK)-P+C^\top C+K^\top R K  & (A-BK)^\top PD \\
  D^\top P(A-BK)& -(\gamma^2I-D^\top P D) 
\end{array}
\mright]<0,
\# 
as now $\circled{1}$ in \eqref{equ:LHS_sep} is null, and the same $\alpha$ above makes $\circled{2}<0$.  
For $\eta\in[0,1/2]$, 
let  $K_{\eta}=K+2\eta(K'-K)$ be the interpolation between $K$ and $K'$. 
 Combining  \eqref{equ:LMI_cond} and \eqref{equ:LMI_cond_K} yields  
\#\label{equ:LMI_cond_4}
0>&2\eta\cdot\mleft[
\begin{array}{cc}
  (A-BK')^\top P(A-BK')-P+C^\top C+K'^\top R K'  & (A-BK')^\top P \notag\\
  P(A-BK')& -(\gamma^2I-D^\top P D)
\end{array}
\mright]\notag\\
&\quad+(1-2\eta)\cdot\mleft[
\begin{array}{cc}
  (A-BK)^\top P(A-BK)-P+C^\top C+K^\top R K  & (A-BK)^\top P \notag\\
  P(A-BK)& -(\gamma^2I-D^\top P D) 
\end{array}
\mright]\notag\\
\geq&\mleft[
\begin{array}{cc}
  (A-BK_\eta)^\top P(A-BK_\eta)-P+C^\top C+K_\eta^\top R K_\eta  & (A-BK_\eta)^\top P \\
  P(A-BK_\eta)& -(\gamma^2I-D^\top P D) 
\end{array}
\mright], 
\#
for any $\eta\in[0,1/2]$. 
The second inequality in \eqref{equ:LMI_cond_4} uses the 
%the fact that
% \$
% &2\eta[(A-BK')^\top P(A-BK')-P+Q+K'^\top R K']+(1-2\eta)[(A-BK)^\top P(A-BK)-P+C^\top C+K^\top R K]\\
% &\quad \geq A^\top PA-A^\top PBK_\eta-K_\eta^\top B^\top PA+K_\eta^\top B^\top PBK_\eta-P+C^\top C+K_\eta^\top R K_\eta\\
% &\quad = (A-BK_\eta)^\top P(A-BK_\eta)-P+C^\top C+K_\eta^\top R K_\eta,
% \$
% which is due to the 
 convexity of the quadratic form. Hence, \eqref{equ:LMI_cond_4} shows that for any stepsize $\eta\in[0,1/2]$, $K_\eta$ that lies between $K$ and $K'$ satisfies the conditions i)-iii) in the theorem.  
 
 Now we prove a similar result for the natural PG update \eqref{eq:exact_npg}. 
 Recall that 
 \#\label{equ:restate_exact_npg}
 K'=K-2\eta [(R+B^\top \tP_{K} B)K-B^\top \tP_{K} A]. 
 \#
 As before, we first choose $P=P_K$. Then,   the LHS of \eqref{equ:LMI_cond_3} under $K'$ can be written as:
 \#\label{equ:natural_trash_1}
 &(A-BK')^\top \tilde{P}_K(A-BK')-P_K+C^\top C+K'^\top R K'\notag\\
% &\quad=
% [B(K-K')]^\top \tP_K A+A^\top\tP_K[B(K-K')]+K'^\top (R+B^\top \tP_K B) K'-K^\top (R+B^\top \tP_K B) K\notag\\
% &\quad=[K'-(R+B^\top\tP_KB)^{-1}B^\top \tP_K A]^\top (R+B^\top\tP_KB)[K'-(R+B^\top\tP_KB)^{-1}B^\top \tP_K A]\notag\\ 
% &\quad\quad -[K-(R+B^\top\tP_KB)^{-1}B^\top \tP_K A]^\top (R+B^\top\tP_KB)[K-(R+B^\top\tP_KB)^{-1}B^\top \tP_K A]\notag\\
 &\quad=(K'-K)^\top (R+B^\top\tP_KB)[K'-(R+B^\top\tP_KB)^{-1}B^\top \tP_K A]\notag\\ 
 &\quad\quad +[K-(R+B^\top\tP_KB)^{-1}B^\top \tP_K A]^\top (R+B^\top\tP_KB)(K'-K), 
 \#
 where the 
% first equation uses the modified Bellman equation \eqref{equ:discret_riccati}, the second  one follows by completing the squares, and the
   equation holds by adding and subtracting $[K-(R+B^\top\tP_KB)^{-1}B^\top \tP_K A]^\top (R+B^\top\tP_KB)[K'-(R+B^\top\tP_KB)^{-1}B^\top \tP_K A]$. 
 Substituting     \eqref{equ:restate_exact_npg} into  \eqref{equ:natural_trash_1} yields 
 \#\label{equ:natural_trash_2}
 &(A-BK')^\top \tilde{P}_K(A-BK')-P_K+C^\top C+K'^\top R K'\notag\\
% &\quad=-2\eta [(R+B^\top \tP_{K} B)K-B^\top \tP_{K} A]^\top (R+B^\top\tP_KB)[K'-(R+B^\top\tP_KB)^{-1}B^\top \tP_K A]\notag\\ 
% &\quad\quad -2\eta [K-(R+B^\top\tP_KB)^{-1}B^\top \tP_K A]^\top (R+B^\top\tP_KB)[(R+B^\top \tP_{K} B)K-B^\top \tP_{K} A]\notag\\
 &\quad=-2\eta [(R+B^\top \tP_{K} B)K-B^\top \tP_{K} A]^\top[(R+B^\top \tP_{K} B)K-B^\top \tP_{K} A] \notag\\
 &\quad\quad +4\eta^2 [(R+B^\top \tP_{K} B)K-B^\top \tP_{K} A]^\top(R+B^\top \tP_{K} B)[(R+B^\top \tP_{K} B)K-B^\top \tP_{K} A]\notag\\ 
&\quad\quad -2\eta [(R+B^\top\tP_KB)K-B^\top \tP_K A]^\top [(R+B^\top \tP_{K} B)K-B^\top \tP_{K} A]. 
 \#
 By requiring the stepsize $\eta$ to satisfy
 \#\label{equ:natural_eta_cond_1}
 \eta\leq \frac{1}{2\|R+B^\top\tP_KB\|},
 \#
 we can bound \eqref{equ:natural_trash_2} as  
 \#\label{equ:natural_trash_3}
 &(A-BK')^\top \tilde{P}_K(A-BK')-P_K+C^\top C+K'^\top R K'\notag\\
 &\quad\leq  
 -2\eta\cdot[(R+B^\top\tP_KB)K-B^\top \tP_K A]^\top [(R+B^\top \tP_{K} B)K-B^\top \tP_{K} A]\leq 0,
 \#
 namely, letting $P=P_K$ leads to the desired LMI that is not strict. 
 
 Now suppose that $P=P_K+\alpha \bar P$ for some $\alpha>0$, where $\bar P>0$ is the solution to \eqref{equ:def_bar_P}. 
 Note that $\alpha$ first still needs to satisfy \eqref{equ:alpha_cond_1}.  
   Also, the LHS of \eqref{equ:LMI_cond_3} 
  can still be separated into $\circled{1}$ and $\circled{2}$ as in \eqref{equ:LHS_sep}. 
From the LHS of the inequality in \eqref{equ:upper_bnd_circle_1}, we have
\small
\#\label{equ:natural_trash_4}
\circled{1}
%&=\big[(R+B^\top \tilde P B)^{-1}B^\top \tilde PA-K'\big]^\top (R+B^\top \tilde P B)\big[(R+B^\top \tilde P B)^{-1}B^\top \tilde P A-K'\big]\notag\\
% &\quad -\big[(R+B^\top \tilde P B)^{-1}B^\top \tilde PA-K\big]^\top (R+B^\top \tilde P B)\big[(R+B^\top \tilde P B)^{-1}B^\top \tilde P A-K\big]\notag\\
 &=(K'-K)^\top (R+B^\top\tP B)[K'-(R+B^\top\tP B)^{-1}B^\top \tP A]  +[K-(R+B^\top\tP B)^{-1}B^\top \tP A]^\top (R+B^\top\tP B)(K'-K)\notag\\
 &=-2\eta [(R+B^\top \tP_{K} B)K-B^\top \tP_{K} A]^\top [(R+B^\top\tP B)K'-B^\top \tP A]\notag\\ 
 &\qquad\qquad\qquad\qquad\qquad -2\eta [(R+B^\top\tP B)K-B^\top \tP A]^\top [(R+B^\top \tP_{K} B)K-B^\top \tP_{K} A]
\#
\normalsize
where 
%the first equation is by completing the squares, 
the first equation follows by adding and subtracting $[(R+B^\top \tilde P B)^{-1}B^\top \tilde PA-K]^\top (R+B^\top \tilde P B)[(R+B^\top \tilde P B)^{-1}B^\top \tilde P A-K']$, and the second one follows from the definition of $K'$ in \eqref{equ:restate_exact_npg}. Suppose that $\eta$ satisfies \eqref{equ:natural_eta_cond_1}, then  we have 
\#
 &\circled{1}=-4\eta [(R+B^\top \tP_{K} B)K-B^\top \tP_{K} A]^\top[(R+B^\top \tP_{K} B)K-B^\top \tP_{K} A] \notag\\
 &\quad\quad +4\eta^2 [(R+B^\top \tP_{K} B)K-B^\top \tP_{K} A]^\top(R+B^\top \tP_{K} B)[(R+B^\top \tP_{K} B)K-B^\top \tP_{K} A]\notag\\
 &\quad\quad-2\alpha\eta [(R+B^\top \tP_{K} B)K-B^\top \tP_{K} A]^\top [ (B^\top\bar P B)K'-B^\top \bar P A]\notag\\ 
 &\quad\quad -2\alpha\eta [ (B^\top\bar P B)K-B^\top \bar P A]^\top [(R+B^\top \tP_{K} B)K-B^\top \tP_{K} A] \label{equ:natural_trash_5}\\
 &\quad\leq -2\eta[(R+B^\top\tP_KB)K-B^\top \tP_K A]^\top [(R+B^\top \tP_{K} B)K-B^\top \tP_{K} A]\notag\\
 &\quad\quad-2\alpha\eta [(R+B^\top \tP_{K} B)K-B^\top \tP_{K} A]^\top [ (B^\top\bar P B)K'-B^\top \bar P A]\notag\\ 
 &\quad\quad -2\alpha\eta [ (B^\top\bar P B)K-B^\top \bar P A]^\top [(R+B^\top \tP_{K} B)K-B^\top \tP_{K} A],\label{equ:natural_trash_6}
 \#
 where the equation follows by separating $\tP$ as $\tP_K+\alpha \bar P$ in \eqref{equ:natural_trash_4}.  Notice that the first two terms on the right-hand side (RHS) of \eqref{equ:natural_trash_5} are identical to the RHS of \eqref{equ:natural_trash_2}. Thus, the inequality \eqref{equ:natural_trash_6} is due to \eqref{equ:natural_trash_3}. Moreover, notice that
 \small
 \$
&-2\alpha\eta[(R+B^\top \tP_{K} B)K-B^\top \tP_{K} A]^\top [ (B^\top\bar P B)K'-B^\top \bar P A] -2\alpha\eta[ (B^\top\bar P B)K-B^\top \bar P A]^\top [(R+B^\top \tP_{K} B)K-B^\top \tP_{K} A] \\
&=-2\alpha\eta[(R+B^\top \tP_{K} B)K-B^\top \tP_{K} A]^\top [ (B^\top\bar P B)K'-B^\top \bar P A] -2\alpha\eta[ (B^\top\bar P B)K'-B^\top \bar P A]^\top [(R+B^\top \tP_{K} B)K-B^\top \tP_{K} A]\\
&\quad -4\alpha\eta^2[(R+B^\top \tP_{K} B)K-B^\top \tP_{K} A]^\top(B^\top\bar P B)[(R+B^\top \tP_{K} B)K-B^\top \tP_{K} A]\\
&\leq2\eta[(R+B^\top \tP_{K} B)K-B^\top \tP_{K} A]^\top[(R+B^\top \tP_{K} B)K-B^\top \tP_{K} A]+2\alpha^2\eta[ (B^\top\bar P B)K'-B^\top \bar P A]^\top [ (B^\top\bar P B)K'-B^\top \bar P A]\\
&\quad -4\alpha\eta^2[(R+B^\top \tP_{K} B)K-B^\top \tP_{K} A]^\top(B^\top\bar P B)[(R+B^\top \tP_{K} B)K-B^\top \tP_{K} A],
\$ 
 \normalsize
 which combined with  \eqref{equ:natural_trash_6} further yields  
 \#\label{equ:natural_trash_7}
 \circled{1}\leq&  2\alpha^2\eta[ (B^\top\bar P B)K'-B^\top \bar P A]^\top [ (B^\top\bar P B)K'-B^\top \bar P A].  
% \notag\\
% &\quad+\eta \alpha^2[ (B^\top\bar P B)K-B^\top \bar P A]^\top [ (B^\top\bar P B)K-B^\top \bar P A].
 \#
 By assumption $\|K\|<\infty$ and $P_K\geq 0$ exists, we know that $\tP_K$ is bounded, and so is  $K'$ obtained from \eqref{equ:restate_exact_npg} using a finite stepsize $\eta$. 
  Also, $\bar P$ has bounded norm. Thus, $\circled{1}$ in     \eqref{equ:natural_trash_7} is $o(\alpha)$.  
% and the inequality uses \eqref{equ:natural_trash_2} and \eqref{equ:natural_trash_3},  if $\eta$  satisfies \eqref{equ:natural_eta_cond_1}. 
%Since $K$ is stabilizing, it lies in the set $\{K\given \rho(A-BK)<1\}$ and thus has bounded norm. 
%Also, $P_K$ and $\bar P$ both have bounded norm, so do $\tP_K$ and  $K'$. Hence, 
Thus, there exists small enough $\alpha>0$ such that $\circled{1}+\circled{2}<0$, since   from \eqref{equ:circle_2_res},  $\circled{2}=-\alpha I+o(\alpha)$. In words, there exists some $P>0$ such that  \eqref{equ:LMI_cond} holds for $K'$ obtained from \eqref{equ:restate_exact_npg} with stepsize satisfying \eqref{equ:natural_eta_cond_1}. 
 This completes the proof of the first argument.

Lastly, by  Lemma \ref{lemma:discrete_bounded_real_lemma},  we equivalently have that the conditions i)-iii) in the theorem hold for $K'$, which
 completes the proof. 
\hfill$\QED$

\subsection{Proof of Theorem \ref{theorem:global_exact_conv}}\label{sec:proof_theorem:global_exact_conv}

We first introduce the following lemma that can be viewed as the counterpart of  the \emph{Cost Difference Lemma} in \cite{fazel2018global}. Unlike the equality relation given in the lemma in \cite{fazel2018global}, we establish both lower and upper bounds for the difference of two matrices  $P_{K'}$ and $P_K$. The proof of the lemma is provided in \S\ref{sec:proof_lemma_cost_diff}.
%The result is formally stated in the following lemma. 

\begin{lemma}[Discrete-Time Cost Difference Lemma]\label{lemma:cost_diff}
Suppose that both $K,K'\in\cK$. 
% with $\cK$ being defined in \eqref{equ:define_cK}.
Then, we have the  following upper bound:
\#
P_{K'}-P_K&\le \sum_{t\geq 0} [(A-BK')^\top(I-\gamma^{-2} P_{K'}DD^\top)^{-1}]^t\big[ -(K-K')^\top E_K - E_K^\top (K-K')\notag\\
&\qquad+(K-K')^\top (R+B^\top \tP_{K} B)(K-K')\big][(I-\gamma^{-2} P_{K'}DD^\top)^{-\top}(A-BK')]^t, \label{eq:CDL_upper}
\#
where  $E_K$ is defined in \eqref{equ:def_mu_Ek}.
%have finite costs and thus $P_K$ and $P_{K'}$ exist. Also, suppose $W^{-1}>\gamma P_K$, $W^{-1}>\gamma P_{K'}$, and thus both $I-\gamma P_{K} W$ and {$I-\gamma P_{K'} W$ are  invertible}. 
If additionally  $\rho\big((A-BK')^\top(I-\gamma^{-2} P_{K}  DD^\top)^{-1}\big)<1$, 
then we also have the   lower bound:
%\#\label{eq:CDL1}
%&P_{K'}-P_K \leq  \sum_{t=0} [(A-BK')^\top(I-\gamma P_{K'}W)^{-1}]^t\big[Q+(K')^\top RK' \notag\\
%&\qquad\qquad\qquad+ (A-BK')^\top \tP_K (A-BK')-P_K\big][(I-\gamma P_{K'}W)^{-\top}(A-BK')]^t. 
%\#
%In addition, we have
%\begin{align}
%\label{eq:CDL2}
%&Q+(K')^\top R K'+(A-BK')^\top \tP_K(A-BK')-P_K\notag\\
%&\quad=2(K'-K)^\top E_K+(K'-K)^\top (R+B^\top \tP_{K} B)(K'-K),
%\end{align}
%where recall $E_K$ is defined in \eqref{equ:def_mu_Ek}. 
%Combining the above two inequalities, we have 
\#
P_{K'}-P_K&\geq \sum_{t\geq 0} [(A-BK')^\top(I-\gamma^{-2} P_{K}  DD^\top)^{-1}]^t\big[ -(K-K')^\top E_K - E_K^\top (K-K') \notag\\
&\qquad+(K-K')^\top (R+B^\top \tP_{K} B)(K-K')\big][(I-\gamma^{-2} P_{K}DD^\top)^{-\top}(A-BK')]^t. \label{eq:CDL_lower}
\# 
\end{lemma}

Notice that Lemma \ref{lemma:cost_diff} also resembles the ``Almost Smoothness Condition"  (Lemma $9$) in \cite{fazel2018global}, which characterizes how the difference between  $P_K$ and $P_{K'}$ relies on the difference between $K$ and $K'$.   Due to the difference between $\tP_K$ and $P_K$, we have to establish lower and upper bounds of $P_{K'}-P_K$ separately.   One can still identify that the leading terms in both bounds depend on $\|K'-K\|$, with the remaining terms being in the order of $o(\|K'-K\|)$ if $K'$ is close to $K$.  
%But still, as $\tP_K$, and thus $E_K$, may be unbounded as $K$ approaches the boundary $\partial \cK$, $P_K$ is 

%. We do have new technical difficulty here in LEQR since we have to play with the relationship between $P_K$ and $\tP_K$. This causes a lot of trouble and eventually we fix this technical issue using \eqref{eq:linA1}.

\vspace{10pt}
\noindent{\textbf{Gauss-Newton:}}
\vspace{4pt} 
  
Recall that  for the Gauss-Newton update, $K'=K-2\eta(R+B^\top \tP_K B)^{-1} E_K$.  
By Theorem \ref{thm:stability_update}, $K'$ also lies in $\cK$ if $\eta\leq 1/2$.
% and $\|K\|<\infty$.   
Then, by the upper bound in \eqref{eq:CDL_upper}, we know  that if $\eta\in[0,1/2]$, 
\begin{align}\label{equ:monotone_p}
&P_{K'}-P_K\le(-4\eta+4\eta^2)\sum_{t\ge 0} [(A-BK')^\top(I-\gamma^{-2} P_{K'}DD^\top)^{-1}]^t\left[ E_K^\top(R+B^\top \tP_K B)^{-1}E_K \right]\notag\\
&\qquad\qquad\qquad\qquad\qquad\qquad\cdot[(I-\gamma^{-2} P_{K'}DD^\top)^{-\top}(A-BK')]^t\le 0, 
\end{align} 
which implies the monotonic decrease of $P_K$ (matrix-wise) along the update. 
Since $P_K$ is lower-bounded, such a monotonic sequence of $\{P_{K_n}\}$ along the iterations must converges to some $P_{K_{\infty}}\in\cK$. Now we show that this $P_{K_{\infty}}$ is indeed $P_{K^*}$.  
By multiplying both sides of \eqref{equ:monotone_p} with 
any matrix  $M> 0$, and then taking the trace, we have  that if $\eta\in[0,1/2]$ 
\begin{align}\label{equ:p_trace_upper_bnd_gn}
&\tr(P_{K'}M)-\tr(P_KM)
\le(-4\eta+4\eta^2)\tr\bigg\{\sum_{t\ge 0} [(A-BK')^\top(I-\gamma^{-2} P_{K'}M)^{-1}]^t\notag\\
&\qquad\qquad\qquad\qquad\qquad\qquad\cdot\left[ E_K^\top(R+B^\top \tP_K B)^{-1}E_K \right][(I-\gamma^{-2} P_{K'}M)^{-\top}(A-BK')]^t M\bigg\}\notag\\
%&\quad \leq -2\eta \tr\bigg\{\left[ E_K^\top(R+B^\top \tP_K B)^{-1}E_K \right]\cdot\notag\\
%&\qquad\qquad\qquad\qquad
%%\underbrace{
%\sum_{t\ge 0} [(I-\gamma P_{K'}W)^{-\top}(A-BK')]^t W[(A-BK')^\top(I-\gamma P_{K'}W)^{-1}]^t
%%}_{\cW_{K',K'}}
%\bigg\}\notag\\
&\quad\leq-2\eta\tr\big[E_K^\top(R+B^\top \tP_K B)^{-1}E_K M\big]
%\leq-2\eta\sigma_{\min}(W)\tr\big[E_K^\top(R+B^\top \tP_K B)^{-1}E_K \big]
\leq \frac{-2\eta\sigma_{\min}(M)}{\sigma_{\max}(R+B^\top \tP_K B)}\tr(E_K^\top E_K)\notag\\
&\quad \leq \frac{-2\eta\sigma_{\min}(M)}{\sigma_{\max}(R+B^\top \tP_{K_0} B)}\tr(E_K^\top E_K)
,
\end{align}
where 
%$\cW_{K',K'}\geq 0$ is a nonnegative definite matrix depending on $K'$, and 
the second inequality follows by keeping only the first term in the infinite summation of positive definite matrices, the third one uses that $\tr(PA)\geq \sigma_{\min}(A)\tr(P)$, and the last one  is due to the monotonic decrease of $P_K$,  and the monotonicity of  $\tP_{K}$ with respect to $P_{K}$,  with  $K_0\in\cK$ being the initialization of $K$ at iteration $0$.  
From iterations $n=0$ to $N-1$, replacing $M$ by $I$, summing over both sides of \eqref{equ:p_trace_upper_bnd_gn},  and dividing by $N$, we further  have
\$
\frac{1}{N}\sum_{n=0}^{N-1}\tr(E_{K_n}^\top E_{K_n})\leq \frac{\sigma_{\max}(R+B^\top \tP_{K_0} B)\cdot \big[\tr(P_{K_0})-\tr(P_{K_{\infty}})\big]}{2\eta\cdot N},
\$
namely, the sequence $\{K_n\}$ converges to the stationary point $K$ such that $E_K=0$ with $O(1/N)$ rate. By Proposition  \ref{coro:opt_control_form_discrete}, this is towards the global optimal control gain $K^*$.  
%if $\big((I-\gamma^{-2} P_{K}DD^\top)^{-\top}(A-BK),D\big)$ is controllable, then such a stationary point is unique, and is indeed the \emph{global} optimal policy dictated by $K^*$. 

%Since $E_K=\bm{0}$ gives the unique optimal solution $K^*=(R+B^\top \tP_{K^*} B)^{-1}B^\top \tP_{K^*} A$, this shows that the sequence $\{\|E_{K_n}\|_F^2\}$ converges to zero with sublinear rate $O(1/N)$, namely, 
%$\{K_n\}$ converges to the optimal control gain $K^*$ with sublinear rate. 

\vspace{10pt}
\noindent{\textbf{Natural  Policy Gradient:}}
\vspace{4pt} 
 
Recall that the natural PG update follows $K'=K-2\eta E_K$. 
By Theorem \ref{thm:stability_update}, $K'$ also lies in $\cK$ if $\eta\leq 1/(2\|R+B^\top \tP_K B\|)$.  
By  the upper bound  \eqref{eq:CDL_upper}, this stepsize  yields that 
\small
\begin{align}\label{equ:monotone_p_2}
&P_{K'}-P_K
%\le\sum_{t\ge 0} [(A-BK')^\top(I-\gamma P_{K'}W)^{-1}]^t\left[-4\eta E_K^\top E_K+4\eta^2E_K^\top(R+B^\top \tP_K B)^{-1}E_K \right]\notag\\
%&\qquad\qquad\qquad\qquad\qquad\qquad\cdot[(I-\gamma P_{K'}W)^{-\top}(A-BK')]^t\notag\\
\le\sum_{t\ge 0} [(A-BK')^\top(I-\gamma^{-2} P_{K'}DD^\top)^{-1}]^t\left[-4\eta E_K^\top E_K+2\eta E_K^\top E_K \right][(I-\gamma^{-2} P_{K'}DD^\top)^{-\top}(A-BK')]^t\notag\\
&\quad\le 0, 
\end{align} 
\normalsize
which also implies the matrix-wise monotonic decrease of $P_{K}$ along the update. Suppose the convergent matrix is $P_{K_{\infty}}$.   
As before, multiplying both sides of \eqref{equ:monotone_p_2} by $M>0$, and taking the trace, yields
\begin{align}\label{equ:p_trace_upper_bnd_ng}
&\tr(P_{K'}M)-\tr(P_KM)
%\le(-4\eta+4\eta^2)\tr\bigg\{\sum_{t\ge 0} [(A-BK')^\top(I-\gamma P_{K'}W)^{-1}]^t\left[ E_K^\top(R+B^\top \tP_K B)^{-1}E_K \right]\notag\\
%&\qquad\qquad\qquad\qquad\qquad\qquad\cdot[(I-\gamma P_{K'}W)^{-\top}(A-BK')]^t W\bigg\}\notag\\
%&\quad \leq -2\eta \tr\bigg\{\left[ E_K^\top(R+B^\top \tP_K B)^{-1}E_K \right]\cdot\notag\\
%&\qquad\qquad\qquad\qquad
%%\underbrace{
%\sum_{t\ge 0} [(I-\gamma P_{K'}W)^{-\top}(A-BK')]^t W[(A-BK')^\top(I-\gamma P_{K'}W)^{-1}]^t
%%}_{\cW_{K',K'}}
%\bigg\}\notag\\
\leq-2\eta\tr\big(E_K^\top E_K M\big)
%\leq-2\eta\sigma_{\min}(W)\tr\big[E_K^\top(R+B^\top \tP_K B)^{-1}E_K \big]
%\leq {-2\eta\sigma_{\min}(W)}\tr(E_K^\top E_K)\notag\\
%&\quad \leq \frac{-2\eta\sigma_{\min}(W)}{\sigma_{\max}(R+B^\top \tP_{K_0} B)}\tr(E_K^\top E_K)
,
\end{align} 
for any $M>0$, 
where the inequality follows by only keeping the first term in the infinite summation. Letting $M=I$, summing up \eqref{equ:p_trace_upper_bnd_ng} from $n=0$ to $n=N-1$, and dividing by $N$, we conclude that 
\$
\frac{1}{N}\sum_{n=0}^{N-1}\tr(E_{K_n}^\top E_{K_n})\leq \frac{\tr(P_{K_0})-\tr(P_{K_{\infty}})}{2\eta\cdot N},
\$ 
namely,  $\{K_n\}$ converges to the stationary point $K$ such that $E_K=$ with $O(1/N)$ rate, which is also the global optimum. 
In addition, since $\{P_{K_n}\}$ is monotonically decreasing, it suffices to require the stepsize 
\$
\eta\in[0,1/(2\|R+B^\top \tP_{K_0} B\|)],
\$
 which completes the proof.  
\hfill$\QED$

\subsection{Proof of Theorem \ref{theorem:local_exact_conv}}\label{sec:proof_theorem:local_exact_conv}

To ease the analysis, we show the convergence rate of a surrogate value $\tr(P_K DD^\top )$. This is built upon the following  relationship between the objective value $\cJ(K)$ and $\tr(P_K DD^\top )$.

\begin{lemma}\label{lemma:logdet_to_trace}
	Suppose that both $K,K'\in\cK$ and  $P_K\geq P_{K'}$. Then, it follows that 
\$
	\cJ(K)-\cJ(K')
%	=-\frac{1}{\gamma}\log\det (I-\gamma P_{K'}W)+\frac{1}{\gamma}\log\det (I-\gamma P_{K}W)
\le \|(I-\gamma^{-2}  D^\top P_{K} D)^{-1}\|\cdot [\tr(P_{K} DD^\top)-\tr(P_{K'} DD^\top)].
	\$
\end{lemma}
\begin{proof}
First, by Sylvester's determinant theorem, 
%,   $\det (I+AB)=\det(I+BA)$. Thus, 
$\cJ(K)$ can be re-written as 
\$
\cJ(K)=-\gamma^2\log\det (I-\gamma^{-2}P_KDD^\top)=-\gamma^2\log\det (I-\gamma^{-2}D^\top P_KD). 
\$
By the mean value theorem, for any matrices $A$ and $B$ with $\det(A),\det(B)>0$, we have
\$
\log\det (A)=\log\det(B)+\tr[(B+\tau (A-B))^{-1}(A-B)]
\$
for some $0\le\tau\le 1$. This leads to  
\small
\$
&\cJ(K)-\cJ(K')=-{\gamma^2}\log\det (I-\gamma^{-2} D^\top P_{K}D )+{\gamma^2}\log\det (I-\gamma^{-2} D^\top P_{K'}D )=\tr[XD^\top(P_{K}-P_{K'})D ]\notag\\
&\qquad\leq \|X\|\cdot [\tr(D^\top P_K D)-\tr(D^\top P_{K'} D)]=\|X\|\cdot [\tr( P_K DD^\top)-\tr( P_{K'} DD^\top)],
\$
\normalsize
where $X=(I-\gamma^{-2} \tau D^\top P_{K'} D -\gamma^{-2} (1-\tau) D^\top P_K D )^{-1}$, and the inequality uses the facts $P_{K}\geq P_{K'}$ and $\tr(PA)\leq \|A\|\cdot \tr(P)$ for any real symmetric $P\geq 0$.  
 Note that by $P_{K}\geq P_{K'}$, 
 \$
  X\leq (I-\gamma^{-2}  D^\top P_{K} D)^{-1}\Longrightarrow \|X\|\leq \|(I-\gamma^{-2}  D^\top P_{K} D)^{-1}\|. 
 \$
% . 
% Moreover, $K\in\cK$ implies that $\|(W^{-1}-\gamma P_{K} )^{-1}\|\leq \|W\|$ uniformly, which combined with \eqref{equ:trash_thm_22_1}  
 This completes   the proof. 
\end{proof}

Lemma \ref{lemma:logdet_to_trace} implies  that 
in order to show the convergence of $\cJ(K)$,    
it suffices to study the convergence  of $\tr(P_{K} DD^\top)$, as long as $\|(I-\gamma^{-2}  D^\top P_{K} D)^{-1}\|$ is bounded along the iterations. 
This is indeed the case since by  \eqref{equ:monotone_p} and  \eqref{equ:monotone_p_2}, $P_K$ is monotone along both   updates \eqref{eq:exact_npg} and \eqref{eq:exact_gn}. 
By induction, if $K_0\in\cK$, i.e., {$I-\gamma^{-2}  D^\top P_{K_0} D>0$, then   $I-\gamma^{-2}  D^\top P_{K_n} D\geq I-\gamma^{-2}  D^\top P_{K_0} D>0$ holds for  all iterations  $n\geq 1$. 
This further yields that   for all $n\geq 1$, $
\|(I-\gamma^{-2}  D^\top P_{K_n} D)^{-1}\|\leq\|(I-\gamma^{-2}  D^\top P_{K_0} D)^{-1}\|$, 
namely,  $\|(I-\gamma^{-2}  D^\top P_{K} D)^{-1}\|$ is uniformly bounded. 
}

Now we show the local linear convergence rate of $\tr(P_K DD^\top)$. By   \eqref{eq:CDL_lower},  for any $K'$ such that $(I-\gamma^{-2} P_{K}DD^\top)^{-\top}(A-BK')$ is stabilizing, we have

\#\label{equ:p_diff_lower_bnd}
&P_{K'}-P_K\geq \sum_{t\ge 0} [(A-BK')^\top(I-\gamma^{-2} P_{K}DD^\top)^{-1}]^t\big[ -(K-K')^\top E_K - E_K^\top (K-K')  \notag\\
&\qquad\qquad\qquad+(K-K')^\top (R+B^\top \tP_{K} B)(K-K')\big]\cdot[(I-\gamma^{-2} P_{K}DD^\top)^{-\top}(A-BK')]^t\\
&\quad\geq \sum_{t\ge 0} [(A-BK')^\top(I-\gamma^{-2} P_{K}DD^\top)^{-1}]^t\big[-E_K^\top (R+B^\top \tP_{K} B)^{-1}E_K\big]\cdot[(I-\gamma^{-2} P_{K}DD^\top)^{-\top}(A-BK')]^t,\notag
\#
where the second inequality follows from completion of   squares. 
%
%By completing the squares, we have
%\#\label{equ:p_diff_lower_bnd_trash}
%&-2(K-K')^\top E_K+(K'-K)^\top (R+B^\top \tP_{K} B)(K'-K)\notag\\
%&=[K'-K+(R+B^\top \tP_{K} B)^{-1}E_K]^\top (R+B^\top \tP_{K} B) [K'-K+(R+B^\top \tP_{K} B)^{-1}E_K]\notag\\
%&\qquad-E_K^\top (R+B^\top \tP_{K} B)^{-1}E_K\notag\\ 
%&\geq -E_K^\top (R+B^\top \tP_{K} B)^{-1}E_K,
%\#
%which can be plugged into 
%\eqref{equ:p_diff_lower_bnd} to yield
%\$
%&P_{K'}-P_K\geq \sum_{t\ge 0} [(A-BK')^\top(I-\gamma P_{K}W)^{-1}]^t\big[ -E_K^\top (R+B^\top \tP_{K} B)^{-1}E_K\big]\cdot[(I-\gamma P_{K}W)^{-\top}(A-BK')]^t.
%\$
By taking traces on both sides of \eqref{equ:p_diff_lower_bnd}, and letting  $K'=K^*$, we have
\small 
\#\label{equ:trace_diff_lower_bnd_2}
\tr(P_{K}DD^\top)-\tr(P_{K^*}DD^\top)&\leq \tr\left[ E_K^\top(R+B^\top \tP_K B)^{-1}E_K \right]\cdot\|\cW_{K,K^*}\|\leq \frac{\tr\left( E_K^\top E_K \right)}{\sigma_{\min}(R)}\cdot\|\cW_{K,K^*}\|,
\#
\normalsize
where $\cW_{K,K^*}$ is defined as  
\$
\cW_{K,K^*}:={\sum_{t\ge 0} [(I-\gamma^{-2} P_{K}DD^\top)^{-\top}(A-BK^*)]^t DD^\top[(A-BK^*)^\top(I-\gamma^{-2} P_{K}DD^\top)^{-1}]^t}. 
\$
%i.e., the solution to the  Lyapunov equation
%\$
%(I-\gamma P_{K}W)^{-\top}(A-BK^*)\cW_{K,K^*}(A-BK^*)^\top(I-\gamma P_{K}W)^{-1}+W=\cW_{K,K^*},
%\$
Note that $K^*\in\cK$ and thus $(I-\gamma^{-2} P_{K^*}DD^\top)^{-\top}(A-BK^*)$ is stabilizing. 
Let $\epsilon:=1-\rho\big((I-\gamma^{-2} P_{K^*}DD^\top)^{-\top}(A-BK^*)\big)$, and note that  $\epsilon>0$.  
By the continuity of $P_K$, 
% from  Lemma \ref{lemma:differentiability_policy_grad} 
 and that of $\rho(\cdot)$ \citep{tyrtyshnikov2012brief}, there exists a ball $\cB(K^*,r)\subseteq\cK$,  centered at  $K^*$ with radius $r>0$,  such that  for any $K\in \cB(K^*,r)$,  
\#\label{equ:linear_rate_trash_0}
\rho\big((I-\gamma^{-2} P_{K}DD^\top)^{-\top}(A-BK^*)\big)\leq 1-\epsilon/2<1. 
\#

\vspace{10pt}
\noindent{\textbf{Gauss-Newton:}}
\vspace{4pt} 

By Theorem \ref{theorem:global_exact_conv}, $\{K_n\}$ approaches $K^*$. Thus, there exists some $K_n\in\cB(K^*,r)$. Let $K=K_n$ and thus $K'=K_{n+1}$. 
Replacing $M$ in \eqref{equ:p_trace_upper_bnd_gn}  by $DD^\top>0$ and combining \eqref{equ:trace_diff_lower_bnd_2}, we have 
\begin{align*}
\tr(P_{K'}DD^\top)-\tr(P_KDD^\top)\leq \frac{-2\eta\sigma_{\min}(DD^\top)\sigma_{\min}(R)}{\sigma_{\max}(R+B^\top \tP_{K_0} B)\|\cW_{K,K^*}\|}[\tr(P_{K}DD^\top)-\tr(P_{K^*}DD^\top)],
\end{align*}
which further   implies   that
\small
\begin{align}\label{equ:linear_rate_trash_1}
\tr(P_{K'}DD^\top)-\tr(P_{K^*}DD^\top)\leq \bigg(1-\frac{2\eta\sigma_{\min}(DD^\top)\sigma_{\min}(R)}{\sigma_{\max}(R+B^\top \tP_{K_0} B)\|\cW_{K,K^*}\|}\bigg)\cdot[\tr(P_{K}DD^\top)-\tr(P_{K^*}DD^\top)]. 
\end{align}
\normalsize
\eqref{equ:linear_rate_trash_1} shows that the sequence $\{\tr(P_{K_{n+p}}DD^\top)\}$ decreases to $\tr(P_{K^*}DD^\top)$ starting from some $K_n\in \cB(K^*,r)$. By continuity, there must exists a close enough $K_{n+p}$, such that the lower-level set $\{K\given \tr(P_KDD^\top)\leq \tr(K_{n+p}DD^\top)\}\subseteq \cB(K^*,r)$. Hence, starting from $K_{n+p}$, the iterates  will never leave $\cB(K^*,r)$. By \eqref{equ:linear_rate_trash_0}, $\cW_{K,K^*}$, as the unique solution to the  Lyapunov equation
\$
[(I-\gamma^{-2} P_{K}DD^\top )^{-\top}(A-BK^*)]\cW_{K,K^*}[(I-\gamma^{-2} P_{K}DD^\top)^{-\top}(A-BK^*)]+DD^\top=\cW_{K,K^*}, 
\$
 must have its norm  bounded by   some constant $\overline{\cW}_{r}>\|DD^\top\|$   for all $K\in\cB(K^*,r)$. Replacing the term $\|\cW_{K,K^*}\|$ in \eqref{equ:linear_rate_trash_1} by $\overline{\cW}_{r}$ gives the uniform local linear contraction of $\{\tr(P_{K_n}DD^\top)\}$, which further leads to the local linear rate of $\{\cJ(K_n)\}$ by Lemma \ref{lemma:logdet_to_trace}. 
 
In addition,  by the upper bound \eqref{eq:CDL_upper} and $E_{K^*}=0$, we have
\#\label{equ:trash_1_dis_q_quad}
\tr(P_{K'}DD^\top)-\tr(P_{K^*}DD^\top)&\leq \tr\Big\{\sum_{t\geq 0} [(A-BK')^\top(I-\gamma^{-2} P_{K'}DD^\top)^{-1}]^t\big[(K'-K^*)^\top (R+B^\top \tP_{K^*} B)\notag\\
&\qquad\quad\cdot (K'-K^*)\big][(I-\gamma^{-2} P_{K'}DD^\top)^{-\top}(A-BK')]^tDD^\top \Big\}. 
\#
For $\eta=1/2$, suppose that some $K=K_n\in\cB(K^*,r)$. Then, $K'=K_{n+1}=(R+B^\top \tilde P_{K} B)^{-1}B^\top\tilde P_{K} A$  yields that
\#\label{equ:trash_2_dis_q_quad}
 &K'-K^*=(R+B^\top \tilde P_{K} B)^{-1}B^\top\tilde P_{K} A-(R+B^\top \tilde P_{K^*} B)^{-1}B^\top\tilde P_{K^*} A\\
% &\quad=[(R+B^\top \tilde P_{K} B)^{-1}-(R+B^\top \tilde P_{K^*} B)^{-1}]B^\top\tilde P_{K} A+[(R+B^\top \tilde P_{K^*} B)^{-1}B^\top(\tP_{K}-\tP_{K^*})A]\notag\\
 &\quad=(R+B^\top \tilde P_{K} B)^{-1}B^\top (\tP_{K}-\tP_{K^*}) B (R+B^\top \tilde P_{K^*} B)^{-1} B^\top\tilde P_{K} A+[(R+B^\top \tilde P_{K^*} B)^{-1}B^\top(\tP_{K}-\tP_{K^*})A]\notag. 
\#
Moreover, notice that 
\#\label{equ:trash_4_dis_q_quad}
&\tP_{K}-\tP_{K^*}=(I-\gamma^{-2} P_K DD^\top)^{-1}P_K-(I-\gamma^{-2} P_{K^*} DD^\top)^{-1}P_{K^*}\\
&\quad=(I-\gamma^{-2} P_{K^*} DD^\top)^{-1}\gamma^{-2}(P_{K}-P_{K^*})DD^\top(I-\gamma^{-2} P_{K} DD^\top)^{-1}+(I-\gamma^{-2} P_{K^*} DD^\top)^{-1}(P_{K}-P_{K^*}),\notag
\#
which, combined with \eqref{equ:trash_2_dis_q_quad}, gives  
\#\label{equ:trash_3_dis_q_quad}
\|K'-K^*\|_F\leq c\cdot\|P_{K}-P_{K^*}\|_F,
\#
for some constant $c>0$. 
Combining    \eqref{equ:trash_1_dis_q_quad} and \eqref{equ:trash_3_dis_q_quad} yields 
\$
\tr(P_{K'}DD^\top)-\tr(P_{K^*}DD^\top)&\leq c'\cdot [\tr(P_{K}DD^\top)-\tr(P_{K^*}DD^\top)]^2,
\$
for some constant $c'$. Note that from some $p\geq 0$ such that  $K_{n+p}$  onwards never leaves $\cB(K^*,r)$, the constant $c'$ is uniformly bounded, which proves the Q-quadratic convergence rate of $\{\tr(P_{K_n}DD^\top)\}$, and thus the rate of $\{\cJ(K_n)\}$, around $K^*$. 

% let $A_{K_n}:=A-BK_n$, then by \eqref{equ:K_prime_GN} we know   $K_{n+1}=(R+B^\top \tilde P_{K_n} B)^{-1}B^\top\tilde P_{K_n} A$, which  combined with  the modified Bellman equation \eqref{equ:discret_riccati} yields
%\#\label{equ:rubbish1_dis_q_quad}
%P_{K_n}-P_{K_{n+1}}-A_{K_{n+1}}^\top (\tP_{K_n}-\tP_{K_{n+1}})A_{K_{n+1}}&=A_{K_n}^\top\tP_{K_n} A_{K_n}+ K_n^\top RK_n-A_{K_{n+1}}^\top\tP_{K_n} A_{K_{n+1}}-K_{n+1}^\top RK_{n+1}\notag\\
%& =(K_{n}-K_{n+1})^\top(R+B^\top \tP_{K_n}B)(K_{n}-K_{n+1}),
%\#
%a Lyapunov equation. As we know that $\{K_n\}$ converges to  $K^*$, so does $\{P_{K_n}\}$ to $P_{K^*}$, letting $n+1$
%Hence, as the solution, $P_{K_n}-P_{K_{n+1}}$ 
%
%
%
%\#\label{equ:rubbish2_dis_q_quad}
%\tr(P_{K'}DD^\top)-\tr(P_{K^*}DD^\top)\leq 
%\#
%\\
%\\
%\issue{2019.08.28}
%\noindent\remind{WE NEED TO ESTABLISH THE Q-QUADRATIC RATE FOR $\eta=1/2$!!} 

\vspace{10pt}
\noindent{\textbf{Natural  Policy Gradient:}}
\vspace{4pt} 

Replacing $M$ in \eqref{equ:p_trace_upper_bnd_ng}   by $DD^\top>0$ and 
combining \eqref{equ:p_trace_upper_bnd_ng} and \eqref{equ:trace_diff_lower_bnd_2} yield  
\begin{align*}
\tr(P_{K'}DD^\top)-\tr(P_{K^*}DD^\top)\leq \bigg(1-\frac{2\eta\sigma_{\min}(R)}{\|\cW_{K,K^*}\|}\bigg)\cdot[\tr(P_{K}DD^\top)-\tr(P_{K^*}DD^\top)]. 
\end{align*} 
Using similar argument as above, one can establish the local linear rate of $\{\cJ(K_n)\}$ with a different contracting factor. 
This concludes the proof. 
\hfill$\QED$
}

%\newpage

\section{Discussions} \label{sec:discussion}

We  now provide additional discussions on the mixed $\cH_2/\cH_\infty$ control design problem.

\subsection{Connection to Zero-Sum LQ Games}\label{subsec:connection_to_games}

It is well known that minimizing the risk-sensitive cost as \eqref{equ:def_obj}, which is the logarithm of the  expected values of exponential functions with  quadratic forms,  can be equivalent to solving a zero-sum dynamic game, for both general settings \citep{whittle1990risk,fleming1992risk,fleming1997risk}, and in particular for LQ settings \citep{jacobson1973optimal}. 
Due to the connection between LEQG and mixed design, as discussed in \S\ref{sec:formulation}, the latter can be related to a zero-sum LQ  game as well. 

Specifically, consider  
%a zero-sum LQ game where 
the system  that follows linear dynamics\footnote{The notation in this section might be slightly abused, considering the notations used in the main text, but shall be self-evident by the context.}  
\$
x_{t+1}=Ax_t+Bu_t+Dv_t, 
\$
with the system state being  $x_t\in\RR^d$, the control inputs  of players $1$ and $2$ being    
$u_t\in\RR^{m_1}$ and $v_t\in\RR^{m_2}$, respectively. The matrices $A,B$, and $D$ all have proper dimensions. The objective of  player $1$ (player $2$) is to minimize (maximize) the infinite-horizon value function,
\#\label{equ:minimax_def}
\inf_{\{u_t\}}\sup_{\{v_t\}}\quad\EE_{x_0\sim\cD}\bigg[\sum_{t=0}^\infty c_t(x_t,u_t,v_t)\bigg]=\EE_{x_0\sim\cD}\bigg[\sum_{t=0}^\infty (x_t^\top Qx_t+u_t^\top R^uu_t-v_t^\top R^vv_t)\bigg],
\#
where the initial state $x_0\sim \cD$ for some  distribution $\cD$, the matrices  $Q\in\RR^{d\times d}$, $R^u\in\RR^{m_1\times m_1}$, and $R^v\in\RR^{m_2\times m_2}$ are all positive definite. 
\emph{Value} of the game, i.e., the value of \eqref{equ:minimax_def} when the $\inf$ and $\sup$  can interchange, is characterized by $\EE_{x_0\sim\cD}(x_0^\top P^*x_0)$, where $P^*$ is the solution to the generalized algebraic Riccati equation (GARE) \citep{bacsar1995h} 
%\small
%\#\label{equ:P_GARE}
\$
P^*=A^\top P^*A +Q-
\begin{bmatrix}
A^\top P^*B & A^\top P^*D
\end{bmatrix}
\begin{bmatrix}
R^u+B^\top P^*B & B^\top P^*D\\
D^\top P^*B & -R^v+D^\top P^*D  
\end{bmatrix}^{-1}
\begin{bmatrix}
B^\top P^*A \\
D^\top P^*A 
\end{bmatrix}.
%\#
\$
Moreover, under the standard assumption that $R^v-D^\top P^*D>0$, the solution policies, i.e., the \emph{Nash equilibrium} (NE) policies that are stabilizing, of the two players have forms of LTI state-feedback, namely, $u_t^*=-K^*x_t$ and $v_t^*=-L^*x_t$ for some matrices $K^*\in\RR^{m_1\times d}$ and $L^*\in\RR^{m_2\times d}$. The corresponding  values of $(K^*,L^*)$ are  given by
\small 
\#
K^*&=\big\{R^u+B^\top [P^*- P^*D(-R^v+D^\top P^*D)^{-1}D^\top P^*]B\big\}^{-1} B^\top P^*[A-D(-R^v+D^\top P^*D)^{-1}D^\top P^*A],\label{equ:closed_form_Ks}\\
L^*&=\big\{-R^v+D^\top [P^*- P^*B(R^u+B^\top P^*B)^{-1}B^\top P^*]D\big\}^{-1} D^\top P^* [A-B(R^u+B^\top P^*B)^{-1}B^\top P^*A]. \label{equ:closed_form_Ls}
\#
\normalsize
As a consequence, it suffices to search over all stabilizing control gain pairs $(K,L)$ that solves
\#\label{equ:nonconvex_concave_def}  
\min_K\max_L ~~\cC(K,L):=\EE_{x_0\sim\cD}\bigg\{\sum_{t=0}^\infty \big[x_t^\top Qx_t+(Kx_t)^\top R^u(Kx_t)-(Lx_t)^\top R^v(Lx_t)\big]\bigg\}. 
\# 
In fact, for any stabilizing   $(K,L)$ that makes $\rho(A-BK-DL)<1$, $\cC(K,L)=\tr(P_{K,L}\Sigma_0)$,  where $\Sigma_0=\EE_{x_0\sim\cD}(x_0^\top x_0)$, and $P_{K,L}$ is the unique solution to the Lyapunov equation
\#\label{equ:P_KL_Sol}
P_{K,L}=Q+K^\top R^u K-L^\top R^v L+(A-BK-DL)^\top P_{K,L}(A-BK-DL). 
\#

Given $K$ that makes\footnote{This condition on $K$ is necessary for finding the equilibrium policy since otherwise, the maximizer can drive the cost to infinity by choosing $L$. See Chapter $3$ of \cite{bacsar1995h} for more discussions.}  $R^v-D^\top P_{K,L} D>0$, maximizing over $L$ on the RHS of \eqref{equ:P_KL_Sol} yields 
\#\label{equ:P_KL_Sol_max_L}
P_{K}^*&=Q+K^\top R^u K+(A-BK)^\top [\underbrace{P_K^*+P_K^*D(R^v-D^\top P_K^*D)^{-1}D^\top P_K^*}_{\tP_K^*}](A-BK),
\#
where $P_{K}^*=P_{K,L(K)}$ with  $L(K)$ being the maximizer that satisfies 
\#\label{equ:minimizer_L_given_K}
L(K)=(-R^v+D^\top P_{K}^*D)^{-1}D^\top P_{K}^*(A-BK). 
\#
Notice that \eqref{equ:P_KL_Sol_max_L} is in fact a Riccati equation  identical to   \eqref{equ:discret_riccati}, with $R^v$ replaced by $\gamma^2 I$, $R^u$ replaced by $R$, and $Q$ replaced by $C^\top C$. 
Hence, the problem \eqref{equ:nonconvex_concave_def} 
 is equivalent to minimizing $\cC(K,L(K))=\tr(P_K^*\Sigma_0)$, subject to \eqref{equ:P_KL_Sol_max_L}, which coincides with the   mixed design problem \eqref{equ:def_mixed_formulation}, where  $\cJ(K)$ takes the form of \eqref{equ:form_J1} with $DD^\top$ replaced by $\Sigma_0$.  
Furthermore, the minimizer of   the RHS on \eqref{equ:P_KL_Sol_max_L} is 
\#\label{equ:Ks_game_NE}
K^*=\big(R^u+B^\top \tP_K^* B\big)^{-1}B^\top \tP_K^*A,
\#
which equals the   global optimum 
%under the controllability condition  in Corollary \ref{coro:opt_control_form_discrete}) \issue{(TO BE CHANGED!!)} 
for the mixed design problems. 

%In the $\cH$-infinity control problem, for example,  we can choose  $R^v=\bar{\gamma}^2\Ib$ with $\bar{\gamma}$ being the upper bound on the desired $\ell_2$ gain disturbance attenuation. 
%In addition, we have the following inequality by definition
%\#\label{equ:weak_duality}
%\inf_{\{u_t\}}\sup_{\{v_t\}}~~\EE_{x_0\sim\cD}\bigg[\sum_{t=0}^\infty c_t(x_t,u_t,v_t)\bigg]\geq \sup_{\{v_t\}}\inf_{\{u_t\}}~~\EE_{x_0\sim\cD}\bigg[\sum_{t=0}^\infty c_t(x_t,u_t,v_t)\bigg],
%\#
%where we refer the left and right-hand  sides of \eqref{equ:weak_duality} as the \emph{upper-value} and \emph{lower-value} of the game \eqref{equ:minimax_def}, respectively. If the upper and  lower values are equal, we refer to it as the \emph{value} of the game. 
%If the solution to \eqref{equ:minimax_def} exists and the infimum and supremum in  \eqref{equ:minimax_def} can interchange, we refer to the solution value in \eqref{equ:minimax_def} as the \emph{value} of the game.

\subsection{Model-Free Algorithms}\label{subsec:model_free}

The connection above provides one  angle  to develop \emph{model-free} RL algorithms for solving mixed design problems. Indeed,    the natural PG in    \eqref{eq:exact_npg}  cannot be sampled using  trajectory data, due to the form of the matrix $\Delta_K$ in \eqref{equ:def_Delta}.
 Fortunately, solution  of the game \eqref{equ:nonconvex_concave_def} can be obtained by model-free PG-based methods, see \cite{zhang2019policy}, and the more recent work \cite{bu2019global}, which, by \eqref{equ:Ks_game_NE} and Proposition    \ref{coro:opt_control_form_discrete},   is equivalent to the 
% stationary point, and also  the 
 global optimum of mixed design problems. 
%  under certain conditions. 
Hence, model-free algorithms that solve the LQ game  \eqref{equ:nonconvex_concave_def} can also be used to solve the mixed design problem  \eqref{equ:def_mixed_formulation}.

Specifically, 
\begin{comment}
the PG of $\cC(K,L)$ w.r.t. $K$ and $L$ can be written as   \citep[Lemma 3.2]{zhang2019policy}
\#
\nabla_K \cC(K,L)=&~2[(R^u+B^\top P_{K,L}B)K-B^\top P_{K,L}(A-DL)]\Sigma_{K,L}\label{equ:policy_grad_K_form}\\
\nabla_L \cC(K,L)=&~2[(-R^v+D^\top P_{K,L}D)L-D^\top P_{K,L}(A-BK)]\Sigma_{K,L},\label{equ:policy_grad_L_form}
\#
where $
\Sigma_{K,L}:=\EE_{x_0\sim\cD}\sum_{t=0}^\infty x_tx_t^\top$ is the correlation matrix. By 
\end{comment}
by Lemma 3.3 of \cite{zhang2019policy}, under certain conditions,   the stationary point $(K,L)$ where $\nabla_K \cC(K,L)=0$ and $\nabla_L \cC(K,L)=0$ coincides with the NE.  
\begin{comment}
Indeed, $\nabla_L \cC(K,L)=0$ with the invertibility of $\Sigma_{K,L}$ yields the maximizer $L(K)$ in \eqref{equ:minimizer_L_given_K}, which combined with $\nabla_K \cC(K,L(K))=0$ in \eqref{equ:policy_grad_K_form} gives the solution \eqref{equ:Ks_game_NE}. 
\end{comment}
Therefore, it is straightforward to develop  PG-based updates  to find the minimizer $L(K)$ for some $K$, and then perform PG-based algorithms to update $K$, which can both be implemented in a model-free fashion, using zeroth-order methods \citep{nesterov2017random,fazel2018global}.  Note that the PO methods in \cite{zhang2019policy} are essentially  also based on this idea, but with the order of $\max$ and $\min$ interchanged, and require a projection step for updating $L$. 
More recently, \cite{bu2019global} has developed double-loop PO methods that remove this  projection. 
Two examples of PG-based methods can be written as  
\#
{\rm \textbf{Policy Gradient:}}~~~~\qquad L'&=L+\alpha \nabla_L\cC(K,L)\approx L+\alpha \widehat{\nabla}_L\cC(K,L)\label{equ:model_free_game_PG_L}\\
  K'&=K-\eta \nabla_K\cC(K,{L(K)})\approx K-\eta \widehat{\nabla}_K\cC(K,\widehat{L(K)})\label{equ:model_free_game_PG_K}\\
 {\rm \textbf{Natural PG:}}~~~~\qquad\qquad L'&=L+\alpha \nabla_L\cC(K,L)\Sigma_{K,L}^{-1}\approx L+\alpha \widehat{\nabla}_L\cC(K,L)\widehat{\Sigma}_{K,L}^{-1} &\label{equ:model_free_game_NPG_L}\\
 K'&=K-\eta \nabla_K\cC(K,{L(K)})\Sigma_{K,{L(K)}}^{-1}\approx K-\eta \widehat{\nabla}_K\cC(K,\widehat{L(K)}) \widehat{\Sigma}_{K,\widehat{L(K)}}^{-1}\label{equ:model_free_game_NPG_K}
\# 
where $\alpha,\eta>0$ are stepsizes, $
\Sigma_{K,L}:=\EE_{x_0\sim\cD}\sum_{t=0}^\infty x_tx_t^\top$ with $u_t=-Kx_t$ and $v_t=-Lx_t$ is the correlation matrix under control pair $(K,L)$, $\widehat{L(K)}$ is the estimate of $L(K)$ obtained by iterating either  \eqref{equ:model_free_game_PG_L} or \eqref{equ:model_free_game_NPG_L},  $\widehat{\nabla}_L\cC(K,L)$, $\widehat{\nabla}_K\cC(K,L)$, and $\widehat{\Sigma}_{K,L}$ are the estimates of $\nabla_L\cC(K,L)$, $\nabla_K\cC(K,L)$, and $\Sigma_{K,L}$ using sampled data, respectively. 
 
Note that the simulator for the game \eqref{equ:nonconvex_concave_def} that generates the data samples can be obtained by the simulator for the mixed design problem \eqref{equ:def_mixed_formulation}, with the disturbance $w_t$ modeled as $w_t=-Lx_t$.   This way, the updates of $L$ in \eqref{equ:model_free_game_PG_L} and \eqref{equ:model_free_game_NPG_L} can be understood as improving the disturbance to  find the \emph{worst-case} one, which manifests the idea of $\cH_\infty$ norm.   
Also, \cite{bu2019global} has verified that in zero-sum LQ games, given a fixed $K$, such an update of $L$ converges to the best-response disturbance $L(K)$ given in \eqref{equ:minimizer_L_given_K}. This justifies the feasibility of our algorithms \eqref{equ:model_free_game_PG_L}-\eqref{equ:model_free_game_NPG_K}. 

In addition, by the form of the  policy gradients for  the game, see Lemma 3.2 in \cite{zhang2019policy},  the exact natural PG update on the LHS of  \eqref{equ:model_free_game_NPG_K} is identical to that for mixed design problems in \eqref{eq:exact_npg}. In other words, the natural PG update   \eqref{eq:exact_npg} can be implemented in a model-free way by virtue of that outer-loop  update of $K$ in a zero-sum LQ game. As shown in  \cite{bu2019global}, such an outer-loop update over $K$  converges to the NE of the game. 
Details of the model-free algorithms are deferred to Algorithms \ref{alg:est_grad_corre}, \ref{alg:model_free_inner_NPG}, and \ref{alg:model_free_outer_NPG} in \S\ref{sec:pseudo_code}.

\section{Simulations}\label{sec:simulations}

In this section, we present some simulation results to corroborate our theory. We mainly focus on the convergence properties for the discrete-time settings. We have also included extensive numerical comparisons with existing packages for solving $\cH_2/\cH_\infty$ mixed design, which can only handle the continuous-time settings \citep{arzelier2011h2,mahmoud1996h}. The problem setup, PO algorithms, and their analyses, for the continuous-time settings can be found in \S\ref{sec:aux_cont_res}. We  show that our PO methods outperform these existing packages in many aspects, though the latter ones can handle more general setups.   

\subsection{Implicit Regularization \& Global/Local Convergence}

\begin{figure*}[!t]
	\centering
	\includegraphics[width = 1.0\textwidth]{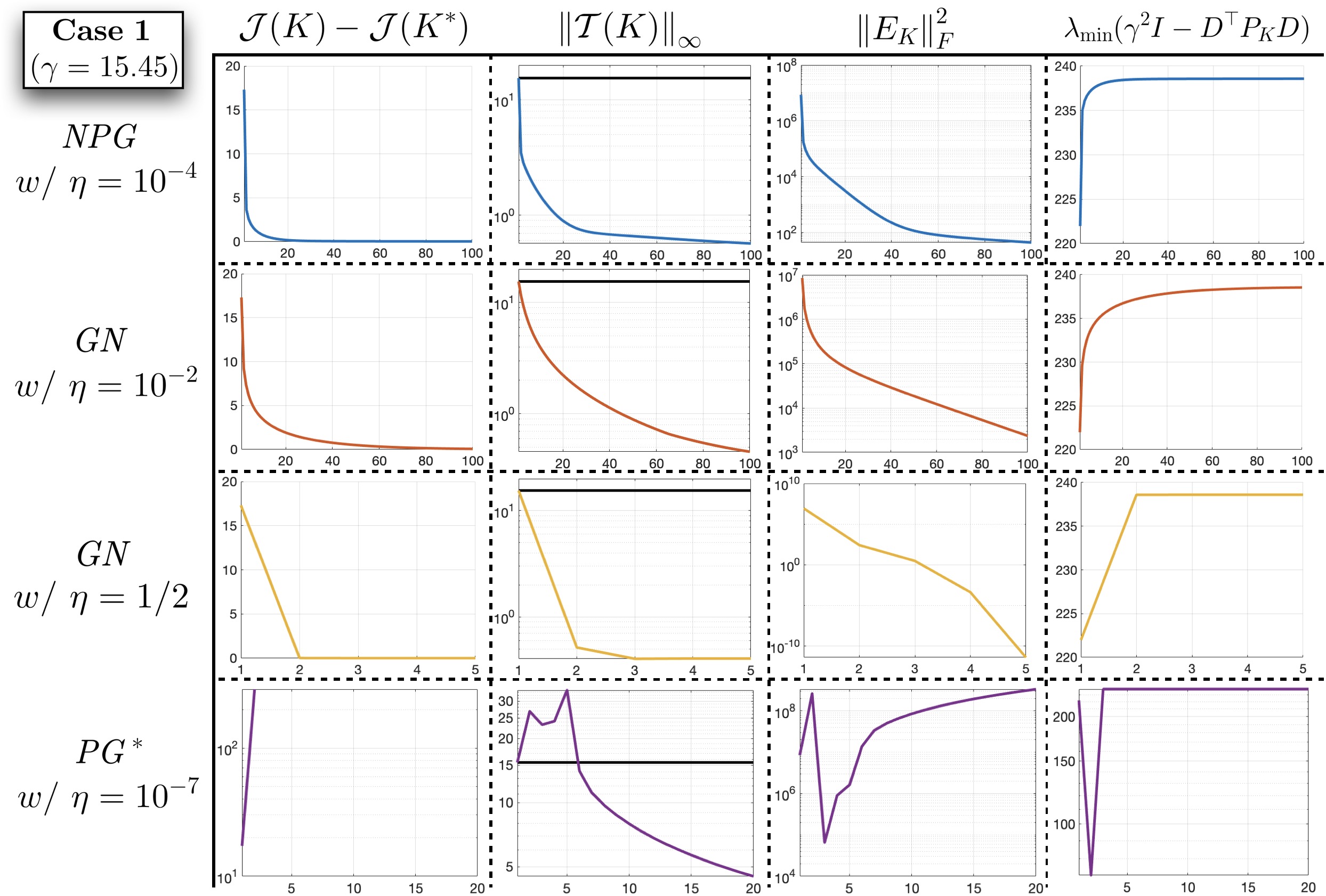}
	\caption{Convergence of the gradient norm square $\|E_{K}\|_F^2$ and the objective $\{\cJ(K)\}$, behaviors of the $\cH_\infty$-norm $\|\cT(K)\|_\infty$,  and the smallest eigenvalue of $\gamma^2 I-D^\top P_K D$, for PG, NPG, and Gauss-Newtons with stepsizes $1\times 10^{-7}$, $1\times 10^{-4}$, $0.01$, and $0.5$, respectively.}
	\label{fig:conv_main_result_Case_2}
\end{figure*}

\begin{figure}[!t]
\centering
\includegraphics[width = 1.0\textwidth]{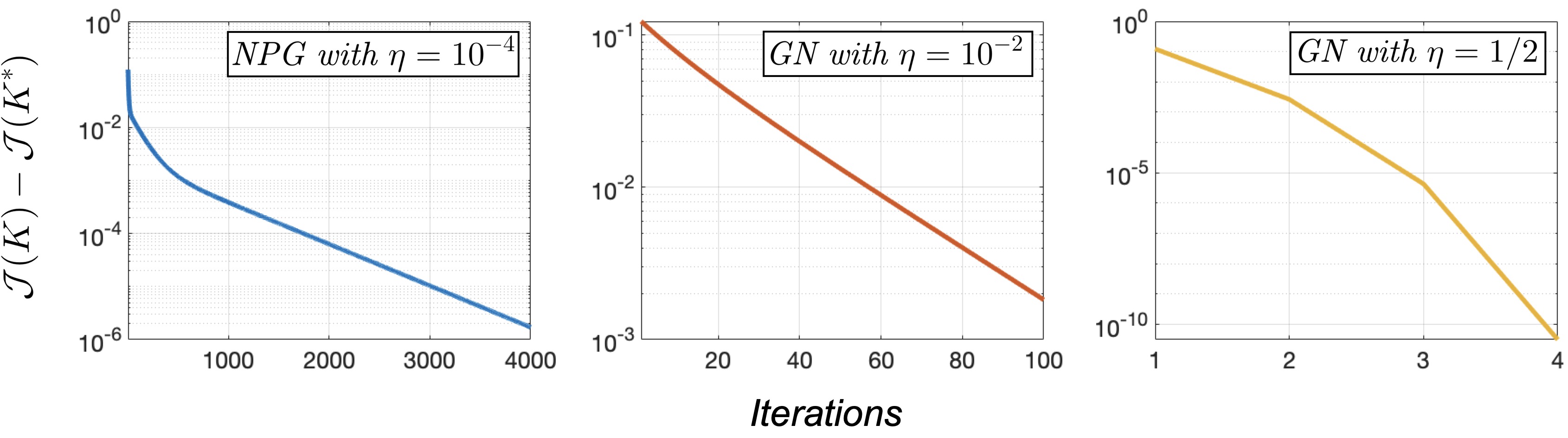}
\caption{Linear, linear and super-linear convergence rates for the NPG update with $\eta = 10^{-4}$, the GN update with $\eta = 10^{-2}$, and the GN update with $\eta = 0.5$, respectively.}
\label{fig:local}	
\end{figure}

We first consider the following example, denoted by \textbf{Case 1}, whose parameters are:
\$
  &A=\left[\begin{matrix}
  	1 & 0 & -10\\
  	-1 & 1 & 0
  	\\
  	0 & 0 & 1
  \end{matrix}
  \right],\quad 
  B=\left[\begin{matrix}
1 & -10 & 0\\
0 & 1 & 0\\
-1 & 0 & 1
\end{matrix}
\right],\quad C^\top C=\left[\begin{matrix}
2 & -1 & 0\\
-1 & 2 & -1\\
0 & -1 & 2
\end{matrix}
\right],\quad E^\top E=\left[\begin{matrix}
5 & -3 & 0\\
-3 & 5 & -2\\
0 & -2 & 5
\end{matrix}
\right],
 \$
and $DD^\top= I$. Note that all  matrices $C^\top C$, $E^\top E$, and $DD^\top$ are positive definite. 
This $DD^\top$ satisfies both the controllability assumption in Proposition  \ref{coro:opt_control_form_discrete}, and the assumption $DD^\top>0$ in Theorem \ref{theorem:local_exact_conv}. We first randomly generate $K_0$ with each element uniformly generated from $[-0.25,0.25]$, such that $K_0$ is stabilizing (i.e., $\rho(A-BK_0)<1$). Then, the  $\cH_\infty$-norm $\|\cT(K_0)\|_\infty$ under $K_0$ is calculated. The value of $\gamma$ is then chosen as $1.00001\cdot \|\cT(K_0)\|_\infty$, making sure that $K_0\in\cK$. We then perform all three algorithms \eqref{eq:exact_pg}-\eqref{eq:exact_gn} in \S\ref{sec:PO_alg} on the above problem setting, and illustrate the convergence of both the gradient norm square $\|E_{K}\|_F^2$, and the objective difference $\{\cJ(K_n)-\cJ(K^*)\}$. The stepsizes $\eta$ for the PG, NPG, and Gauss-Newton updates  are $1\times 10^{-7}$, $1\times 10^{-4}$, and $0.01$, respectively. We have also used $\eta=1/2$ for the Gauss-Newton update. 
 
% \textcolor{red}{THE FOLLOWING TEXTS IN CYAN ARE UNCHANGED.}
%\textcolor{cyan}
{As shown in Figure \ref{fig:conv_main_result_Case_2}, for both performance criteria, all four update rules converge successfully. At the beginning of the iterations, NPG and Gauss-Newton with $\eta=10^{-2}$ indeed yield sublinear convergence of the gradient norm square; as the iterations  proceed, linear convergence rate appears. Moreover, for Gauss-Newton with $\eta=1/2$, super-linear convergence rate has also been observed. These observations corroborate our theory in both Theorems \ref{theorem:global_exact_conv} and \ref{theorem:local_exact_conv}. 

Moreover, we have also illustrated the behaviors of the $\cH_\infty$-norm $\|\cT(K)\|_\infty$ and the smallest eigenvalue of $\gamma^2 I-D^\top P_K D$, denoted by $\lambda_{\min}(\gamma^2 I-D^\top P_K D)$, in Figure \ref{fig:conv_main_result_Case_2}. It is seen that along the iterations, with the stepsizes that guarantee convergence, the $\cH_\infty$-norm is below the bound $\gamma=15.45$ for all four update rules, which validates the implicit regularization result we have in Theorem \ref{thm:stability_update}. As another evidence for implicit regularization in accordance to Lemma \ref{lemma:discrete_bounded_real_lemma}, it is shown  that the matrix $\gamma^2 I-D^\top P_K D>0$ along iterations. 
%For \textbf{Case 2}, the stepsizes for PG, NPG, and Gauss-Newton are $1\times 10^{-8}$, $5\times 10^{-4}$, and $0.1$, respectively. We have also tested Gauss-Newton with $\eta=1/2$. 
%Note that the stepsizes for \textbf{Case 2} are all smaller than those for \textbf{Case 1}. 

Notice that the initialization $K_0$ is very close to the boundary, as $\gamma=1.00001\cdot \|\cT(K_0)\|_\infty$. It is shown that the vanilla PG update, even though with infinitesimal stepsize ($10^{-7}$), still violates the $\cH_\infty$-norm constraint and fails to converge. We do observe, however, in several other numerical examples, that vanilla PG update converges successfully.  It is thus not clear whether there exists a \emph{constant} stepsize choice for the \emph{global convergence} and \emph{robustness perservation} of the vanilla PG iterates, which is left for future investigation.  

%the  stepsize for PG has been chosen infinitesimal ($10^{-8}$), compared to other stepsizes, in order to prevent the iterates from going beyond $\cK$. As shown in Figure \ref{fig:conv_main_result_Case_2}, all three update rules converge to the global optimal solution, although the convergence of PG is not monotone. We note that Gauss-Newton with $\eta=1/2$ still enjoys the super-linear rate, and is not plotted in Figure \ref{fig:conv_main_result_Case_2}, so that other curves (with much slower convergence rates) can be seen clearly. Compared to Figure \ref{fig:conv_main_result_Case_1}, the sublinear convergence rate is more obvious in Figure \ref{fig:conv_main_result_Case_2}. However, although such an infinitesimal stepsize enables convergence for PG, as shown in Figure \ref{fig:conv_aux_result_Case_2} (a), the $\cH_\infty$-norm constraint is not guaranteed to be satisfied during iterations, namely, the property of implicit regularization is not ensured for PG. Deciding whether there exists a constant stepsize choice for PG that preserves the $\cH_\infty$-norm is still open, and is left as our future work. }

To further verify our local convergence rates, we have also initialized our algorithms by randomly searching over $\RR^{3\times 3}$ to find a $K_0\in\cK$ such that $\|K_0-K^*\|_F \leq 0.3$. The convergence patterns are presented in Figure \ref{fig:local}, which clearly demonstrates the faster local rates.

\subsection{Escaping  Suboptimal   Stationary Points}

\begin{figure}[!t]
\centering
\includegraphics[width = 1.0\textwidth]{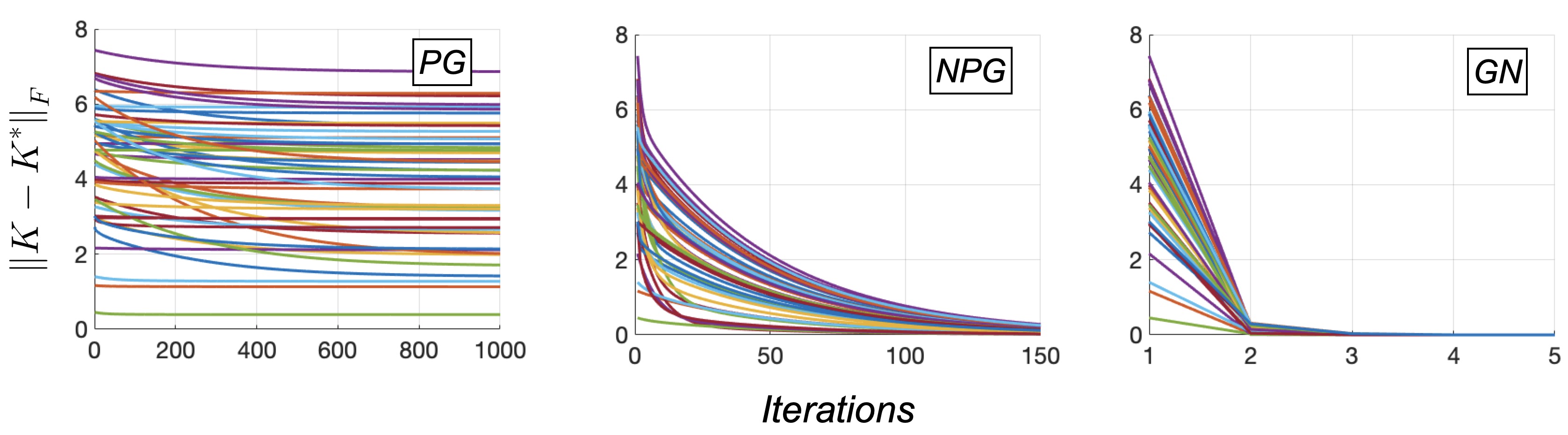}
\caption{Convergence of the NPG and GN updates to $K^*$ , when there exists an infinite number of stationary points (i.e. $K$ s.t. $\nabla\cJ(K) = 0$). The stepsizes for the PG, NPG, and GN updates are $10^{-3}$, $10^{-2}$, and $0.5$, respectively. }
\label{fig:star}	
\end{figure}

We also investigate the setting where the controllability assumption in Proposition  \ref{coro:opt_control_form_discrete}, and the assumption $DD^\top>0$ in Theorem \ref{theorem:local_exact_conv} do not hold. In this case, there might exist multiple stationary points, many of which are suboptimal. 
%Validating Theorem \ref{theorem:global_exact_conv}, we show next that the PNG and GN  updates can avoid these suboptimal stationary points, and converge to the global optimum successfully. In start contrast, the vanilla PG can easily get stuck at these suboptimal stationary points, depending on its initialization. 

Specifically, consider the following problem  parameter, which is denoted by \textbf{Case 2}:
\begin{align*}
	A = \begin{bmatrix} 2 & 0\\ 0 & 0\end{bmatrix}, \quad
	B = D = \begin{bmatrix} 1 & 0\\ 0 & 0\end{bmatrix}, \quad 
	C = \begin{bmatrix} 0 & 0\\ 0 & 0 \\ 1 & 2\end{bmatrix}, \quad 
	E =\begin{bmatrix} 1 & 0\\ 0 & 1 \\ 0 & 0\end{bmatrix}, \quad
	\begin{bmatrix} C^{\top} \\ E^{\top} \end{bmatrix}\cdot\begin{bmatrix} C & E \end{bmatrix} =  \begin{bmatrix} Q & \bm{0}_{2\times 2} \\ \bm{0}_{2\times 2} & R\end{bmatrix}.
\end{align*}
Note that the system is open-loop unstable, as $\rho(A)>1$. 
We choose $\gamma = 10$. Then, one can verify that the above mixed design problem admits an optimum  $K^* =(R+B^\top \tP_{K^*} B)^{-1}B^\top \tP_{K^*} A= \begin{bmatrix}1.6186 & 0 \\ 0 & 0\end{bmatrix}$. Moreover, there exist an infinite number of stationary points, which share the form of $K=\begin{bmatrix}1.6186 & 0 \\ 0 & c\end{bmatrix}$ for any  $c\in \RR$, and make $\nabla\cJ(K) = 0$. Despite this, following Theorem \ref{theorem:global_exact_conv}, the NPG and GN updates provably converge to $K^*$, which automatically escape other suboptimal stationary points. We numerically evaluate the convergence to $K^*$ in Figure \ref{fig:star}, for both the NPG and GN updates. For each of the 50 trails, we fix a random seed and initialize the algorithm by randomly searching a $K_0 \in \RR^{2\times 2}$ that satisfies $K_0 \in \cK$. It can be observed that two PG methods converge to $K^*$ in all trails. In start contrast, the vanilla PG update can easily get stuck at these suboptimal stationary points, depending on its initialization.

This can be understood as another  meaning of \emph{implicit regularization}: for this specific nonconvex problem, two certain search directions automatically bias the iterates to avoid bad local minima, and always towards  the global optimal one.

\subsection{Comparison with Existing $\cH_2/\cH_\infty$ Control Solvers}

To better justify the superiority of our PO methods, we  numerically compare their convergence properties with other numerical $\cH_2/\cH_\infty$ mixed control design packages, including the HIFOO method \citep{arzelier2011h2} and the \texttt{h2hinfsyn} method implemented in Matlab \citep{mahmoud1996h}, which is based on linear matrix inequalities.  Note that these full-fledged packages can only handle continuous-time settings. To make the comparison fair, we also implement our PO methods for the continuous-time settings as studied  in \S\ref{sec:aux_cont_res}. 

We mainly compare them in terms of: 1) the $\cH_2$ and $\cH_\infty$ norms of the controller that the algorithms converge to; 2) the \emph{computation complexity}  (runtime); 3) the \emph{$\cH_\infty$-norm constraint violation}. We will validate  that our PO methods indeed outperform HIFOO in these aspects. The larger-scale the dynamical system is, the more pronounced our advantages are, with provable robustness preserving guarantees.

\vspace{7pt}
\noindent{\textbf{Simulation Setup.}}~~ All the experiments are executed on a MacBook Pro 2019 with a 2.8 GHz Quad-Core Intel Core i7 processor with Matlab R2020b. The device also has a 16GB 2133MHz LPDDR3 memory and an Intel Iris Plus 655 Graphics. To make the runtime comparison fair (or even in favor of HIFOO), we set the following parameters of HIFOO (version 3.501 with Hanso version 2.01): \texttt{options.fast $=1$} for using a \emph{fast} optimization method; \texttt{options.prtlevel $=0$} to suppress unnecessary printing statements. Other parameters of HIFOO are set to be default. {For Matlab's \texttt{h2hinfsyn} function, \texttt{tol} has been set to $10^{-6}$. }For our PO methods, we set the stepsizes of the NPG and GN updates to be $1/(2\|R\|)$ and $1/2$, respectively, and solve the following mixed design problems \textbf{Cases 3-6}  {until $\cJ(K) - \cJ(K^*) < 10^{-6}$}. 
% as follows, . 
%to find a $K \in \cK$ such that $\cJ(K) - \cJ(K^*) \leq 1e^{-6}$. 

%of the two proposed policy gradient methods with that of the open-sourced package HIFOO [Arzelier et al., 2011], for solving the mixed $\cH_2/\cH_{\infty}$ design control problems. Due to that HIFOO supports only the continuous-time setting, all the experiments are carried out in the continuous-time, while our proposed policy gradient methods easily extend to discrete-time setting.  

%\section{A Simple Example} \label{sec:simple}  
\vspace{7pt}
\noindent\textbf{A Simple Example.}~~
We first consider a simple setting, denoted by \textbf{Case 3}. The time-invariant system dynamics are characterized by $\dot{x} = Ax+Bu+Dw$, $z = Cx+Eu$, where
\begin{align*}
	&A = \begin{bmatrix} 1 & 0 & -10\\ -1 &1 &0\\ 0 &0 &1\end{bmatrix}, \quad
	B = \begin{bmatrix} 1 & -10 & 0\\ 0 & 1 &0\\ -1 &0 &1\end{bmatrix}, \quad 
	C = \begin{bmatrix} 0 & 0 & 0\\ 0 &0 &0\\ 0 &0 &0 \\ 1 & 0 &2\end{bmatrix}, \quad \\
	&\qquad\quad~~ D = \begin{bmatrix} 0.5 & 0 & 0\\ 0 & 0.5 &0\\ 0 &0 &0.5\end{bmatrix}, \quad
	E = \begin{bmatrix} 1 & 0 & 0\\ 0 &1 &0\\ 0 &0 &1 \\ 0 & 0 & 0\end{bmatrix}.
\end{align*}
One can verify that $E^{\top}[C \ E] = [{0}, \ {I}]$, satisfying our assumption. Then, we solve
\$
	\min_{K}~ \cJ(K) = \tr(P_KDD^{\top}), \qquad s.t. \quad K \in \cK. 
\$
%where we use Matlab's \texttt{icare} function to compute $P_K$ for a given $K \in\cK$. 
The simulations are run over 100 trails with the random seed being fixed at $1, \cdots, 100$, respectively.  The optimal (minimax) disturbance attenuation level of \textbf{Case 3} is $\gamma^* \approx 0.53$, as computed/verified both by Matlab's \texttt{hinfsyn} function and HIFOO's \texttt{hifoo(P, 'h')} function. We summarize the following interesting findings based on the  {comparisons between HIFOO and our PO methods} in Table \ref{table1}: 
\begin{enumerate}
	\item \textit{(PO methods achieve  lower $\cH_2$ and $\cH_\infty$ norms, faster)}. When $\gamma = 5$, both HIFOO and PO methods preserve the $\|\cT(K)\|_{\infty} < \gamma$ constraint during the optimization process, and output some convergent gain matrix $K$. However, HIFOO converges to a \emph{local} minimum  due to that it is directly optimizing the $\cH_2$ norm of the closed-loop transfer function, and the landscape for such an optimization problem is unclear. In contrast, our PO  methods, by definition,  optimize an \emph{upper bound} of the $\cH_2$-norm, and as proved theoretically, converge  to the \emph{global} minimum of the problem. The solution yields lower values for  $\|\cT(K)\|_{2}$ and $\|\cT(K)\|_{\infty}$ compared to those of the HIFOO  output. Indeed, this shows that minimizing the $\cH_2$-norm upper bound as in \cite{bernstein1989lqg,mustafa1991lqg} and our paper can obtain reasonably  good solutions. 	
	 More importantly, the average runtimes of our PG methods average over 100 trails are around $5.93\times$ faster than HIFOO. 
	\item \textit{(PO methods always preserve $\cH_\infty$-norm constraint)}. When $\gamma = 3$ (which is still far from $\gamma^* \approx 0.53$), our   methods consistently preserve the $\|\cT(K)\|_{\infty} < \gamma$ constraint during the optimization process, validating our theoretical findings. However, the  HIFOO iterates have reached $\|\cT(K)\|_{\infty} = 3.4353>3$ along the way.  
%	\item In safety-critical control systems with model uncertainty, such failure to preserve robustness can cause catastrophic effects, e.g., destabilizing the system in face of disturbances. 
	Further, PO methods also find solutions that have  lower  $\|\cT(K)\|_{2}$ and $\|\cT(K)\|_{\infty}$ norms, with a $5.85\times$ faster runtime.
	\item \textit{(Smaller $\gamma$ leads to worse performance for HIFOO)}. When $\gamma = 1$, our PO methods still preserve the $\|\cT(K)\|_{\infty} < \gamma$ constraint during the optimization process. 
%	\item , though requiring more time to generate a feasible initial point.  
	Our PO methods also converge to the optimum point with both small $\cH_2$ and $\cH_\infty$ norms, while HIFOO's performance is degraded much more. We remark that in this case, the runtime of our methods is longer than that of HIFOO, however, the time was mostly consumed in finding the initialization that is robustly stable, by randomly generating $K_0$ in a certain region. This becomes harder to find for a smaller $\gamma$. Our simple  initialization method takes more than 90\% of the time, which might be less efficient than the advanced initialization technique used in HIFOO.     
\end{enumerate}
\begin{table}[t]
\small
\centering
\begin{tabular}{|c|c|c|c|c|c|c|}
\hline
\textbf{Case 3 w/ {$\gamma = 5$}} & \textit{HIFOO} & \textit{NPG} & \textit{GN} & \textit{$\|\cT(K)\|_2$ Diff.} & \textit{$\|\cT(K)\|_{\infty}$ Diff.} & \textit{\textcolor{red}{Speedup}}\\ \hline
\textit{Runtime} & $0.4569s$        & $0.0768s$      & $0.0779s$ & / & / & $\textcolor{red}{\sim 5.93\times}$    \\ \hline
\textit{$\|\cT(K)\|_2$ reached} & $1.1149$        & $0.9811$      & $0.9811$ &  $0.1338$ & / &  / \\ \hline
\textit{$\|\cT(K)\|_{\infty}$ reached} & $1.3372$        & $1.0124$      & $1.0124$ & / & $ 0.3248$ &  /  \\ \hline
\textit{$\|\cT(K)\|_{\infty}$-constraint violated} & $0$        & $0$      & $0$ & / & 0 &  / \\ \hline
\end{tabular}

\begin{tabular}{|c|c|c|c|c|c|c|}
\hline
\textbf{Case 3 w/ {$\gamma = 3$}} & \textit{HIFOO} & \textit{NPG} & \textit{GN} & \textit{$\|\cT(K)\|_2$ Diff.} & \textit{$\|\cT(K)\|_{\infty}$ Diff.} & \textit{\textcolor{red}{Speedup}}\\ \hline
\textit{Runtime} & $0.6727s$        & $0.1198s$      & $0.1130s$ & / & / & $\textcolor{red}{\sim 5.85\times}$    \\ \hline
\textit{$\|\cT(K)\|_2$ reached} & $1.4310$        & $0.9813$      & $0.9813$ &  $0.4497$ & / &  / \\ \hline
\textit{$\|\cT(K)\|_{\infty}$ reached} & $1.8940$        & $0.9947$      & $0.9947$ & / & $0.8994$ &  /  \\ \hline
\textit{$\|\cT(K)\|_{\infty}$-constraint violated} & \textcolor{red}{$0.4353$}        & \textcolor{red}{$0$}      & \textcolor{red}{$0$} & / & \textcolor{red}{$0.4353$} &  / \\ \hline
\end{tabular}

\begin{tabular}{|c|c|c|c|c|c|c|}
\hline
\textbf{Case 3  w/ {$\gamma = 1$}} & \textit{HIFOO} & \textit{NPG} & \textit{GN} & \textit{$\|\cT(K)\|_2$ Diff.} & \textit{$\|\cT(K)\|_{\infty}$ Diff.} & \textit{\textcolor{red}{Speedup}}\\ \hline
\textit{Runtime} & $1.7885s$        & $4.7219s$      & $4.8845s$ & / & / & $\textcolor{red}{\sim 0.37\times}$    \\ \hline
\textit{$\|\cT(K)\|_2$ reached} & $2.7237$        & $1.0038$      & $1.0038$ &  $1.7199$ & / &  / \\ \hline
\textit{$\|\cT(K)\|_{\infty}$ reached} & $2.8695$        & $0.8143$      & $0.8143$ & / & $2.0552$ &  /  \\ \hline
\textit{$\|\cT(K)\|_{\infty}$-constraint violated} & \textcolor{red}{$1.8695$}        & \textcolor{red}{$0$}      & \textcolor{red}{$0$} & / & \textcolor{red}{$1.8695$} &  / \\ \hline
\end{tabular}
\normalsize

%\begin{tabular}{|c|c|c|c|}
%\hline
%\textbf{Case 1 w/ random $K_0$} & \textit{HIFOO} & \textit{NPG} & \textit{GN} \\ \hline
%\textit{Average Runtime}       & 0.8925s        & 0.0091s      & 0.0090s     \\ \hline
%\end{tabular} 
\caption{Comparison average over 100 trails between HIFOO and two proposed PO methods, for solving the mixed design \textbf{Case 3}. All three methods initialize $K_0$ on their own. For two PO methods, a $K_0 \in \cK$ is found by randomly search over $[-1,1]^{3\times 3}$, which takes up $>90\%$ of the total runtime. In contrast, HIFOO uses an in-house method to find initial points. \emph{$\|\cT(K)\|_2$ Diff.} and \emph{$\|\cT(K)\|_\infty$ Diff.} represent the difference of $\cH_2$ and $\cH_\infty$ norms achieved by HIFOO and our PO methods (which are identical for NPG and GN, as our methods have guarantees for finding the global optimum of our mixed design problem).} 
\label{table1}
\end{table}

Regarding the comparison with the \texttt{h2hinfsyn} function, we 
%@article{mahmoud1996h,
%  title={$H_{\infty}$ design with pole placement constraints: an LMI approach},
%  author={Mahmoud, Chilali and Pascal, Gahinet},
%  journal={IEEE Transactions on Automatic Control},
%  volume={41},
%  number={3},
%  pages={358--367},
%  year={1996}
%}
%\textcolor{brown}{
present the results in Table \ref{matlab_comparison}. 
%Note that  the computation time of our PO methods do not include the time for finding an initial point. 
It is shown that in this simple $3\times 3$ problem, \texttt{h2hinfsyn} and our PG methods converge to nearly the same solution, while our computation time is around $4\times$ faster. Next, we will show that $\texttt{h2hinfsyn}$ scales poorly with respect to the problem dimensions, while our PO methods converge efficiently in high-dimensional problems.   }

\begin{table}[t]
\small
\centering
\begin{tabular}{|c|c|c|c|c|c|c|}
\hline
\textbf{Case 3 w/ {$\gamma = 5$}} & \textit{Matlab} & \textit{NPG} & \textit{GN}  & \textit{\textcolor{red}{Speedup}}\\ \hline
\textit{Runtime} & $0.0337s$        & $0.0088s$      & $0.0086s$ & $\textcolor{red}{\sim 3.87\times}$    \\ \hline
\textit{$\|\cT(K)\|_2$ reached} & $0.9811$        & $0.9811$      & $0.9811$ &  n.a. \\ \hline
\textit{$\|\cT(K)\|_{\infty}$ reached} & $1.0126$        & $1.0124$      & $1.0124$ &  n.a.  \\ \hline
\end{tabular}

\begin{tabular}{|c|c|c|c|c|c|c|}
\hline
\textbf{Case 3 w/ {$\gamma = 3$}} & \textit{Matlab} & \textit{NPG} & \textit{GN} & \textit{\textcolor{red}{Speedup}}\\ \hline
\textit{Runtime} & $0.0334s$        & $0.0077s$      & $0.0081s$ & $\textcolor{red}{\sim 4.23\times}$    \\ \hline
\textit{$\|\cT(K)\|_2$ reached} & $0.9813$        & $0.9813$      & $0.9813$ &  n.a. \\ \hline
\textit{$\|\cT(K)\|_{\infty}$ reached} & $0.9949$        & $0.9947$      & $0.9947$ &  n.a.  \\ \hline
\end{tabular}

\begin{tabular}{|c|c|c|c|c|c|c|}
\hline
\textbf{Case 3 w/ {$\gamma = 1$}} & \textit{Matlab} & \textit{NPG} & \textit{GN}  & \textit{\textcolor{red}{Speedup}}\\ \hline
\textit{Runtime} & $0.0283s$        & $0.0066s$      & $0.0063s$ & $\textcolor{red}{\sim 4.35\times}$    \\ \hline
\textit{$\|\cT(K)\|_2$ reached} & $1.0037$        & $1.0038$      & $1.0038$ &  n.a. \\ \hline
\textit{$\|\cT(K)\|_{\infty}$ reached} & $0.8145$        & $0.8143$      & $0.8143$ &  n.a.  \\ \hline
\end{tabular}
\normalsize
%\begin{tabular}{|c|c|c|c|}
%\hline
%\textbf{Case 1 w/ random $K_0$} & \textit{HIFOO} & \textit{NPG} & \textit{GN} \\ \hline
%\textit{Average Runtime}       & 0.8925s        & 0.0091s      & 0.0090s     \\ \hline
%\end{tabular}
\caption{{Comparison average over 100 trails between Matlab's \texttt{h2hinfsyn} and two proposed PO methods, for solving the mixed design \textbf{Case 3}. For the two PO methods, a $K_0 \in \cK$ is found by randomly search over $[-1,1]^{3\times 3}$, and the computation time for finding such an initial point is not taken into account for fair comparison. In contrast, \texttt{h2hinfsyn} implements a LMI-based synthesis procedure \citep{mahmoud1996h}.}}
\label{matlab_comparison}
\end{table}

\vspace{7pt}
\noindent\textbf{More Challenging Cases.}~~
We test some more challenging cases with higher dimensions to further  demonstrate the efficiency of our PO methods. In \textbf{Cases 4-6}, the dimensions of the control gain matrices are $15 \times 15, 60\times 60, 90\times 90$, respectively, corresponding to the number of decision variables being $225$, $3600$, $8100$, respectively.   Problem parameters are too long to enumerate here, and are provided at \href{https://www.dropbox.com/sh/mfrenjttwkidmbv/AACAWnmjL4NWgDa76Atc6DUya?dl=0}{\textbf{here}}, together with all the code and data.   

%In high-dimensional systems, finding a desired initialization point is difficult for both HIFOO and our PG methods. Thus, we designed $A$ to be a Hurwitz matrix and initialize $K_0= \bm{0}$. 
The simulations are run over 10 trails with fixed random seeds. Both HIFOO{, \texttt{h2hinfsyn}, } and our PO methods converge to almost the same control gain matrices in these cases, without constraint violation. This again implies that minimizing the $\cH_2$-norm upper bound (instead of $\cH_2$-norm directly) can usually  achieve quite competitive  solutions.  Notably, our PO methods are around $8\times$, $47 \times$, $295\times$ faster than HIFOO, respectively, in \textbf{Cases 4-6}, as reported in Table \ref{table2}.  {Our PO methods are also much faster than \texttt{h2hinfsyn}, as it can hardly solve \textbf{Cases 5-6} and fails to return a solution even after very long runtime.}. This verifies that our PO algorithms indeed enjoy better scalability, and the higher the dimension is, the more pronounced our advantage is. These observations have justified that our PO methods are not only theoretically sound, but also numerically competitive. 

\begin{table}[!t]
\centering
\begin{tabular}{|c|l|l|l|l|c|}
\hline
\textbf{Average runtime} & \multicolumn{1}{c|}{\textit{HIFOO}} & \multicolumn{1}{c|}{\textit{Matlab}} &\multicolumn{1}{c|}{\textit{NPG}} & \multicolumn{1}{c|}{\textit{GN}} & \textit{\textcolor{red}{Speedup}} \\ \hline
\textit{\textbf{Case 4}}  & \multicolumn{1}{c|}{$0.3742s$}   &  \multicolumn{1}{c|}{$95.2663s$}   & \multicolumn{1}{c|}{$0.0481s$}      & \multicolumn{1}{c|}{$0.0420s$}     & \textcolor{red}{$\sim 8(>2117)\times$}        \\ \hline
\textit{\textbf{Case 5}}              &\multicolumn{1}{c|}{$18.4380s$}             &   \multicolumn{1}{c|}{fail,  $>7200s$}            & $0.3906s$                           & $0.3902s$                          & \textcolor{red}{$\sim 47(>18461)\times$}       \\ \hline
\textit{\textbf{Case 6}}            & \multicolumn{1}{c|}{$241.4416s$}         &    \multicolumn{1}{c|}{fail,  $>14400s$}              & $0.8167s$                           & $0.8103s$                          & \textcolor{red}{$\sim 295(>36922)\times$}      \\ \hline
\end{tabular}
\caption{Average runtime comparison over 10 trails between HIFOO{, Matlab's \texttt{h2hinfsyn} function,} and two proposed PO methods for solving \textbf{Cases 4-6}. The speedup times outside and inside the parenthesis denote the ones of our PO methods compared to HIFOO and Matlab, respectively.}
\label{table2}
\end{table}

\begin{table}[!t]
\centering
\begin{tabular}{|c|l|l|l|c|}
\hline
\textbf{Average $\|\cT(K)\|_2$ reached w/ $K_0 = \bm{0}$} & \multicolumn{1}{c|}{\textit{HIFOO}} & \multicolumn{1}{c|}{\textit{Matlab}} & \multicolumn{1}{c|}{\textit{NPG}} & \multicolumn{1}{c|}{\textit{GN}} \\ \hline
\textit{\textbf{Case 4} w/ $\|\cT(K_0)\|_2 = 2.3979$ and $\gamma = 10$}  & \multicolumn{1}{c|}{$0.4713$}        &\multicolumn{1}{c|}{$0.4713$} & \multicolumn{1}{c|}{$0.4713$}   & \multicolumn{1}{c|}{$0.4713$}    \\ \hline
\textit{\textbf{Case 5} w/ $\|\cT(K_0)\|_2 = 9.2195$ and $\gamma = 15$}           & \multicolumn{1}{c|}{$1.1239$}                            &\multicolumn{1}{c|}{fail} & \multicolumn{1}{c|}{$1.1239$}                           & \multicolumn{1}{c|}{$1.1239$}                              \\ \hline
\textit{\textbf{Case 6} w/ $\|\cT(K_0)\|_2 = 10.9716$ and $\gamma = 20$}           & \multicolumn{1}{c|}{$1.2178$}        &\multicolumn{1}{c|}{fail}                   & \multicolumn{1}{c|}{$1.2178$}                           & \multicolumn{1}{c|}{$1.2178$}                       \\ \hline
\end{tabular}
\caption{Average $\|\cT(K)\|_2$ reached over 10 trails between HIFOO{, Matlab's \texttt{h2hinfsyn} function,} and two proposed PO methods for solving  \textbf{Cases 4-6}.}
\end{table}

\begin{table}[!t]
\centering
\begin{tabular}{|c|l|l|l|l|c|}
\hline
\textbf{Average $\|\cT(K)\|_{\infty}$ reached w/ $K_0 = \bm{0}$} & \multicolumn{1}{c|}{\textit{HIFOO}} &  \multicolumn{1}{c|}{\textit{Matlab}}& \multicolumn{1}{c|}{\textit{NPG}} & \multicolumn{1}{c|}{\textit{GN}} \\ \hline
\textit{\textbf{Case 4} w/ $\|\cT(K_0)\|_{\infty} = 3.3912$ and $\gamma = 10$}  & \multicolumn{1}{c|}{$0.0935$}     &\multicolumn{1}{c|}{$0.0935$}   & \multicolumn{1}{c|}{$0.0935$}      & \multicolumn{1}{c|}{$0.0935$}        \\ \hline
\textit{\textbf{Case 5} w/ $\|\cT(K_0)\|_{\infty} = 13.0384$ and $\gamma =15$}           & \multicolumn{1}{c|}{$0.1376$}     &\multicolumn{1}{c|}{fail}                       & \multicolumn{1}{c|}{$0.1375$}                           & \multicolumn{1}{c|}{$0.1375$}                                \\ \hline
\textit{\textbf{Case 6} w/ $\|\cT(K_0)\|_{\infty} =  15.5161$ and $\gamma = 20$}           & \multicolumn{1}{c|}{$0.1357$}         &\multicolumn{1}{c|}{fail}                  & \multicolumn{1}{c|}{$0.1356$}                           & \multicolumn{1}{c|}{$0.1356$}                                 \\ \hline
\end{tabular}
\caption{Average $\|\cT(K)\|_{\infty}$ reached over 10 trails between HIFOO{, Matlab's \texttt{h2hinfsyn} function,} and two proposed PO methods for solving \textbf{Cases 4-6}.}
\end{table}

\section{Concluding Remarks}\label{sec:conclusions}

In this paper, we have investigated the convergence theory of policy optimization methods for $\cH_2$ linear control with $\cH_\infty$-norm robustness guarantees. Viewed as a constrained nonconvex optimization, this problem was addressed by PO methods with provable convergence to the global optimal policy. More importantly, we showed that the proposed PO methods enjoy the implicit regularization property, despite the lack of  coercivity of the cost function. 
We expect the present work to serve as an initial step toward further understanding of RL algorithms on robust/risk-sensitive control tasks. We conclude this main part of the  paper with several ongoing/potential research directions.

\paragraph{Implicit regularization of other PO  methods:} It is of particular interests to investigate whether other PO methods  enjoy similar implicit regularization properties. One important example that has not been analyzed in this paper is the PG method. Notice that the model-free implementation of the PG  method, see update \eqref{eq:exact_pg}, does not require the connection between mixed design and zero-sum LQ games, as the gradient can be sampled via zeroth-order methods directly. Among other  examples are quasi-Newton methods with pre-conditioning matrices other than that in \eqref{eq:exact_gn}, accelerated PG using the idea from \cite{nesterov1983method}, and variance reduced PG methods \citep{papini2018stochastic,xu2019improved}. 

%Hence a rigorous analysis of such methods may lead to more efficient data-driven algorithms for LEQR and other risk-sensitive control problems.

\paragraph{Linear quadratic games:} Thanks to the connection discussed in \S\ref{subsec:connection_to_games},  our  LMI-based techniques for showing implicit regularization in Theorem \ref{thm:stability_update} may be of  independent interest to improve the convergence of nested policy gradient methods in \cite{zhang2019policy}, and even simultaneously-moving policy-gradient methods, for solving zero-sum LQ games using PO methods. This will place PO  methods for multi-agent RL (MARL) under a more solid theoretical footing, as LQ games have served as a significant benchmark for MARL \citep{chasnov2019convergence,mazumdar2019policy}. 
Rigorous analysis for this setting have been partially addressed in our ongoing work, and in a more recent work \cite{bu2019global}. 

\paragraph{Model-based v.s. model-free methods for robust control:} There is an increasing literature in \emph{model-based} learning-based control with robustness concerns \citep{aswani2013provably,berkenkamp2015safe,berkenkamp2017safe,dean2017sample,dean2019safely}.  
On the other hand, our work serves as an intermediate step toward  establishing the sample complexity of model-free PO methods for this setting.  
Hence, it is natural and interesting  to compare the data efficiency (sample complexity) and  computational scalability of the two lines of work. Note that such a comparison has been made in  \cite{tu2018gap} for LQR problems without addressing  the issue of robustness.

\paragraph{Beyond LTI systems and state-feedback controllers:}   It is possible to extend our  analysis to the  mixed $\cH_2/\cH_\infty$ control of other types of dynamical systems such as  periodic systems \citep{bittanti1996analysis}, Markov jump linear systems \citep{costa2006discrete}, and switching systems \citep{liberzon2003switching}. These more general system models are widely adopted in control applications, and extensions to these cases will significantly expand  the utility of our theory.  
%Intuitively, the PO  landscape for these  cases are similar to that considered here, and certain PO methods should have similar implicit regularization property. This needs to be verified rigorously. 
On the other hand, it is interesting while challenging to study PO for \emph{output-feedback}  mixed design, where a dynamic controller parameterized by $(A_K,B_K,C_K,D_K)$ is synthesized \citep{apkarian2008mixed}. This way, the PO landscape depends on the order of the parameterization, making the analysis  more involved. 
%A rigorous study is worth pursuing in the future.

\paragraph{PO landscape and algorithms for $\cH_\infty$ control synthesis:}  Our algorithms   are based on the condition that   an initial policy satisfying the specified $\cH_\infty$-norm constraint is available. To efficiently find such an initialization,  it is natural  to study the PO landscape of $\cH_\infty$ control synthesis \citep{doyle1988state,gahinet1994linear,apkarian2006nonsmooth}, where the goal is to find the controller that not only satisfies certain $\cH_\infty$-norm bound, but also minimizes it. It seems that the cost function for $\cH_\infty$ control is still coercive. However, the main challenge of  PO for $\cH_\infty$ control is that the cost function is non-smooth \citep{apkarian2006nonsmooth}, which necessitates the use of subgradient methods. The LMI arguments developed here may shed new lights on the convergence analysis of these methods.

%In this paper, we investigated the convergence theory of policy optimization methods for linear control with $\cH_\infty$-norm robustness guarantees. Viewed as a constrained nonconvex optimization problem, this problem was addressed by PO methods with provably convergence to the global optimal policy. More importantly, we showed that the proposed PO methods enjoy the implicit regularization property, in absence of the coercivity of the cost function. 
%Thanks to the connection discussed in \S\ref{subsec:connection_to_games},  the techniques developed for showing implicit regularization may be of  independent interest to improve the convergence of nested-gradient methods \cite{zhang2019policy}, and even simultaneously-moving policy-gradient methods, for solving zero-sum LQ games using RL  \cite{bacsar1995h}, in our ongoing work. Another future  direction is to study the policy optimization landscape of $\cH_\infty$ control synthesis \cite{doyle1988state,gahinet1994linear,apkarian2006nonsmooth}, where the goal is to find the controller that not only satisfies certain $\cH_\infty$-norm bound, but also minimizes it. 

%\begin{itemize}
%\item Implicit bias of other optimization methods
%\item Model-based v.s. model free for robust control
%\item Optimization landscape of $\mathcal{H}_\infty$ control
%\item  Future directions: prima-dual for LQ-game, using this implicit regularization idea; minimization of Hinf norm directly, though non-smooth.
%\end{itemize}

\section*{Acknowledgements} 
K. Zhang and T. Ba\c{s}ar were supported  in part by the US Army Research Laboratory (ARL) Cooperative Agreement W911NF-17-2-0196, and in part by the Office of Naval Research (ONR) MURI Grant N00014-16-1-2710.  The authors would like to thank Peter Seiler, Geir  Dullerud, Na Li, and Mihailo Jovanovic for the valuable  comments and feedback on our manuscript, as well as the helpful   suggestions from the  anonymous reviewers of L4DC and SICON.  The authors would also like to thank Xiangyuan (Rocker) Zhang for helping with the simulations. 
 
%\newpage
\small 
\bibliographystyle{ims}
\bibliography{main}
\normalsize

\appendix
\clearpage

\section{Results for Continuous-Time Setting}\label{sec:aux_cont_res}

In this appendix, we   present the counterparts of the formulation and  results in the main part of the paper for the continuous-time mixed $\cH_2/\cH_{\infty}$ design problem. 

\subsection{Formulation}\label{subsec:cont_prob_form}

Consider the  linear dynamics
\$
\dot{x}=Ax+Bu+Dw,\qquad z=Cx+E u,
\$
where $x\in\RR^{m}$ is the state, $u\in\RR^d$ is the control,  $w\in\RR^n$ is the disturbance, $z\in\RR^l$ is the controlled output, and $A,B,C,D,E$ are the matrices of proper dimensions. 
%Consider the \emph{admissible} control policy $\mu_t$ to be a mapping from the history of state-action pairs till time $t$ to action $u_t$. 
The performance measure of the mixed $\cH_2/\cH_\infty$ design problem is usually some upper bound of the $\cH_2$-norm of the system \citep{khargonekar1991mixed}.  
With the  \emph{state-feedback} information structure, it has been shown in 
%for both continuous-time 
\cite{khargonekar1991mixed}
% and discrete-time \cite{kaminer1993mixed} settings 
 that, LTI state-feedback controller suffices to achieve the optimal performance measure. As a consequence, it suffices to consider only stationary control policies parametrized as  $u=-Kx$.  
Note that   Assumption \ref{assum:coeff_matrices} is also standard for continuous-time settings. Hence, the transfer function $\cT(K)$ from the disturbance  $w$ to the output $z$ also has the form of  \eqref{equ:mixed_design_transfer2}. The robustness of the designed controller  can thus be guaranteed by the  constraint on the $\cH_\infty$-norm; see definition in \eqref{equ:def_cont_Hinf_norm}. In particular, the constraint is $\|\cT(K)\|_{\infty}<\gamma$ for some $\gamma>0$. Define $\cK$ to be the feasible set  as 
\$\cK:=\big\{K\biggiven (A-BK)\, \mbox{being Hurwitz, and}\,\, \|\cT(K)\|_{\infty}<{\gamma}\big\}.
\$ 
For continuous-time setting, 
the   cost  $\cJ(K)$ usually only takes the form of  \eqref{equ:form_J1} \citep{mustafa1989relations,mustafa1991lqg}, with the $P_K$ replaced by the solution to the continuous-time Riccati equation  
\#
 (A-BK)^\top P_K+P_K(A-BK)+C^\top C+K^\top R K+\gamma^{-2} P_K  DD^\top P_K=0.\label{equ:cont_riccati} 
\#
\normalsize
%To be consistent, we also refer to  \eqref{equ:cont_riccati}  as the \emph{modified Bellman equation}. 
In sum, the continuous-time mixed  design can thus be formulated as 
\#\label{equ:def_mixed_formulation_cont}
\min_K \quad \cJ(K)=\tr(P_KDD^\top),\qquad s.t.\quad K\in\cK,
\#
a constrained nonconvex optimization problem.

%\issue{2019.08.21. }
%
%\issue{In the introduction of the ``bigger picture'' on mixed design, for both continuous \cite{khargonekar1991mixed} and discrete \cite{kaminer1993mixed} time, if we consider the optimal control design in the  ``state-feedback'' case, then the ``dynamic full information controller'' has the same optimal performance as just ``stationary state-feedback controllers''. If we also have access to the noise, then the ``dynamic full information controller'' performance is still achievable by ``stationary state-feedback controllers'' for ``continuous-time'', but, interestingly, not for the ``discrete-time'' \cite{kaminer1993mixed}. Anyways, it suffices to only consider ``stationary state-feedback controllers'' here, since we have no access to the noise/disturbance/exogenous input. Hence, we may want to connect to this discrete-time result in \cite{kaminer1993mixed}, to justify that the results for LEQR is ``expected'', since we only have state, no noise observations (write this as a remark).}
%\\

%For notational convenience, we define the feasible set of mixed $\cH_{2}/\cH_{\infty}$ control design  as
%\#\label{equ:define_cK}
%\cK:=\big\{K\biggiven \|\cT(K)\|_{\infty}<{\gamma}\big\}. 
%\# 

%\issue{WE INTRODUCE THE NONCONVEXITY of $\cK$ HERE TOO} 

\subsubsection{Bounded Real Lemma}

There also exists a continuous-time Bounded Real Lemma \citep{zhou1996robust,rantzer1996kalman} that relates the $\cH_\infty$-norm bound  to the solution of a Riccati equation and an LMI.

%Though the constraint \eqref{equ:define_cK} is concise, it is hard to enforce over $K$ in policy optimization, since the constraint is characterized  in the frequency domain. 
%Interestingly, by using a significant  result in robust control theory, i.e., \emph{Bounded Real Lemma} \cite{zhou1996robust,rantzer1996kalman}, constraint \eqref{equ:define_cK} can be related to the solution of a Riccati equation and a Riccati inequality. 
%We formally introduce the lemmas for both continuous- and discrete- time settings 
%as follows.

\begin{lemma}[Continuous-Time Bounded Real Lemma]\label{lemma:cont_bounded_real_lemma}
Consider the continuous-time  transfer function $\cT(K)$ defined in \eqref{equ:mixed_design_transfer2}, which is recalled here as
\$
\renewcommand\arraystretch{1.3}
\cT(K):=\mleft[
\begin{array}{c|c}
  A-BK & D \\
  \hline
  (C^\top C+K^\top R K)^{1/2}& 0
\end{array}
\mright].
\$
Suppose $K$ is stabilizing, i.e., $(A-BK)$ is Hurwitz, 
then,  the  following conditions are equivalent: 
%consider a discrete-time dynamical system with transfer function $\cT(K)$ defined as
%\#\label{equ:dynamic_sys}
%\renewcommand\arraystretch{1.3}
%\cT(K):=\mleft[
%\begin{array}{c|c}
%  A-BK & W^{1/2} \\
%  \hline
%  (Q+K^\top R K)^{1/2}& \bm{0} 
%\end{array}
%\mright]. 
%\#
%XXXXXX
\begin{itemize}
		\item $\|\cT(K)\|_{\infty}<{\gamma}$, which, due to that $(A-BK)$ is  Hurwitz, further implies that $K\in\cK$; 
%		\item i.e., $(A-BK)$ being Hurwitz and $\|\cT(K)\|_{\infty}<{\gamma}$; 
%		\issue{XXXXXXXX} 
		\item The Riccati equation   
%		\eqref{equ:def_PK}-\eqref{equ:def_tPK} 
		\eqref{equ:cont_riccati}
		admits  a unique stabilizing  solution $P_K\geq  0$  such that the matrix $A-BK+\gamma^{-2} DD^\top P_K$ is Hurwitz; 
		\item There exists some  {$P> 0$}, such that 
%\#\label{equ:cont_equiva_set_cK_cond_1}
%\mleft[
%\begin{array}{cc}
%  (A-BK)^\top P+P(A-BK)+Q+K^\top R K  & P \\
%  P& -\frac{1}{\gamma}W^{-1} 
%\end{array}
%\mright]<0,
%\#
%or equivalently, 
\#\label{equ:cont_equiva_set_cK_cond2}
(A-BK)^\top P+P(A-BK)+C^\top C+K^\top R K+\gamma^{-2} P  D D^\top P<0. 
\#  
	\end{itemize}
\end{lemma}

The three equivalent conditions in Lemma   \ref{lemma:cont_bounded_real_lemma} 
 will be frequently used in the  analysis. 
 Similarly, the unique stabilizing  solution to \eqref{equ:cont_riccati} for any $K\in\cK$, is also \emph{minimal}, if the pair   $(A-BK,D)$ is stabilizable, see 
\cite[Corollary $13.13$, page $339$] {zhou1996robust}, which is indeed the case since any $K\in\cK$ is stabilizing. Hence, it suffices to consider only stabilizing solution $P_K$ of the Riccati equation  \eqref{equ:cont_riccati} in this case.

\subsection{Landscape and Algorithms}\label{subsec:cont_alg}
Next we study the optimization landscape and policy-based algorithms for the continuous-time mixed $\cH_2/\cH_\infty$ design.

\subsubsection{Optimization Landscape}

As in the discrete-time setting, this problem is nonconvex, and enjoys no coercivity. Proofs of the following lemmas are provided in \S\ref{sec:proof_lemma:nonconvex} and \S\ref{sec:proof_lemma_mixed_design_no_coercivity}, respectively.  
 
\begin{lemma}[Nonconvexity of Continuous-Time Mixed $\cH_{2}/\cH_{\infty}$ Design]\label{lemma:nonconvex_Hinf_norm_set_cont}
 	The continuous-time mixed $\cH_{2}/\cH_{\infty}$ design  problem \eqref{equ:def_mixed_formulation_cont}  is nonconvex. 
% 	, for both continuous- and discrete-time settings. 
 \end{lemma}
 
 \begin{lemma}[No Coercivity of Continuous-Time  Mixed $\cH_{2}/\cH_{\infty}$ Design]\label{lemma:mixed_design_no_coercivity_cont}
	The cost function \eqref{equ:def_mixed_formulation_cont}  for continuous-time mixed $\cH_{2}/\cH_{\infty}$ design is not coercive. 
%	 Particularly, as $K\to \partial \cK$, where $\partial \cK$ is the boundary of the constraint set $\cK$, the cost $\cJ(K)$ does not necessarily  approach  infinity. 
\end{lemma}

 The following lemma, whose proof is deferred to 
% \S\ref{sec:proof_lemma_differentiability} and 
\S\ref{sec:proof_lemma_differentiability_policy_grad_ct}, 
% respectively,   
 establishes the differentiability and the policy gradient form of  $\cJ(K)$ given by \eqref{equ:def_mixed_formulation_cont}. 

\begin{lemma}\label{lemma:differentiability_policy_grad_ct}
	The objective defined in \eqref{equ:def_mixed_formulation_cont} is differentiable 
%	, and thus continuous, 
	in $K$ for any $K\in\cK$, and the policy gradient of $\cJ(K)$ with respect to $K$ has the following form
	\$
	\nabla \cJ(K)={2[RK-B^\top P_K] \Lambda_K}, 
	\$
	where $\Lambda_K\in \RR^{m\times m}$ is the solution to the Lyapunov equation
	\#\label{equ:Lambda_Lya_def}
	\Lambda_K(A-BK+\gamma^{-2}  DD^\top P_K)^\top+(A-BK+\gamma^{-2}  DD^\top P_K)\Lambda_K+DD^\top=0. 
	\#
%	with $\cK$ being defined in \eqref{equ:def_cK_cont}.
%	, i.e.,  the $K$ such that $\cH_{\infty}$-norm $\|\cT(K)\|_{\infty}<1/\sqrt{\gamma}$. 
\end{lemma}

%\begin{lemma}[Policy Gradient for Continuous-Time    Mixed Design]\label{lemma:policy_grad_ct}
%	For any $K\in\cK$, the gradient of $\cJ(K)$ w.r.t. $K$ has the following form
%	\$
%	\nabla \cJ(K)={2[RK-B^\top P_K] \Lambda_K}, 
%	\$
%	where $\Lambda_K\in \RR^{m\times m}$ is the solution to the Lyapunov equation
%	\#\label{equ:Lambda_Lya_def}
%	\Lambda_K(A-BK+\gamma^{-2}  DD^\top P_K)^\top+(A-BK+\gamma^{-2}  DD^\top P_K)\Lambda_K+DD^\top=0. 
%	\#
%\end{lemma}

Lemma \ref{lemma:differentiability_policy_grad_ct} further leads to the following proposition, which provides  the formula for the optimal controller for mixed $\cH_2/\cH_\infty$ control design. 

\begin{proposition}\label{coro:opt_control_form_cont}
	Suppose that the continuous-time mixed $\cH_2/\cH_\infty$ design admits a global optimal control gain solution $K^*\in\cK$, and $(A,C)$ is detectable, then one such  solution  has the form of $K^*=R^{-1}B^\top P_{K^*}$. Additionally, if the pair $\big(A-BK+\gamma^{-2}DD^\top P_{K},D\big)$ is controllable at some stationary point of $\cJ(K)$ such that $\nabla \cJ(K)=0$, then this is   the unique stationary point, and corresponds to the unique  global optimizer  $K^*$.  
%	Suppose that the continuous-time mixed $\cH_2/\cH_\infty$  design admits a control gain solution $K^*\in\cK$, and for any stationary point  $K\in\cK$ such that $\nabla \cJ(K)=0$, 
%	the pair $(A-BK+\gamma^{-2}DD^\top P_{K},D)$ is controllable.  Then, such a solution is unique, and has the form of 	$K^*=R^{-1}B^\top P_{K^*}$. 
\end{proposition}
%\begin{proof}
%By \cite[Lemma $3.18$ \emph{(iii)}]{zhou1996robust}, the solution to the Lyapunov equation \eqref{equ:Lambda_Lya_def} $\Lambda_K>0$, since $(A-BK+\gamma^{-2}DD^\top P_{K},D)$ is  controllable, or equivalently, $((A-BK+\gamma^{-2}DD^\top P_{K})^\top,D^\top)$ is observable. Thus,  the necessary optimality condition $\nabla \cJ(K)=0$ yields  that  $K^*=R^{-1}B^\top P_{K^*}$ is the unique stationary point,  which is thus the unique global optimizer. 
%\end{proof}

%Under this constraint, \issue{the stationary feedback control gain satisfying the following equation achieves the minimum value of $\cJ(K)$    (actually there is no reference I can find that explicitly gives this formula. One option is that we only say we are looking for some $K$ here. Then, by differentiability, we show that the only stationary point is the one below. Done. Thus, we put it later.). 

%Moreover, this is equivalently to require the following LMI hold, i.e., there exists a $P>0$ such that:
%\#\label{equ:LMI_cond_ct}
%\mleft[
%\begin{array}{cc}
%  (A-BK)^\top P+P(A-BK)+Q+K^\top R K  & P \\
%  P& -\frac{1}{\gamma}W^{-1} 
%\end{array}
%\mright]<0. 
%\# 
%The goal of the problem is to find the optimal $K$ such that 
%} 	
	
Proposition  \ref{coro:opt_control_form_cont}, proved in \S\ref{proof:coro_opt_control_form}, not only gives the form of one global optimal solution, but also  shows that under certain controllability conditions, the stationary point of $\cJ(K)$ is unique, and corresponds to that  global optimum. Note that the controllability condition is satisfied  if $DD^\top >0$.  
%for continuous-time LEQR, with $D=W^{1/2}>0$. 
Such a desired property motivates the development of first-order methods to solve this nonconvex optimization problem.

\subsubsection{Policy Optimization Algorithms}\label{subsec:PO_alg}

Consider three    policy-gradient  based algorithms as follows. 
\begin{flalign}
{\rm \textbf{Policy  Gradient:}}~\qquad\qquad~~~~~~~~ K'&=K-\eta \nabla\cJ(K)=K-2\eta [RK-B^\top P_K] \Lambda_K\label{eq:exact_pg_ct}\\
{\rm \textbf{Natural Policy Gradient:}}~~\qquad  K'&=K-\eta \nabla\cJ(K)\Lambda_K^{-1}=K-2\eta (RK-B^\top P_K) &\label{eq:exact_npg_ct}
 \\
 {\rm \textbf{Gauss-Newton:}}\qquad\qquad~~~\qquad K'&=K-\eta R^{-1}\nabla\cJ(K)\Lambda_K^{-1}=K-2\eta (K-R^{-1}B^\top P_K)
 &\label{eq:exact_gn_ct}
 \end{flalign}
with $\eta>0$ being the stepsize. The updates are designed to follow the updates for discrete-time settings in  \eqref{eq:exact_pg_ct}-\eqref{eq:exact_gn_ct}. 

\subsection{Theoretical Results}\label{subsec:cont_theory}

We  now   establish the convergence results on the   algorithms  proposed in \S\ref{subsec:PO_alg}. 

\subsubsection{Implicit Regularization}
 
As in  the discrete-time setting, we first show that both the natural PG update  \eqref{eq:exact_npg_ct} and the Gauss-Newton update  \eqref{eq:exact_gn_ct} enjoy the \emph{implicit regularization} property. 
%, i.e., if $K\in\cK$, there exists constant stepsize $\eta$ such that $K'$ also lies in $\cK$.   

%we can also show  that if the stepsize $\eta$ is chosen properly, the updated control gain $K'$ from   \eqref{eq:exact_npg_ct}-\eqref{eq:exact_gn_ct} remains  robustly stable.  
%In particular, we prove that if the closed-loop transfer function under control gain $K$ satisfies $\|\cT(K)\|_{\infty}<{\gamma}$, then $K'$ also satisfies that $\|\cT(K')\|_{\infty}<{\gamma}$.   

\begin{theorem}[Implicit Regularization for Continuous-Time Mixed Design]\label{thm:stability_update_ct}
	For any control gain $K\in\cK$, i.e., $(A-BK)$ being Hurwitz and $\|\cT(K)\|_\infty<{\gamma}$,  with $\|K\|<\infty$, 
	suppose that the stepsize  $\eta$ satisfies:
	\begin{itemize}
	\item Natural policy gradient \eqref{eq:exact_npg_ct}: $\eta\leq {1}/{(2\|R\|)}$,
	\item Gauss-Newton \eqref{eq:exact_gn_ct}: $\eta\leq {1}/{2}$. 
%	\item Gradient descent \eqref{eq:exact_pg}: XXXXX, 
	\end{itemize}	
% and  the control gain $K$ with $\|K\|<\infty$ lies in $\cK$, i.e., the  $\cH_\infty$-norm 
%  of the closed-loop transfer function 
%  $\|\cT(K)\|_{\infty}<{\gamma}$. 
%  ensures that the modified Bellman equation \eqref{equ:mod_Bellman_ct} has  a solution $P_{K}>0$.  
%  ; ii) $W^{-1}-\gamma P_{K}>0$; iii) $\rho\big((A-BK)^\top(I-\gamma P_{K}W)^{-1}\big)<1$.   
	 Then the $K'$ obtained from \eqref{eq:exact_npg_ct}-\eqref{eq:exact_gn_ct}   also lies in $\cK$. 
%	 , i.e., $\|\cT(K')\|_{\infty}< {\gamma}$.
	  Equivalently,   $K'$ is stabilizing, i.e., $(A-BK')$ is Hurwitz, and 
%	 is also stabilizing and 
	also enables the Riccati equation \eqref{equ:cont_riccati}  to admit  a stabilizing  solution $P_{K'}\geq 0$ such that  $A-BK'+\gamma^{-2} DD^\top P_{K'}$ is Hurwitz. 
%	 Equivalently, such a $K'$ 
%%	obtained from \eqref{eq:exact_gn_ct}-\eqref{eq:exact_npg_ct}
%	  satisfies that the closed-loop $\cH_\infty$-norm $\|\cT(K')\|_{\infty}<1/\sqrt{\gamma}$. 
%	 of the closed-loop transfer function $H(K')$ as defined in \eqref{equ:mixed_design_transfer_2}  is smaller than $1/\sqrt{\gamma}$. 
\end{theorem}  

The proof of Theorem \ref{thm:stability_update_ct} can be found in  \S\ref{sec:proof_thm_stability_update_ct},  
the key of which is the use of the Bounded Real Lemma, i.e., Lemma \ref{lemma:cont_bounded_real_lemma}, so that it suffices to find some $P>0$ that makes \eqref{equ:cont_equiva_set_cK_cond2} hold. Such a $P$ can then be constructed by perturbing $P_K$, the solution to the Riccati equation under   $K$. 
Note that it is not clear yet if the vanilla PG update \eqref{eq:exact_pg_ct} also enjoys the implicit regularization property. 

%Theorem \ref{thm:stability_update_ct} shows that the \emph{implicit regularization} property of the updates \eqref{eq:exact_npg_ct}-\eqref{eq:exact_gn_ct} exists in the continuous-time setting as well.  
%\issue{WE NEED SOME SIMILAR DISCUSSION ON THE DIFFERENCE FROM LQR, IN THIS CONTINUOUS-TIME SETTING.}

\subsubsection{Global Convergence}

Now we are ready to present the global convergence of the updates \eqref{eq:exact_npg_ct} and \eqref{eq:exact_gn_ct}. 

%Now we  present the  convergence results of the exact gradient methods for continuous-time mixed $\cH_2/\cH_{\infty}$  design.  
%We start by showing  the \emph{global} convergence to a class of stationary points  in general, and to  the \emph{global} optimum  under  additional conditions. 

\begin{theorem}[Global Convergence   for Continuous-Time  Mixed Design]\label{theorem:global_exact_conv_ct} Suppose that $K_0\in\cK$, $\|K_0\|<\infty$, and $(A,C)$ is detectable.    Then, under the 
%following 
stepsize choices  as in Theorem \ref{thm:stability_update_ct},  
%\begin{itemize}
%\item Gauss-Newton \eqref{eq:exact_gn_ct}: $\eta\in[0,1/2]$
%\item Natural policy gradient \eqref{eq:exact_npg_ct}: $[0,1/(2\|R\|)]$,
%\end{itemize}
both  updates  \eqref{eq:exact_npg_ct} and  \eqref{eq:exact_gn_ct}  converge to the global optimum $K^*=R^{-1}B^\top P_{K^*}$, 
%stationary points $K$ where $R K-B^\top P_{K}=\bm{0}$, 
in the sense that the average of $\{\|R K_n-B^\top P_{K_n}\|_F^2\}$ along iterations converges to zero with $O(1/N)$ rate. 
%In addition, if the pair $(A-BK+\gamma^{-2}DD^\top P_{K},D)$ is controllable at the stationary point $K$, then such a convergence is towards the unique \emph{global} optimal policy.   
\end{theorem}

The proof of Theorem \ref{theorem:global_exact_conv_ct} is deferred to \S\ref{sec:proof_theorem:global_exact_conv_ct}. Note that the controllability  assumption made in Proposition  \ref{coro:opt_control_form_cont} is not required here. In other words, even if there might be multiple stationary points, the updates \eqref{eq:exact_npg_ct} and  \eqref{eq:exact_gn_ct} can still avoid the spurious ones and converge to the globally optimal  one $K^*$. This can be viewed as another implication of \emph{implicit regularization} of the natural PG and Gauss-Newton methods: always avoiding bad local minima and converging to the global optimal one. 
%has been , and is satisfied if $DD^\top>0$. Hence, the updates \eqref{eq:exact_npg_ct} and  \eqref{eq:exact_gn_ct} converge to the globally optimal policy in this case. 

%\issue{XXXXXXXXXXX TO BE CHANGED XXXXXXXXXXX}

% that the pair $(A-BK+\gamma^{-2}DD^\top P_{K},D)$ is controllable, is automatically satisfied for continuous-time LEQR problem where $DD^\top=W>0$, namely, Theorem \ref{theorem:global_exact_conv_ct} establishes the \emph{global}  convergence to the \emph{global} optimum. 
%\issue{
%\noindent XXXXXX
%Discussion on the global convergence results. Note that for general cases, i.e., when $DD^\top$ not $>0$, we can only show convergence to \emph{a class of stationary points} where $\|R K_n-B^\top P_{K_n}\|_F^2\to 0$. These points become \emph{a unique global optimum} when that pair is observable.
%\noindent XXXXXX
%}

Although there is no existing result on the global convergence of policy gradient for  \emph{continuous-time} LQR,  we believe that using similar techniques as in \cite{fazel2018global}, global linear rate can be achieved. In comparison, only sublinear rate can be shown for  global convergence for the continuous-time  mixed design.   But still, (super)-linear rates can be established locally around the optimum, whose  proof is provided in \S\ref{sec:proof_theorem:local_exact_conv_ct}. 

\begin{theorem}[Local (Super-)Linear  Convergence   for Continuous-Time  Mixed Design]\label{theorem:local_exact_conv_ct} 
Suppose that the  conditions  in Theorem \ref{theorem:global_exact_conv_ct} hold, and additionally  $DD^\top>0$ holds. Then, under the 
%following 
stepsize choices  as in Theorem \ref{thm:stability_update_ct}, both  updates  \eqref{eq:exact_npg_ct} and  \eqref{eq:exact_gn_ct}  converge to the optimal control gain $K^*$ with \emph{locally linear} rates,  in the sense that the objective $\{\cJ(K_n)\}$ defined in \eqref{equ:form_J1}  converges to $\cJ(K^*)$ with linear rate. In addition, if $\eta=1/2$,  the Gauss-Newton update \eqref{eq:exact_gn_ct} converges to   $K^*$ with locally \emph{Q-quadratic}  rate. 
\end{theorem}

As for the discrete-time results, the  locally linear rates are caused by the fact that  the \emph{gradient dominance} property holds only locally for mixed design problems. Moreover, notice that the Gauss-Newton update here with stepsize $\eta=1/2$ resembles the \emph{policy iteration} algorithm for continuous-time LQR \citep{kleinman1968iterative}, where $P_K$ is the solution to a Lyapunov equation, instead of the  Riccati equation \eqref{equ:cont_riccati} in our problem. 
Hence, the local Q-quadratic rate here is largely expected as in \cite{kleinman1968iterative}. 
%This type of  globally sublinear and locally (super-)linear convergence resembles  the behavior  of (Quasi)-Newton methods for nonconvex optimization  \cite{nesterov2006cubic,ueda2010convergence}, and policy gradient  methods for zero-sum LQ games \cite{zhang2019policy}. 

\begin{remark}[Model-Free Algorithms]\label{remark:discuss_cont}
For continuous-time settings, the relationship between $\cH_2/\cH_\infty$ mixed design, risk-sensitive control (continuous-time LEQG), maximum entropy $\cH_\infty$ control, and zero-sum LQ  differential games have also been established in the literature \citep{mustafa1989relations,jacobson1973optimal}. Hence, model-free algorithms can also be developed from the perspective of LQ games, as we discussed in \S\ref{sec:discussion} for the  discrete-time setting. 
\end{remark}

%\issue{XXXX}

\clearpage

\section{Supplementary  Proofs}\label{sec:supp_proof}
In this section, we provide supplementary proofs for several results stated before.

\subsection{Proof of Lemmas   
%\ref{lemma:nonconvex_dis} and \ref{lemma:nonconvex_cont}
\ref{lemma:discrete_bounded_real_lemma} and \ref{lemma:cont_bounded_real_lemma}}\label{sec:proof_lemma_bounded_real_lemma}
%\issue{
%\begin{proof} 
As a surrogate, we also define a discrete-time transfer function $\tilde \cT(K)$  as 
%\#\label{equ:dynamic_sys_real_cont}
\#\label{equ:dynamic_sys_surro}
\renewcommand\arraystretch{1.3}
\tilde \cT(K):=\mleft[
\begin{array}{c|c}
  A-BK &  1/{\gamma} \cdot D \\
  \hline
  (C^\top C+K^\top R K)^{1/2}& 0
\end{array}
\mright]. 
\#
%\#
% a transfer function $\tilde \cT(K)$ as 
%\$
%\renewcommand\arraystretch{1.3}
%\tilde \cT(K):=\mleft[
%\begin{array}{c|c}
%  A-BK &  1/{\gamma} \cdot D \\
%  \hline
%  (C^\top C+K^\top R K)^{1/2}& \bm{0} 
%\end{array}
%\mright]. 
%\$
%First, since $Q>0$, the pair $(A-BK,Q+K^\top R K)$ is observable. 
%With the transfer function $\tilde \cT(K)$ defined in \eqref{equ:dynamic_sys_real},
First note that since $\rho(A-BK)<1$,  $\tilde \cT(K)$ is a proper and real rational stable transfer matrix. Thus, the discrete-time  Bounded Real Lemma \cite[Theorem $21.12$, page $539$]{zhou1996robust} can be applied. 
Note that the statement ($c$) in  \cite[Theorem $21.12$]{zhou1996robust} is equivalent to the second condition in the lemma. 
%conditions in the definition of $\cK$ in \eqref{equ:define_cK}. 
%i)-iii) in Theorem \ref{thm:stability_update}. 
In particular,   
the Riccati equation   \eqref{equ:discret_riccati}  
%\eqref{equ:def_PK}-\eqref{equ:def_tPK}
 here is identical to the equation in \cite[Theorem $21.12$ ($c$)]{zhou1996robust}, and the inequality condition in \cite[Theorem $21.12$ ($c$)]{zhou1996robust} translates to $I-\gamma^{-2} D^{\top}P_{K}D>0$. 
% , which is equivalent to condition $W^{-1}-\gamma P_{K}>0$. 
 Moreover, the stability argument in \cite[Theorem $21.12$ ($c$)]{zhou1996robust} translates to that the matrix $(I-\gamma^{-2} DD^\top P_{K})^{-1}(A-BK)$ is stable. 
In addition, by \cite[Proposition $1$]{ionescu1992computing}, if such a stabilizing solution $P_K$ exists, it must be unique. 
%, i.e., $\rho\big((A-BK)^\top(I-\gamma P_KW)^{-1}\big)<1$.  
Hence, by the equivalence between   ($a$) and ($c$) in \cite[Theorem $21.12$]{zhou1996robust}, we conclude that the second condition  is equivalent to  $\|\tilde \cT(K)\|_\infty<1$, which   is further equivalent to   $\|\cT(K)\|_\infty<{\gamma}$.

Moreover, by the equivalence between statements  ($c$) and ($d$) in \cite[Theorem $21.12$]{zhou1996robust}, 
%, for any $K\in\cK$, 
%the conditions i)-iii) in our theorem can be proved by showing 
the following linear matrix inequality (LMI) condition also holds equivalently: 
%, namely, statement ($d$) of  \cite[Theorem $21.12$]{zhou1996robust} holds: 
there exists some $P>0$, such that 
%\#\label{equ:LMI_cond}
%\mleft[
%\begin{array}{cc}
%  (A-BK')^\top P(A-BK')-P+Q+K'^\top R K'  & (A-BK')^\top P \\
%  P(A-BK')& P-\frac{1}{\gamma}W^{-1} 
%\end{array}
%\mright]<0.
%\# 
%Suppose 
$I-\gamma^{-2} D^\top P D>0$ and 
\$
%then by Schur complement, \eqref{equ:LMI_cond} is equivalent to the following inequality:
&(A-BK)^\top P(A-BK)-P+C^\top C+K^\top R K+(A-BK)^\top PD(\gamma^2I-D^\top P D)^{-1}D^\top P(A-BK)<0,
\$
which reduces to the third condition in the lemma. This proves Lemma \ref{lemma:discrete_bounded_real_lemma}. 

By definition of  $\cT(K)$ in \eqref{equ:mixed_design_transfer2}, since $K$ is stabilizing, we know that $\cT(K)$ is proper and real rational stable. Thus, the continuous-time  Bounded Real Lemma  \cite[Corollary $13.24$, page $352$]{zhou1996robust} can be applied. The statement ($iv$) in \cite[Corollary $13.24$]{zhou1996robust} is equivalent to the second condition in the lemma.   
%i)-iii) in Theorem \ref{thm:stability_update}. 
In particular,   the matrices $R$ and $H$ in \cite[Corollary $13.24$]{zhou1996robust} have the specific forms of $R:= \gamma^2\cdot I$ and 
\$
H:=\mleft[
\begin{array}{cc}
  (A-BK)^\top   & 1/\gamma^2\cdot D D^\top\\
  -(C^\top C+K^\top RK)& -(A-BK)
\end{array}
\mright], 
\$
so that the Riccati equation induced by $H$ is  identical to \eqref{equ:cont_riccati}. 
%By Assumption \ref{assum:coeff_matrices}, 
%Since $Q>0$, the pair $(A-BK,Q+K^\top R K)$ is observable. 
We also note that by \cite[Theorem $1$]{molinari1973stabilizing}, if such a stabilizing solution exists, it must be unique. 
Thus, the equivalence between the first two conditions follows by  the equivalence of statements $(i)$ and $(iv)$ in  \cite[Corollary $13.24$]{zhou1996robust}.

Define a surrogate transfer function $\tilde \cT(K)$ as in \eqref{equ:dynamic_sys_surro}. 
Then, $\|\cT(K)\|_\infty<{\gamma}$ is equivalent to $\|\tilde\cT(K)\|_{\infty}<1$. 
By the continuous-time KYP Lemma  \cite[Lemma $7.3$, page $212$]{dullerud2013course}, this is equivalent to the statement that there exists some $P>0$ such that 
\$
\mleft[
\begin{array}{cc}
  (A-BK)^\top P+P(A-BK)+C^\top C+K^\top R K  & 1/{\gamma}\cdot PD \\
 1/{\gamma}\cdot D^\top P& -I 
\end{array}
\mright]<0. 
\$
%which is equivalent to the LMI \eqref{equ:cont_equiva_set_cK_cond_1}.
Since $-I<0$, by Schur complement, it is also equivalent to \eqref{equ:cont_equiva_set_cK_cond2}, which  completes the proof of Lemma \ref{lemma:cont_bounded_real_lemma}. 
\hfill$\QED$

\subsection{Proof of Lemmas   
%\ref{lemma:nonconvex_dis} and \ref{lemma:nonconvex_cont}
\ref{lemma:nonconvex_Hinf_norm_set} and \ref{lemma:nonconvex_Hinf_norm_set_cont}}\label{sec:proof_lemma:nonconvex}
Recall that for both discrete-time and continuous-time settings, we consider the transfer function $\cT(K)$ defined in \eqref{equ:mixed_design_transfer2} of identical form. 
Consider the example with matrices $A,~B,~Q,~R$ all being $3\times 3$ identity matrices, $W=D^\top D=0.01\cdot I$. For the discrete-time setting, we choose 
\$
K_1=\mleft[
\begin{array}{ccc}
  1 & 0 & -1 \\
  -1 & 1 & 0 \\
  0 & 0 & 1
\end{array}
\mright],\quad 
K_2=\mleft[
\begin{array}{ccc}
  1 & -2 & 0 \\
  0 & 1 & 0 \\
  -1 & 0 & 1
\end{array}
\mright],
\$
and $K_3=(K_1+K_2)/2$, 
then all control gains  $K_1,~K_2,~K_3$  stabilize the system $(A,B)$, since  $\rho(A-BK_1),\rho(A-BK_2)=0<1,~\rho(A-BK_3)=0.8660<1$.  
Nonetheless, we  have $\|\cT(K_1)\|_\infty=0.4350$, $\|\cT(K_2)\|_\infty=0.7011$, while  $\|\cT(K_3)\|_\infty=1.6575$. Hence, the ${\gamma}$-lower-level set of $\cH_\infty$-norm of $\cT(K)$ is nonconvex for any ${\gamma}\in(0.7011,1.6575)$. 

Similarly, for the continuous-time setting, we choose 
\$
K_1=\mleft[
\begin{array}{ccc}
  2 & 0 & -1 \\
  -1 & 2 & 0 \\
  0 & 0 & 2
\end{array}
\mright],\quad 
K_2=\mleft[
\begin{array}{ccc}
  2 & -2 & 0 \\
  0 & 2 & 0 \\
  -1 & 0 & 2
\end{array}
\mright],
\$
and $K_3=(K_1+K_2)/2$, then $K_1,~K_2,~K_3$ all stabilize the system $(A,B)$, since  
\$
\max_{i\in[3]}~[\Re\lambda_{i}(A-BK_1)]=\max_{i\in[3]}~[\Re\lambda_{i}(A-BK_2)]=-1<0,\quad \max_{i\in[3]}~[\Re\lambda_{i}(A-BK_3)]=-0.134<0.
\$
However, $\|\cT(K_1)\|_\infty=0.3860$, $\|\cT(K_2)\|_\infty=0.5306$, while $\|\cT(K_3)\|_\infty=1.1729$. Therefore, $\cK$ is not convex when ${\gamma}\in(0.5306,1.1729)$, which completes the proof. 
\hfill$\QED$

\subsection{Proof of Lemmas  \ref{lemma:mixed_design_no_coercivity} and \ref{lemma:mixed_design_no_coercivity_cont}}\label{sec:proof_lemma_mixed_design_no_coercivity}
%\issue{XXXXXXXXXXXXXXXXXXX}
Note that $\|K\|$ may be unbounded  for  $K\in\cK$. We show via counterexamples that for $K$ with $\|K\|<\infty$, the cost does not necessarily goes to infinity as  $K$ approaches the boundary of $\cK$. Suppose $DD^\top >0$ is full-rank. 
For cost $\cJ(K)$ of form  \eqref{equ:form_J1}, it  remains finite as long as $P_K$ is finite. For continuous-time settings, by the Bounded Real Lemma, i.e., Lemma \ref{lemma:cont_bounded_real_lemma},   for any $K\in\cK$, $A-BK+\gamma^{-2}DD^\top P_K$ is always Hurwitz. Then, if  there is a sequence $\{K_n\}$ approaching $\partial \cK$ such that $\lambda_{\max}(P_{K_n})\to \infty$ as $n\to\infty$, then there must exist some $N>0$ such that for $n\geq N$, the real part of $A-BK+\gamma^{-2}DD^\top P_K$ is greater than $0$, causing a contradiction. Thus, $\cJ(K)$ in \eqref{equ:form_J1} is always finite. For discrete-time settings, by Lemma \ref{lemma:discrete_bounded_real_lemma}, $I-\gamma^{-2}D^\top P_K D>0$ always holds for $K\in\cK$. Thus, $\lambda_{\max}(P_{K})$ also has to be finite. 

For cost $\cJ(K)$ of the forms  \eqref{equ:form_J2} and \eqref{equ:form_J3},  with  $DD^\top >0$, it is finite if both $P_K$ is finite and $I-\gamma^{-2}D^\top P_K D>0$ is non-singular. The first condition is not violated  as already shown above. We now show via a $1$-dimensional example that the second condition is not violated either as $K\to\partial \cK$. In fact, the Riccati equation 
%  \eqref{equ:cont_riccati} and 
  \eqref{equ:discret_riccati}  that defines  $P_K$ becomes a quadratic equation for the $1$-dimensional case: 
\#\label{equ:quadratic_1_d_discrete}
%&{\rm \textbf{Continuous-time:}}\qquad\qquad\quad \gamma^{-2} D^2  P_K^2+2(A-BK) P_K+C^2+ R K^2=0,
%\label{equ:cont_P_solution}
%\\
%&{\rm \textbf{Discrete-time:}}\quad 
D^2P_K^2-[\gamma^2-(A-BK)^2\gamma^2 +(C^2+RK^2)D^2]P_K +(C^2+RK^2)\gamma^2=0.
%\label{equ:discrete_P_solution}
\#
Thus, it is possible that the condition for the \emph{existence} of  solutions to the quadratic equations is \emph{restricter} than the conditions on $P	_K$ in the Bounded Real Lemma. Specifically, the solutions have the following  form
\#
%&{\rm \textbf{Continuous-time:}}\qquad\quad P_K=\frac{-2(A-BK)\pm\sqrt{{[2(A-BK)]^2-4\gamma^{-2}D^2(C^2+RK^2)}}}{2\gamma^{-2}D^2},\label{equ:cont_P_solution}\\
%&{\rm \textbf{Discrete-time:}}\quad 
&P_K=\frac{\gamma^2-(A-BK)^2\gamma^2 +(C^2+RK^2)D^2}{2D^2}\label{equ:discrete_P_solution}\\
&\qquad\qquad\qquad\qquad\qquad\pm\frac{\sqrt{[\gamma^2-(A-BK)^2\gamma^2 +(C^2+RK^2)D^2]^2-4D^2(C^2+RK^2)\gamma^2}}{2D^2}.\notag
\#
%For \eqref{equ:cont_P_solution}, 
%%if $B^2<\gamma^{-2}D^2R$, then 
%%the leading term of 
%the discriminant $\blacktriangle:=[2(A-BK)]^2-4\gamma^{-2}D^2(C^2+RK^2)\geq 0$
%% is negative, and thus $\blacktriangle\geq 0$ 
% yields a closed set of $K$. Suppose the choices of $A,B,C,D,R,\gamma$ ensure that $\blacktriangle=0$ admits at least one solution. 
%  Moreover, the matrices  
%\$
%A-BK+\gamma^{-2}D^2P_K=\pm\sqrt{\blacktriangle}/2,
%\$
%where one of them is always stabilizing. 
%In addition, since $A-BK<0$, it follows that $P_K\geq 0$. 
%This means that as long as \eqref{equ:cont_P_solution} admits two solutions, one of them satisfies Bounded Real Lemma, and thus $K\in\cK$. 
%As a result, as $K$ approaches the boundary of $\{K\given \blacktriangle\geq 0\}$, which also constitutes the boundary $\partial\cK$, the stabilizing $P_K$ approaches $-(A-BK)/(\gamma^{-2}D^2)>0$, a finite value. The above argument can be numerically verified by choosing $A=1.75$, $B=2$, $C^2=1$, $R=1$, $D^2=0.01$, $\gamma=8.1081$. In this case, if $K\to 0.8832$, which is the point such that $\blacktriangle\to 0$, then $A-BK+\gamma^{-2}D^2P_K<0$, and $P_K\to 3.2525>0$, a finite value. 
%Similarly, for 
%For discrete-time solution \eqref{equ:discrete_P_solution}, 
Denote the discriminant of \eqref{equ:quadratic_1_d_discrete}  by $\blacklozenge$, and let $\blacklozenge=0$ admit at least one solution. 
Moreover, 
\$
1-\gamma^{-2}D^2P_K&=1-\frac{1-(A-BK)^2 +\gamma^{-2}(C^2+RK^2)D^2}{2}\pm\frac{\gamma^{-2}\sqrt{\blacklozenge}}{2}\\
&=\frac{1+(A-BK)^2}{2}-\frac{\gamma^{-2}(C^2+RK^2)D^2}{2}\pm\frac{\gamma^{-2}\sqrt{\blacklozenge}}{2},
\$
which, as $\blacklozenge\to 0$, can be greater than $0$ with small enough $D$ and large enough $\gamma$. 
Additionally, if the choices of $A,B,C,D,R,\gamma$ ensure that $(A-BK)(1-\gamma^{-2}P_KD^2)^{-1}<1$, then such a $K\in\cK$. 
This way, as $K$ approaches the boundary of $\{K\given \blacklozenge\geq 0\}$, it is also approaching $\partial \cK$, while the value of $P_K$ approaches $[\gamma^2-(A-BK)^2\gamma^2 +(C^2+RK^2)D^2]\cdot(2D^2)^{-1}$, a finite value. The above argument can be verified numerically by  choosing $A=2.75$, $B=2$, $C^2=1$, $R=1$, $D^2=0.01$, $\gamma=0.2101$. In this case, $1-\gamma^{-2}D^2P_K\to 0.2354>0$ and $(A-BK)(1-\gamma^{-2}P_KD^2)^{-1}\to 0.9998<1$ if $K\to 1.2573$, which is the value that makes $\blacklozenge\to 0$. However, the corresponding $P_K\to [\gamma^2-(A-BK)^2\gamma^2 +(C^2+RK^2)D^2]\cdot(2D^2)^{-1}=3.3752>0$, a finite value that also satisfies $1-\gamma^{-2}D^2P_K>0$. 
Hence, both the costs in \eqref{equ:form_J2} and \eqref{equ:form_J3} approach a finite value, which  completes the proof.
\hfill$\QED$

\subsection{Proof of Lemmas  \ref{lemma:differentiability_policy_grad}}\label{sec:proof_lemma_differentiability_policy_grad}
%\begin{proof}
%	The proof follows by using the implicit function theorem \cite{krantz2012implicit}.

%\vspace{3pt}
%~\\
%\noindent{\bf{Discrete-Time:}} 
%  	\vspace{3pt}

Note that $\cJ(K)$ defined in \eqref{equ:form_J2}  is differentiable with respect to  $P_K$, provided that 
%$I-\gamma^{-2} P_KDD^\top$ is non-singular, i.e.,  
$\det (I-\gamma^{-2} P_KDD^\top)> 0$. This holds for any $K\in\cK$ since  by Lemma \ref{lemma:discrete_bounded_real_lemma}
%\small
\$
I-\gamma^{-2} D^\top P_KD>0\Rightarrow \det (I-\gamma^{-2} D^\top P_KD)=\det (I-\gamma^{-2}  P_KDD^\top)>0. 
%\Rightarrow \det (I-\gamma P_KW)>0,
\$
\normalsize
where we have used Sylvester's determinant theorem that $\det (I+AB)=\det(I+BA)$. 
%the fact that $W^{-1}>0$ and $\det(AB)=\det A\cdot\det B$. 
Thus, it suffices to show that $P_K$ is differentiable with respect to $K$.

Recall that 
\#\label{equ:def_tP_discret}
\tP_K=P_K+P_KD(\gamma^2I-D^\top P_K D)^{-1}D^\top P_K=(I-\gamma^{-2}P_KDD^\top)^{-1}P_K, 
\#
where the second equation uses matrix inversion lemma, 
and define the operator $\Psi:\RR^{m\times m}\times \RR^{d\times m}\to \RR^{m\times m}$ as 
	\$
	\Psi(P_K,K):&=C^\top C+K^\top RK+(A-BK)^\top \tP_K(A-BK).  
%	\\
%	&=C^\top C+K^\top RK+(A-BK)^\top (P^{-1}_K-\gamma W)^{-1}(A-BK),
	\$ 
%	where we use \eqref{equ:def_tP_discret} to rewrite $\tP_K$.  
	Note that $\Psi$ is continuous with respect to both  $P_K$ and $K$, provided that $\gamma^2I-D^\top P_K D>0$. Also note that the Riccati equation \eqref{equ:discret_riccati} can be written as 
	\#\label{equ:Psi_Mod_Bellman}
	\Psi(P_K,K)=P_K.
	\# 
	Notice the fact 
%	relationship between  Kronecker product and matrix vectorization, it follows 
	that for any matrices $A$, $B$, and $X$ with proper dimensions
	\#\label{equ:prop_vect_product}
	\vect(AXB) =\big(B^\top \otimes A\big) \vect(X). 
	\#
	Thus, by vectorizing both sides of \eqref{equ:Psi_Mod_Bellman}, 
	we have
	\#\label{equ:vect_fixed_point}
	&\vect\big(\Psi(P_K,K)\big)=\vect(C^\top C+K^\top RK)+\vect\big((A-BK)^\top \tP_K(A-BK)\big)
	\\
	&\quad=\vect(C^\top C+K^\top RK)+\big[(A-BK)^\top \otimes (A-BK)^\top\big]\cdot\vect\big((I-\gamma^{-2}P_KDD^\top)^{-1}P_K\big)
	=\vect(P_K). \notag
	\#
	By defining $\tilde\Psi:\RR^{m^2}\times \RR^{dm}\to \RR^{m^2}$ as a new mapping such that $\tilde\Psi\big(\vect(P_K),\vect(K)\big):=\vect\big(\Psi(P_K,K)\big)$, the fixed-point equation \eqref{equ:vect_fixed_point} can be re-written as 
	\#\label{equ:vec_fixed_point}
	\tilde\Psi\big(\vect(P_K),\vect(K)\big)=\vect(P_K).
	\# 
	Since $\vect$ is a linear mapping, 
%	in order to show that $P_K$ is differentiable w.r.t. $K$, 
	it now suffices to show that $\vect(P_K)$ is differentiable with respect to $\vect(K)$.   To this end,  we apply the  implicit function theorem \citep{krantz2012implicit} on the fixed-point equation \eqref{equ:vec_fixed_point}. 
	To ensure the applicability, we first note that the set $\cK$ defined in \eqref{equ:define_cK} is an open set. In fact, by Lemma  \ref{lemma:discrete_bounded_real_lemma}, for any $K\in\cK$, there exists some $P>0$ such that 
	the two LMIs in \eqref{equ:discrete_equiva_set_cK_cond} hold. Since the inequality is strict, there must exists a small enough ball around $K$ such for any $K'$ in the ball, the LMIs still hold. Hence, the set $\cK$ is open by definition. 
	
	Moreover,  by the chain rule of  matrix differentials \cite[Theorem $9$]{magnus1985matrix}, we  know that 
	\small
	\#\label{equ:matrix_diff_Pt}
	\frac{\partial\vect\big((I-\gamma^{-2}P_KDD^\top)^{-1}P_K\big)}{\partial \vect^\top (P_K)}
%	&=(P_K\otimes I)\cdot\frac{\partial \vect[(I-\gamma^{-2}P_KDD^\top)^{-1}]}{\partial \vect^\top (P_K)}+[I\otimes (I-\gamma^{-2}P_KDD^\top)^{-1}]\cdot\frac{\partial\vect(P_K)}{\partial \vect^\top(P_K)}\notag\\
	=(P_K\otimes I)\cdot\frac{\partial \vect[(I-\gamma^{-2}P_KDD^\top)^{-1}]}{\partial \vect^\top (P_K)}+I\otimes (I-\gamma^{-2}P_KDD^\top)^{-1},
	\#
	\normalsize
	where $I$ denotes the identity matrix of proper dimension.
	
	Now we claim that 
   \#\label{equ:matrix_diff_trash_1}
	\frac{\partial \vect[(I-\gamma^{-2}P_KDD^\top)^{-1}]}{\partial \vect^\top (P_K)}=[(
	\gamma^{-2} DD^\top )\cdot (I-\gamma^{-2}P_KDD^\top)^{-1}]\otimes (I-\gamma^{-2}P_KDD^\top)^{-1}.
	\#
	To show this, we compare the element at the $[(j-1)m+i]$-th row and the $[(l-1)m+k]$-th column of both sides  of \eqref{equ:matrix_diff_trash_1} with $i,j,k,l\in[m]$, where  both sides are matrices of dimensions  $m^2\times m^2$. On the LHS,  notice  that    
	\$
	\frac{\partial (I-\gamma^{-2}P_KDD^\top)^{-1}}{\partial [P_K]_{k,l}}=(I-\gamma^{-2}P_KDD^\top)^{-1}\cdot \frac{\partial (\gamma^{-2} P_K DD^\top)}{\partial [P_K]_{k,l}}\cdot (I-\gamma^{-2}P_KDD^\top)^{-1},
	\$
	which follows from     $(F^{-1})'=-F^{-1}F'F^{-1}$ for some matrix function $F$. Also notice that 
	\$
	\frac{\partial (\gamma^{-2} P_K DD^\top)}{\partial [P_K]_{k,l}}=\gamma^{-2} \left[\begin{matrix}
		\rule[3pt]{30pt}{0.4pt}&0&\rule[3pt]{30pt}{0.4pt}\\
%		 & \vdots &  \\
		[DD^\top]_{l,1}&\cdots & [DD^\top]_{l,m}  \\  
%		& \vdots & \\
		\rule[3pt]{30pt}{0.4pt}&0&\rule[3pt]{30pt}{0.4pt}
	\end{matrix}\right]\leftarrow k\text{-th row},
	\$ 
	where only the $k$-th row is non-zero and is filled with the $l$-th row of $DD^\top$. Due to these two facts, we have
	\#\label{equ:matrix_diff_trash_2}
	&\bigg[\frac{\partial \vect[(I-\gamma^{-2} P_K DD^\top)^{-1}]}{\partial \vect^\top (P_K)}\bigg]_{(j-1)m+i,(l-1)m+k}=\frac{\partial [(I-\gamma^{-2} P_K DD^\top)^{-1}]_{i,j}}{\partial [P_K]_{k,l}}\notag\\
	&\quad =\gamma [(I-\gamma^{-2} P_K DD^\top)^{-1}]_{i,k}\cdot\sum_{q=1}^m [DD^\top]_{l,q} \cdot[(I-\gamma^{-2}P_K DD^\top)^{-1}]_{q,j}.
	\#
	On the right-hand side of \eqref{equ:matrix_diff_trash_1},  we have 
\#\label{equ:matrix_diff_trash_3}
	&\big[[(
	\gamma^{-2} DD^\top)\cdot (I-\gamma^{-2} P_K DD^\top)^{-1}]\otimes (I-\gamma^{-2} P_KDD^\top)^{-1}\big]_{(j-1)m+i,(l-1)m+k}\notag\\
	&\quad=[(
	\gamma^{-2} DD^\top)\cdot (I-\gamma^{-2} P_K DD^\top)^{-1}]_{j,l}\cdot [(I-\gamma^{-2} P_KDD^\top)^{-1}]_{i,k}\notag\\
	&\quad =[(
	\gamma^{-2} DD^\top)\cdot (I-\gamma^{-2} P_K DD^\top)^{-1}]_{l,j}\cdot [(I-\gamma^{-2} P_KDD^\top)^{-1}]_{i,k},
	\#
	where the first equation follows from the definition of Kronecker product, and the second one is due to that  the matrix 
	\#\label{equ:symetrix_trash_1}
	(\gamma^{-2} DD^\top)\cdot (I-\gamma^{-2} P_K DD^\top)^{-1}
%	=\gamma^{-2} DD^\top[I+ P_KD(\gamma^{-2}I-D^\top P_K D)^{-1}D^\top]
=D(\gamma^2 I-D^\top P_K D)^{-1}D^\top
	\# is symmetric. 
	Thus, \eqref{equ:matrix_diff_trash_2} and \eqref{equ:matrix_diff_trash_3} are identical for any $(i,j,k,l)$, which proves \eqref{equ:matrix_diff_trash_1}. 
	
	By substituting \eqref{equ:matrix_diff_trash_1} into  \eqref{equ:matrix_diff_Pt}, we have
\$
		&\frac{\partial\vect\big((I-\gamma^{-2} P_K DD^\top)^{-1}P_K\big)}{\partial \vect^\top (P_K)}=(P_K\otimes I)\cdot[(
	\gamma^{-2} DD^\top)\cdot (I-\gamma^{-2} P_K DD^\top)^{-1}]\otimes (I-\gamma^{-2} P_K DD^\top)^{-1}\\
	&\qquad\qquad\qquad\qquad\qquad\qquad\qquad+I\otimes (I-\gamma^{-2} P_K DD^\top)^{-1}\notag\\
	&\quad= [(
	\gamma^{-2} P_KDD^\top)\cdot (I-\gamma^{-2} P_K DD^\top)^{-1}]\otimes (I-\gamma^{-2} P_K DD^\top)^{-1}+I\otimes (I-\gamma^{-2} P_K DD^\top)^{-1}\notag\\
	&\quad= [I+(
	\gamma^{-2} P_KDD^\top)\cdot (I-\gamma^{-2} P_K DD^\top)^{-1}]\otimes (I-\gamma^{-2} P_K DD^\top)^{-1} \\
	&\quad=(I-\gamma^{-2} P_K DD^\top)^{-1}\otimes (I-\gamma^{-2} P_K DD^\top)^{-1},
	\$
	where the second equation uses the fact that $(A\otimes B)(C\otimes D)=(AC)\otimes(BD)$,  the third one uses $(A\otimes B)+(C\otimes B)=(A+C)\otimes B$, and the last one uses matrix inversion lemma. Hence, by   \eqref{equ:matrix_diff_Pt}, we can write the partial derivative of $  \tilde\Psi\big(\vect(P_K),\vect(K)\big)$ as
	\$
	&\frac{\partial \tilde\Psi\big(\vect(P_K),\vect(K)\big)}{\partial \vect^\top (P_K)}=\big[(A-BK)^\top \otimes (A-BK)^\top\big]\cdot \frac{\partial\vect\big((I-\gamma P_KDD^\top)^{-1}P_K\big)}{\partial \vect^\top (P_K)}\notag\\
	&\quad=\big[(A-BK)^\top \otimes (A-BK)^\top\big]\cdot  \big[(I-\gamma P_KDD^\top)^{-1}\otimes (I-\gamma P_KDD^\top)^{-1}\big]\notag\\
	&\quad=\big[(A-BK)^\top(I-\gamma P_KDD^\top)^{-1}\big] \otimes \big[(A-BK)^\top(I-\gamma P_KDD^\top)^{-1}\big]. 
	\$
	Therefore, the partial derivative 
	\small
\$
	&\frac{\partial  \big[\tilde\Psi\big(\vect(P_K),\vect(K)\big)-\vect(P_K)\big]}{\partial \vect^\top (P_K)}=\big[(A-BK)^\top(I-\gamma^{-2} P_K DD^\top)^{-1}\big] \otimes \big[(A-BK)^\top(I-\gamma^{-2} P_K DD^\top)^{-1}\big]-I, 
	\$
	\normalsize
	which is invertible, since the eigenvalues of $[(A-BK)^\top(I-\gamma^{-2} P_K DD^\top)^{-1}] \otimes [(A-BK)^\top(I-\gamma^{-2} P_K DD^\top)^{-1}]$ are the products of the eigenvalues of $(A-BK)^\top(I-\gamma^{-2} P_K DD^\top)^{-1}$, and  	 the matrix $(A-BK)^\top(I-\gamma^{-2} P_K DD^\top)^{-1}$ has spectral radius less than $1$ for all $K\in\cK$. Also, since $\tilde\Psi\big(\vect(P_K),\vect(K)\big)-\vect(P_K)$ is continuous with respect to both $\vect(P_K)$ and $\vect(K)$,   by the implicit function theorem \citep{krantz2012implicit}, we know that there exists an open  neighborhood around $\vect(P_K)$ and $\vect(K)$ (thus including $\vect(P_K)$ and $\vect(K)$), so that 	$\vect(P_K)$ is a continuously differentiable function with respect to $\vect(K)$, and  so is $P_K$ with respect to $K$, in the neighborhood. Note that this holds for any $K\in\cK$. This proves the differentiability of the objective $\cJ(K)$ at all $K\in\cK$. 
%	, and also implies its   continuity in $K$ at all $K\in\cK$. 
	
	Now we establish the form of the policy gradient. By Lemma \ref{lemma:discrete_bounded_real_lemma}, we know that 
	for any $K\in\cK$, $(A-BK)^\top(I-\gamma^{-2} P_KDD^\top)^{-1}$  is stable and $I-\gamma^{-2}D^\top P_K D>0$.   
%	the differentiability of the objective is guaranteed for those $K\in\cK$
%	, namely, for those $K$ that make $P_K> 0$ exist, $(A-BK)^\top(I-\gamma P_KW)^{-1}$  stable, and $W^{-1}-\gamma P_K>0$. 
	Therefore,  the expression  $\Delta_K$ in \eqref{equ:def_Delta} exists, and so does the expression for $\nabla \cJ(K)$. 
	We then verify the expressions by showing the form of the directional derivative $\nabla_{K_{ij}} \cJ(K)$, i.e., the derivative with respect to  each element $K_{ij}$ in the matrix $K$.  	By definition of $\cJ(K)$ in \eqref{equ:form_J2},  we have
	\small
	\#\label{equ:trash_1}
	&\nabla_{K_{ij}} \cJ(K)=-{\gamma^2}\tr\big\{(I-\gamma^{-2} P_KDD^\top)^{-\top}[\nabla_{K_{ij}}(I-\gamma^{-2} P_KDD^\top)]^{\top}\big\}\\
	&\quad =-{\gamma^2}\tr\big[(I-\gamma^{-2} P_KDD^\top)^{-1}\nabla_{K_{ij}}(I-\gamma^{-2} P_KDD^\top)\big] =\tr\big[(I-\gamma^{-2} P_KDD^\top)^{-1}\nabla_{K_{ij}}(P_K DD^\top)\big],\notag
	\#
	\normalsize
	where the first equality follows from the chain rule and the fact that $\nabla_X \log\det X=X^{-\top}$, and the second one follows from the fact that $\tr(A^\top B^\top)=\tr(BA)^\top=\tr(BA)=\tr(AB)$.
%	 and the third is because  $I$ does not   depend on $K$.
	  Furthermore, since $DD^\top$  is independent of $K$, and $\tr(ABC)=\tr(BCA)$, we  obtain from \eqref{equ:trash_1} and \eqref{equ:symetrix_trash_1} that
	\#\label{equ:trash_2}
	\nabla_{K_{ij}} \cJ(K)&=\tr\big[(I-\gamma^{-2} P_KDD^\top)^{-1}\nabla_{K_{ij}}P_K \cdot DD^\top\big]=\tr\big[\nabla_{K_{ij}}P_K\cdot DD^\top(I-\gamma^{-2} P_KDD^\top)^{-1}\big]\notag\\
	&=\tr\big[\nabla_{K_{ij}}P_K \cdot D(I-\gamma^{-2} D^\top P_K D)^{-1}D^\top\big]. 
	\#
	
	Now we establish the recursion of $\nabla_{K_{ij}} \cJ(K)$ using Riccati  equation \eqref{equ:discret_riccati}. Specifically, letting {$M:=D(I-\gamma^{-2} D^\top P_K D)^{-1}D^\top$},  we have from \eqref{equ:discret_riccati}, \eqref{equ:def_tP_discret}, and \eqref{equ:trash_2} that 
	\#\label{equ:trash_3}
	&\nabla_{K_{ij}} \cJ(K)=\tr\big(\nabla_{K_{ij}}P_K \cdot M\big)
%	=\nabla_{K_{ij}}\tr\big(P_K  M\big)
	\notag\\ 
	&\quad=(2RK M)_{ij}-\big[2B^\top \tP_K(A-BK)M\big]_{ij}+\tr\big[(A-BK)^\top (\nabla_{K_{ij}}\tP_K)(A-BK)M\big],
	\#
	where on the right-hand side of \eqref{equ:trash_3}, the first term is due to the fact that $\nabla_K\tr(K^\top RK M)=2RK M$ for any positive definite (and thus symmetric) matrix $M$, the second term is the gradient with $\tP_K$ fixed, and the third term is the gradient with $A-BK$ fixed. 
	
	In addition, 
%	from \eqref{equ:def_tPK}, we can establish the relationship between $\nabla_{K_{ij}}\tP_K$ and $\nabla_{K_{ij}}P_K$. In particular,  
	by taking the derivative on both sides of \eqref{equ:def_tP_discret}, we have
	\#\label{equ:rela_PK}
	&\nabla_{K_{ij}}\tP_K=\nabla_{K_{ij}}(I-\gamma^{-2}P_KDD^\top)^{-1}P_K+(I-\gamma^{-2}P_KDD^\top)^{-1}\nabla_{K_{ij}}P_K\notag\\
%	=-(P^{-1}_K-\gamma W)^{-1}\cdot\nabla_{K_{ij}}(P^{-1}_K-\gamma W)\cdot(P^{-1}_K-\gamma W)^{-1}\notag\\
	&\quad=(I-\gamma^{-2}P_KDD^\top)^{-1}\nabla_{K_{ij}}P_K\cdot\gamma^{-2}DD^\top(I-\gamma^{-2}P_KDD^\top)^{-1}P_K+(I-\gamma^{-2}P_KDD^\top)^{-1}\nabla_{K_{ij}}P_K\notag\\
	&\quad=(I-\gamma^{-2}P_KDD^\top)^{-1}\nabla_{K_{ij}}P_K\cdot[D(\gamma^2 I-D^\top P_K D)^{-1}D^\top P_K+I]\notag\\
	&\quad=(I-\gamma^{-2} P_K DD^\top)^{-1}\cdot \nabla_{K_{ij}}P_K\cdot (I-\gamma^{-2} DD^\top P_K)^{-1},
	\#
	where the second equation has used the fact that 
	\$
	\nabla_X (P^{-1})=-P^{-1}\cdot \nabla_X P \cdot P^{-1},
	\$
	the third one has used  \eqref{equ:symetrix_trash_1}, 
	and the last one has used the matrix inversion lemma. 
	Also, notice that $(I-\gamma^{-2} DD^\top P_K)^{-1}=(I-\gamma^{-2} P_KDD^\top)^{-\top}$. Thus, \eqref{equ:rela_PK} can be written as 
	\#\label{equ:trash_4}
	\nabla_{K_{ij}}\tP_K=(I-\gamma^{-2} P_K DD^\top)^{-1}\cdot \nabla_{K_{ij}}P_K\cdot (I-\gamma^{-2} P_KDD^\top)^{-\top}.
	\#
	Substituting \eqref{equ:trash_4} into \eqref{equ:trash_3} yields the following  recursion
		\$
	\nabla_{K_{ij}} \cJ(K)&=(2RK M)_{ij}-\big[2B^\top \tP_K(A-BK)M\big]_{ij}\\
	&\qquad+\tr\big[\nabla_{K_{ij}}P_K\cdot\underbrace{(I-\gamma^{-2} P_KDD^\top)^{-\top}(A-BK)M(A-BK)^\top (I-\gamma^{-2} P_K DD^\top)^{-1}}_{M_1}\big].\notag
	\$
	By performing recursion on $\tr\big(\nabla_{K_{ij}}P_K \cdot M_1\big)$, and combining all the $i,j$ terms into a matrix, we obtain the form of the gradient given in Lemma \ref{lemma:differentiability_policy_grad}.		
\hfill$\QED$

\subsection{Proof of Lemma   \ref{lemma:differentiability_policy_grad_ct}}\label{sec:proof_lemma_differentiability_policy_grad_ct}

Similarly as the proof in \S\ref{sec:proof_lemma_differentiability_policy_grad}, $\cJ(K)$ defined in \eqref{equ:def_mixed_formulation_cont} is always differentiable with respect to $P_K$, which makes it sufficient to show the differentiability of $P_K$  with respect to $K$. 	
	We proceed   by using the implicit function theorem, which requires the set $\cK$ considered to be open. 
	By  Lemma   \ref{lemma:cont_bounded_real_lemma}, requiring $K$ to be in $\cK$ is equivalent to requiring the  existence of $P> 0$ such that \eqref{equ:cont_equiva_set_cK_cond2} holds. Since the LHS of \eqref{equ:cont_equiva_set_cK_cond2} is continuous in $K$, and the LMI is strict,  for any $K\in\cK$, there must exist a small enough ball around $K$ such that for any $K'$ inside the ball, the LMI \eqref{equ:cont_equiva_set_cK_cond2} still holds. Hence, $\cK$ is an open set.  
	 
	  By \eqref{equ:prop_vect_product}, we have from  \eqref{equ:cont_riccati} that
	  \small
\#\label{equ:ct_diff_pf_trash_1}
	\psi(\vect(P_K),K):&=[I\otimes(A-BK)^\top +(A-BK)^\top\otimes I]\vect(P_K)+\vect(C^\top C+K^\top R K)+\vect(\gamma^{-2} P_K  D D^\top P_K)\notag\\
	&=0,
	\#
	\normalsize
	where we define the LHS of  \eqref{equ:ct_diff_pf_trash_1} to be $\psi(\vect(P_K),K)$. Since  $\vect(\cdot)$ is a linear mapping, it suffices to show that $\vect(P_K)$ is differentiable with respect to  $K$. 
	By  Theorem $9$ in \cite{magnus1985matrix}, we have
	\$
	&\frac{\partial \vect(\gamma^{-2} P_K  DD^\top  P_K)}{\partial \vect^\top (P_K)}=(P_K\otimes I)\cdot\frac{\partial\vect(\gamma^{-2} P_K  DD^\top)}{\partial \vect^\top(P_K)}\}+[I\otimes (\gamma^{-2} P_K  DD^\top)]\cdot \frac{\partial \vect(P_K)}{\partial \vect^\top(P_K)}\notag\\
	&\quad=(P_K\otimes I)\cdot[(\gamma^{-2}  DD^\top)\otimes I]+I\otimes (\gamma^{-2} P_K  DD^\top)=(\gamma^{-2} P_K  DD^\top)\otimes I+I\otimes (\gamma^{-2} P_K  DD^\top),
	\$
	where the first equation is due to the chain rule  of matrix differentials, the second equation is by definition, and the last one uses the fact that $(A\otimes B)(C\otimes D)=(AC)\otimes (BD)$.
	Hence, by definition of $\psi(\vect(P_K),K)$ in \eqref{equ:ct_diff_pf_trash_1}, we have 
	\#\label{equ:ct_diff_pf_trash_2}
	\frac{\partial \psi(\vect(P_K),K)}{\partial \vect^\top(P_K)}&=I\otimes(A-BK)^\top +(A-BK)^\top\otimes I+(\gamma^{-2} P_K  DD^\top)\otimes I+I\otimes (\gamma^{-2} P_K  DD^\top)\notag\\
	&=I\otimes(A-BK+\gamma^{-2}  DD^\top P_K)^\top +(A-BK+\gamma^{-2}  DD^\top P_K)^\top\otimes I,
	\#
	where the second equation follows from the facts that $(A\otimes B)+(C\otimes B)=(A+C)\otimes B$ and $(P_KDD^\top)^\top=DD^\top P_K$. 
	On the other hand, by Lemma  \ref{lemma:cont_bounded_real_lemma},  the matrix $A-BK+\gamma^{-2}  DD^\top P_K$ is Hurwitz. As a result, ${\partial \psi(\vect(P_K),K)}/{\partial \vect^\top(P_K)}$ from \eqref{equ:ct_diff_pf_trash_2} is invertible. By the implicit function theorem, we conclude that $\vect(P_K)$ is continuously differentiable with respect to $K$, and so is $P_K$ and thus $\cJ(K)$, at some open neighborhood of $\vect(P_K)$ and $K$ (including $\vect(P_K)$ and $K$).  
%	This  also implies the    continuity of $\cJ(K)$ in $K$ at all $K\in\cK$, which 
This completes the proof. 

Now we establish the form of the policy gradient. Define 
\$
\tilde \psi(P_K,K):=(A-BK)^\top P_K+P_K(A-BK)+C^\top C+K^\top RK+\gamma^{-2} P_K  DD^\top P_K. 
\$
Then we have
	\#\label{equ:PG_ct_pf_trash_1}
	\tilde\psi_{P_K}(P_K,K)dP_{K}=dP_{K}(A-BK)+(A-BK)^\top dP_K+\gamma^{-2} dP_KDD^\top P_K+\gamma^{-2} P_KDD^\top dP_K,
	\#
	where $\tilde\psi_{P_K}$ denotes the differential of $\tilde\psi$ with respect to ${P_K}$. Moreover, we have
	\#\label{equ:PG_ct_pf_trash_2}
	\tilde\psi_{K}(P_K,K)dK=P_{K}(-BdK)+(-BdK)^\top P_K+(dK)^\top RK+K^\top R dK.
	\#
	Since $dP_{K}=P'_{K}dK$ and 
	\$
	\tilde\psi_{P_K}(P_K,K)dP_{K}+\tilde\psi_{K}(P_K,K)dK=0, 
	\$ we combine \eqref{equ:PG_ct_pf_trash_1} and \eqref{equ:PG_ct_pf_trash_2} to obtain 
	\#\label{equ:PG_ct_pf_trash_2.5}
	&P'_{K}dK(A-BK)+(A-BK)^\top P'_{K}dK+\gamma^{-2} P'_{K}dKDD^\top P_K+\gamma^{-2} P_KDD^\top P'_{K}dK\\
	&\quad=P_{K}(BdK)+(BdK)^\top P_K-(dK)^\top RK-K^\top R dK=(P_KB-K^\top R)dK+(dK)^\top (B^\top P_K- RK).\notag
	\#
	
	On the other hand, we have
	\#\label{equ:PG_ct_pf_trash_3}
	\cJ'(K)dK=\tr(\nabla \cJ(K)^\top dK),
	\#
	where $\nabla \cJ(K)\in\RR^{d\times m}$,   while by definition of $\cJ(K)$, we also have
	\#\label{equ:PG_ct_pf_trash_4}
	\cJ'(K)dK=\tr(P'_KdK DD^\top),
	\#
	where from \eqref{equ:PG_ct_pf_trash_2.5}, $P'_KdK$ is the solution to the Lyapunov equation
	\#\label{equ:PG_ct_pf_trash_5}
	&P'_{K}dK(A-BK+\gamma DD^\top P_K)+(A-BK+\gamma DD^\top P_K)^\top P'_{K}dK\notag\\
	&\quad=(P_KB-K^\top R)dK+(dK)^\top (B^\top P_K-  RK). 
	\#
	Since $\Lambda_K$ is also the solution to a Lyapunov equation \eqref{equ:Lambda_Lya_def}, we multiply \eqref{equ:PG_ct_pf_trash_5} by $\Lambda_K$ and multiply \eqref{equ:Lambda_Lya_def} by $P'_{K}dK$, and then take trace on both sides of both equations, to obtain the following identity
	\#\label{equ:PG_ct_pf_trash_6}
%	&\Lambda_K P'_{K}dK(A-BK+\gamma^{-2} DD^\top P_K)+\Lambda_K(A-BK+\gamma^{-2} DD^\top P_K)^\top P'_{K}dK\notag\\
%	&\quad=\Lambda_K(P_KB-K^\top R)dK+\Lambda_K(dK)^\top (B^\top P_K-  RK)\\
%	&P'_{K}dK\Lambda_K(A-BK+\gamma^{-2} DD^\top P_K)^\top+P'_{K}dK(A-BK+\gamma^{-2} DD^\top P_K)\Lambda_K+P'_{K}dKW=0. \\
	-\tr\big[\Lambda_K(P_KB-K^\top R)dK+\Lambda_K(dK)^\top (B^\top P_K-  RK)\big]=\tr[P'_{K}dKDD^\top]=\cJ'(K)dK.
	\#
	By further equating \eqref{equ:PG_ct_pf_trash_3}, \eqref{equ:PG_ct_pf_trash_4}, and \eqref{equ:PG_ct_pf_trash_6}, we have
	\$
	\tr(\nabla \cJ(K)^\top dK)&=\tr\big[\Lambda_K(K^\top R -P_KB)dK+\Lambda_K(dK)^\top (R K-B^\top P_K)\big]\\
	&=\tr\big[2\Lambda_K( K^\top R -P_KB)dK\big].
	\$
	This gives the expression of the policy gradient $\nabla \cJ(K)=2[R K-B^\top P_K] \Lambda_K$, completing the proof of Lemma \ref{lemma:differentiability_policy_grad_ct}. 
\hfill$\QED$

\subsection{Proof of Propositions    \ref{coro:opt_control_form_discrete} and \ref{coro:opt_control_form_cont}}\label{proof:coro_opt_control_form}

\noindent{\bf{Discrete-Time:}} 
  	\vspace{3pt}
  	
The proof is based on a game-theoretic perspective on the problem. First, for any $K\in\cK$, by applying Theorem 3.7 in \cite{bacsar1995h}, with $A,~B,~D,~Q,~R,~\gamma$ therein being replaced by $A-BK,~0,~D,~Q+K^\top R K,~R,~\gamma$ here, we obtain that the Riccati equation in \eqref{equ:form_J1} corresponds to the generalized algebraic Riccati equation (3.52b) in \cite{bacsar1995h}, for this auxiliary game. By Lemma \ref{lemma:discrete_bounded_real_lemma}, the solution $P_K\geq 0$ satisfies (3.53) in \cite{bacsar1995h}. Recall that $P_K$ is the unique stabilizing solution to  \eqref{equ:form_J1}, and is thus also minimal if $(A-BK,D)$ is stabilizable \cite[Theorem $3.1$]{ran1988existence}, which is indeed the case since $K\in\cK$  is stabilizing. Hence, by  \cite[Theorem 3.7]{bacsar1995h} (ii)(iv), the controller and the disturbance that attain the upper-value of the game have,  respectively, the forms of $u_t=0$ and $w_t=(\gamma^2 I-D^\top P_{K}D)^{-1}D^\top P_{K}(A-BK)x_t$ for all $t$. Note that  $(A,Q^{1/2})$ being detectable   in \cite[Theorem 3.7]{bacsar1995h} is not used when applying (ii)(iv). This shows that in the original game with $A,~B,~D,~Q,~R,~\gamma$ (as defined in \cite[Chapter 3.7]{bacsar1995h}), and with a  fixed $K\in\cK$,  the maximizing disturbance has the form as $w_t$ above, and the value under the pair $(K, -(\gamma^2 I-D^\top P_{K}D)^{-1}D^\top P_{K}(A-BK))$ is indeed $x_0^\top P_{K}x_0$. By again applying \cite[Theorem 3.7]{bacsar1995h} to the original game, we know that the value is $x_0^\top P_{K^*}x_0$, and  is achieved by the optimal controller $u_t^*=-K^* x_t$ and the maximizing disturbance $w_t^*=[(\gamma^2 I-D^\top P_{K^*}D)^{-1}D^\top P_{K^*}(A-BK^*)]x_t$, with $K^*$ being defined in the proposition. By definition of the value of the game, we know that $x_0^\top P_{K}x_0\geq x_0^\top P_{K^*}x_0$ for any $K\in\cK$. As the above arguments hold for any $x_0$, we know that $P_{K}\geq P_{K^*}$. Finally, notice that for $K,K^*\in\cK$, if $P_K\geq P_{K^*}$, then $0<I-\gamma^{-2}D^\top P_KD\leq I-\gamma^{-2}D^\top P_{K^*}D$ (cf. Lemma \ref{lemma:discrete_bounded_real_lemma}). By $\det(I-\gamma^{-2}P_KDD^\top )=\det(I-\gamma^{-2}D^\top P_KD)$, we know that $\cJ(K)\geq \cJ(K^*)$ for any $K\in\cK$. This completes the proof for the first half of the proposition.  
  	
%\begin{proof}
%	Since the optimum of $\cJ(K)$ is assumed to be achieved by some $K^*\in\cK$, then 
For the second half of the proposition, note that 
%$\Delta_K$ is the solution to the Lyapunov equation
%\small
%\$
%\Delta_K=\big[(I-\gamma^{-2} P_KDD^\top)^{-\top}(A-BK)\big]\Delta_K\big[(A-BK)^\top(I-\gamma^{-2} P_KDD^\top)^{-1}\big]+D(I-\gamma^{-2}D^\top P_K D)^{-1}D^\top,
%\$
%\normalsize
%which always satisfies 
$\Delta_K\geq 0$ since $I-\gamma^{-2}D^\top P_K D>0$ for any $K\in\cK$ by Lemma \ref{lemma:discrete_bounded_real_lemma}. Also, since $(I-\gamma^{-2}D^\top P_K D)^{-1}\geq I$, we know that 
\#\label{equ:Delta_K_lower_bnd}
\Delta_K\geq \sum_{t=0}^\infty \big[(I-\gamma^{-2} P_KDD^\top)^{-\top}(A-BK)\big]^tDD^\top\big[(A-BK)^\top(I-\gamma^{-2} P_KDD^\top)^{-1}\big]^t.
\# 
By \cite[Lemma $21.2$]{zhou1996robust}, the RHS of \eqref{equ:Delta_K_lower_bnd} is always positive definite, since 
 $\big((I-\gamma^{-2}  P_KDD^\top)^{-\top}(A-BK),D\big)$ is controllable, i.e., $\big((A-BK)^\top(I-\gamma^{-2}  P_KDD^\top)^{-1},D^\top\big)$ is observable. Thus, $\Delta_K>0$ is full-rank. 
	By the necessary optimality condition $\nabla \cJ(K)=0$,  it follows that $K^*=(R+B^\top \tP_{K^*} B)^{-1}B^\top \tP_{K^*} A$ is the unique stationary point,  which is thus the unique global optimizer. This completes the proof of  Proposition    \ref{coro:opt_control_form_discrete}. 
 
\vspace{7pt}	
\noindent{\bf{Continuous-Time:}} 
  	\vspace{3pt}
  	
  	The proof is analogous to the above one, except that for proving the first half, one applies \cite[Theorem 4.8]{bacsar1995h} for the continuous-time setting. Note that $(A,C)$ being detectable is now  needed to apply \cite[Theorem 4.8]{bacsar1995h}. Note that the objective \eqref{equ:def_mixed_formulation_cont} is monotone in the eigenvalues of $P_K$, in that if $P_K\geq P_{K^*}$ for any $K\in\cK$, then $\cJ(K)\geq \cJ(K^*)$. For proving the second half,  by \cite[Lemma $3.18$ \emph{(iii)}]{zhou1996robust}, the solution to the Lyapunov equation \eqref{equ:Lambda_Lya_def} $\Lambda_K>0$, since $(A-BK+\gamma^{-2}DD^\top P_{K},D)$ is  controllable, or equivalently, the pair $((A-BK+\gamma^{-2}DD^\top P_{K})^\top,D^\top)$ is observable. Thus,  the necessary optimality condition $\nabla \cJ(K)=0$ yields  that  $K^*=R^{-1}B^\top P_{K^*}$ is the unique stationary point,  which is thus the unique global optimizer. This completes the proof of  Proposition    \ref{coro:opt_control_form_cont}. 
%\end{proof}
\hfill$\QED$

\subsection{Proof of Lemma \ref{lemma:cost_diff}}\label{sec:proof_lemma_cost_diff}
%\issue{XXXXXXXXXXXX 08.28 XXXXXXXXXXXXXX}
%\begin{proof}
%We need to modify the proof of Lemma 7 in \cite{fazel2018global} to show \eqref{eq:CDL1}. 

We start with 
the following helper lemma. 

\begin{lemma}\label{lemma:pt_p_relation}
Suppose that $K,K'\in\cK$. 
%\issue{$W^{-1}>\beta P_K$ and $W^{-1}>\beta P_{K'}$} 
%for some $K$ and $K'$. 
Then we have  that  $I-\gamma^{-2} P_{K'}DD^\top $ is invertible, and 
\begin{align}
\label{eq:linA1}
%(I-\beta P_{K'} W)^{-1} (P_K-\beta P_{K'} W P_{K'}) (I-\beta W P_{K'})^{-1}\le \tP_K.}\\
(I-\gamma^{-2} P_{K'}DD^\top)^{-1} (P_{K}-\gamma^{-2} P_{K'} DD^\top P_{K'}) (I-\gamma^{-2} DD^\top P_{K'})^{-1}\le \tP_{K}. 
\end{align}
\end{lemma}

\begin{proof}
First, since $K,K'\in\cK$, by Lemma \ref{lemma:discrete_bounded_real_lemma}, $I-\gamma^{-2} D^\top P_{K'}D>0$ is invertible. Thus, $\det(I-\gamma^{-2} P_{K'}DD^\top)=\det(I-\gamma^{-2} D^\top P_{K'}D)\neq 0$, namely, $I-\gamma^{-2} P_{K'}DD^\top$ is invertible. 
Then the desired fact is equivalent to
\$
&P_{K}-\gamma^{-2} P_{K'} DD^\top P_{K'}\le (I-\gamma^{-2} P_{K'}DD^\top)\tP_K (I-\gamma^{-2} DD^\top P_{K'})\\
&\qquad=\tP_K-\gamma^{-2} P_{K'}DD^\top\tP_K-\gamma^{-2} \tP_K DD^\top P_{K'}+\gamma^{-4} P_{K'} DD^\top \tP_K DD^\top P_{K'}
\$
which can be further simplified as
\small
\#
\label{eq:linA2}
(\tP_K-P_K)-\gamma^{-2} P_{K'}DD^\top\tP_K-\gamma^{-2} \tP_K DD^\top P_{K'}+\gamma^{-2} P_{K'} DD^\top P_{K'}+\gamma^{-4} P_{K'} DD^\top \tP_K DD^\top P_{K'}\ge 0. 
\#
\normalsize 
By $
\tilde{P}_K=(I-\gamma^{-2} P_K DD^\top)^{-1} P_K$ and \eqref{equ:symetrix_trash_1}, we have 
\$
&\gamma^{-2} P_{K'}DD^\top\tP_K=\gamma^{-2} P_{K'}DD^\top(I-\gamma^{-2} P_K DD^\top)^{-1} P_K =P_{K'}D(\gamma^2 I-D^\top P_K D)^{-1}D^\top P_{K}. 
%\\
%&=\gamma^{-2}  P_{K'}(W^{-1}-\gamma^{-2}P_K)^{-1} P_K.  
\$ 
Thus, it follows that
\$
&(\tP_K-P_K)-\gamma^{-2} P_{K'}DD^\top \tP_K-\gamma^{-2} \tP_K DD^\top  P_{K'}\\
&\quad=(P_K-P_{K'})D(\gamma^2I-D^\top P_K D)^{-1}D^\top (P_K-P_{K'})-P_{K'}D(\gamma^2I-D^\top P_K D)^{-1}D^\top P_{K'}. 
%\\
%&\gamma^{-2} (P_K-P_{K'})(W^{-1}-\gamma^{-2} P_K)^{-1}(P_K-P_{K'})-\gamma^{-2} P_{K'}(W^{-1}-\gamma^{-2} P_K)^{-1}P_{K'}.
\$ 
Therefore, \eqref{eq:linA2} is equivalent to
\$
&(P_K-P_{K'})D(\gamma^2I-D^\top P_K D)^{-1}D^\top (P_K-P_{K'})-P_{K'}D(\gamma^2I-D^\top P_K D)^{-1}D^\top P_{K'}\\
&\qquad\quad+\gamma^{-2} P_{K'} DD^\top  P_{K'}+\gamma^{-4} P_{K'} DD^\top \tP_K DD^\top P_{K'}\ge 0. 
\$
Given the fact $\gamma^2I>D^\top P_K D$ and another fact that 
\#\label{equ:trash411}
-P_{K'}D(\gamma^2I-D^\top P_K D)^{-1}D^\top P_{K'}+\gamma^{-2} P_{K'} DD^\top  P_{K'}+\gamma^{-4} P_{K'} DD^\top \tP_K DD^\top P_{K'}=0,
\# 
we know that the above inequality holds and hence our lemma is true.
To show that \eqref{equ:trash411} holds, 
it suffices to apply the matrix inversion lemma, i.e.,
\$
(\gamma^2I-D^\top P_K D)^{-1}=\gamma^{-2}I+\gamma^{-4}D^\top (-\gamma^{-2}P_KDD^\top+I)^{-1} P_KD=\gamma^{-2}I+\gamma^{-4}D^\top \tP_K D,
%\\(W^{-1}-\beta P_K)^{-1}=W+\beta W \tP_K W=W+\beta W(P^{-1}_K-\beta W)^{-1} W.
\$ 
where the first equation uses the matrix inversion lemma. 
This completes the proof.
\end{proof}

%We start by developing the following helper lemma.

%We need to extend the proofs of the helping lemmas in \cite{fazel2018global} using several important linear algebra facts. One such fact is that
%given $W^{-1}-\beta P_K>0$, we have $\tP_K=P_K+\beta P_K(W^{-1}-\beta P_K)^{-1}P_K >P_K$. 
%A more useful fact is the following result. 

By definition of  $\tP_{K}$ in  \eqref{equ:def_tP_K} and the Riccati  equation \eqref{equ:discret_riccati}, we have
\small 
\#
&P_{K'}=C^\top C+(K')^\top R K'+(A-BK')^\top \tP_{K'} (A-BK')\label{equ:trashlm421}\\
&=C^\top C+(K')^\top R K'+(A-BK')^\top (I-\gamma^{-2} P_{K'}DD^\top)^{-1} P_{K'} (A-BK')\notag\\
%&=Q+(K')^\top R K'+(A-BK')^\top (I-\beta P_{K'}W)^{-1} P_{K'} (I-\beta W P_{K'}) (I-\beta P_{K'}W)^{-\top}(A-BK')-P_K\notag\\
&= C^\top C+(K')^\top R K'+(A-BK')^\top (I-\gamma^{-2} P_{K'}DD^\top)^{-1} ( P_{K'}-\gamma^{-2} P_{K'}DD^\top P_{K'})(I-\gamma^{-2} P_{K'}DD^\top)^{-\top}(A-BK')\notag\\
&=C^\top C+(K')^\top R K'+(A-BK')^\top (I-\gamma^{-2} P_{K'}DD^\top)^{-1} (P_{K}-\gamma^{-2} P_{K'}DD^\top P_{K'}) (I-\gamma^{-2} P_{K'}DD^\top)^{-\top}(A-BK')\notag\\
&\qquad+(A-BK')^\top (I-\gamma^{-2} P_{K'}DD^\top)^{-1} (P_{K'}-P_K) (I-\gamma^{-2} P_{K'}DD^\top)^{-\top}(A-BK').\notag
\#
\normalsize
   By \eqref{eq:linA1} in Lemma \ref{lemma:pt_p_relation}, we further  have
\$
P_{K'}-P_K&\le C^\top C+(K')^\top R K'+(A-BK')^\top  \tP_{K} (A-BK')-P_K\\
&+(A-BK')^\top (I-\gamma^{-2} P_{K'}DD^\top)^{-1} (P_{K'}-P_K) (I-\gamma^{-2} P_{K'}DD^\top)^{-\top}(A-BK').
\$
By induction, we can apply the above inequality iteratively to show that
\#\label{eq:CDL1}
&P_{K'}-P_K \leq  \sum_{t\geq 0} [(A-BK')^\top(I-\gamma^{-2} P_{K'}DD^\top)^{-1}]^t\big[C^\top C+(K')^\top RK' \notag\\
&\qquad\qquad\qquad+ (A-BK')^\top \tP_K (A-BK')-P_K\big][(I-\gamma^{-2} P_{K'}DD^\top)^{-\top}(A-BK')]^t. 
\#
On the other hand, we have
\#
&C^\top C+(K')^\top R K'+(A-BK')^\top \tP_K(A-BK')-P_K\notag\\
&=C^\top C+(K'-K+K)^\top R (K'-K+K)+(A-BK-B(K'-K))^\top \tP_K(A-BK-B(K'-K))-P_K\notag\\
&=(K'-K)^\top\left((R+B^\top \tP_K B)K-B^\top \tP_K A\right)+\left((R+B^\top \tP_K B)K-B^\top \tP_K A\right)^\top (K'-K)\notag\\
&\qquad+(K'-K)(R+B^\top \tP_K B)(K'-K),\label{equ:Q_diff_to_K_diff}
\#
which can be substituted into \eqref{eq:CDL1} to obtain the upper  bound in \eqref{eq:CDL_upper}.

For the lower bound \eqref{eq:CDL_lower}, note that the conditions in Lemma \ref{lemma:pt_p_relation} also hold here when the roles of $K$ and $K'$ are interchanged. Thus, we have
\$
(I-\gamma^{-2} P_{K}DD^\top)^{-1} (P_{K'}-\gamma^{-2} P_{K} DD^\top P_{K}) (I-\gamma^{-2} DD^\top P_{K})^{-1}\le \tP_{K'},
\$
which gives a lower bound on the RHS of \eqref{equ:trashlm421} directly as
\small
\#\label{equ:trashlm422}
&P_{K'}-P_K=C^\top C+(K')^\top R K'+(A-BK')^\top \tP_{K'} (A-BK')-P_K\notag\\
&\geq C^\top C+(K')^\top R K'+(A-BK')^\top \big[(I-\gamma^{-2} P_{K}DD^\top)^{-1} (P_{K'}-\gamma^{-2} P_{K} DD^\top P_{K}) (I-\gamma^{-2} DD^\top P_{K})^{-1}\big] (A-BK')-P_K\notag\\
%&= C^\top C+(K')^\top R K'+(A-BK')^\top \big[(I-\gamma^{-2} P_{K} DD^\top)^{-1} (P_{K}-\gamma^{-2} P_{K} DD^\top P_{K}) (I-\gamma^{-2} DD^\top  P_{K})^{-1}\big] (A-BK')-P_K\notag\\
%&\qquad\quad +(A-BK')^\top \big[(I-\gamma^{-2} P_{K} DD^\top)^{-1} (P_{K'}-P_{K}) (I-\gamma^{-2} DD^\top P_{K})^{-1}\big] (A-BK')\notag\\
&= C^\top C+(K')^\top R K'+(A-BK')^\top \underbrace{\big[(I-\gamma^{-2} P_{K} DD^\top)^{-1} (P_{K}-\gamma^{-2} P_{K} DD^\top P_{K}) (I-\gamma^{-2} DD^\top P_{K})^{-1}\big]}_{\tP_K} (A-BK')-P_K\notag\\
&\quad\qquad +(A-BK')^\top \big[(I-\gamma^{-2} P_{K} DD^\top)^{-1} (P_{K'}-P_{K}) (I-\gamma^{-2} DD^\top P_{K})^{-1}\big] (A-BK').
\#
\normalsize
Continuing  unrolling the RHS of \eqref{equ:trashlm422} and  substituting into  \eqref{equ:Q_diff_to_K_diff}, we obtain the desired lower bound in \eqref{eq:CDL_lower}, which completes the proof.  
\hfill$\QED$
%\end{proof}

\subsection{Proof of Theorem  \ref{thm:stability_update_ct}}\label{sec:proof_thm_stability_update_ct}

We first argue that it suffices to find some $P>0$ for $K'$, such that
\#\label{equ:LMI_cond_ct_2}
(A-BK')^\top P+P(A-BK')+C^\top C+K'^\top RK'+\gamma^{-2} P^\top DD^\top P<0. 
\# 
Denote the LHS of \eqref{equ:LMI_cond_ct_2} by $-M<0$, then such a $P>0$ yields 
\$
(A-BK')^\top P+P(A-BK')=-M-C^\top C-K'^\top RK'-\gamma^{-2} P^\top DD^\top P\leq -M<0,
\$
which implies that $(A-BK')$ is Hurwitz, i.e., $K'$ is  stabilizing. Thus, Lemma \ref{lemma:cont_bounded_real_lemma} can be applied to $K'$, which shows that  \eqref{equ:LMI_cond_ct_2} is equivalent to $\|\cT(K')\|_{\infty}<{\gamma}$. This means that $K'\in\cK$. Hence, we will focus on finding such a $P>0$ hereafter.

%By Lemma \ref{lemma:cont_bounded_real_lemma}, proving $\|\cT(K')\|_{\infty}<{\gamma}$ is equivalent to proving that the following LMI holds for some $P>0$:
%\#\label{equ:LMI_cond_ct_2}
%(A-BK')^\top P+P(A-BK')+C^\top C+K'^\top RK'+\gamma^{-2} P^\top DD^\top P<0. 
%\# 

We  first show that   the Gauss-Newton update \eqref{eq:exact_gn_ct} with stepsize $\eta=1/2$ enables \eqref{equ:LMI_cond_ct_2} to hold.  
% such that  the  conditions  \eqref{equ:LMI_cond_2} and \eqref{equ:LMI_cond_3} hold.  
Specifically, we have 
\#\label{equ:K_prime_GN_ct}
K'=K-R^{-1}(RK-B^\top P_K)=R^{-1}B^\top P_K.
\#
By Lemma \ref{lemma:cont_bounded_real_lemma}, the closed-loop system  $A-BK+\gamma^{-2} DD^\top P_K$ is stable since $K\in\cK$.  Hence, the following Lyapunov equation admits a solution $\bar P>0$:
 \#\label{equ:def_bar_P_ct}
 (A-BK+\gamma^{-2} DD^\top P_K)^\top \bar P+\bar P(A-BK+\gamma^{-2} DD^\top P_K)=-I. 
 \# 
Hence, we choose $P=P_{K}+\alpha\bar P>0$ as the candidate for some $\alpha>0$. The LHS of \eqref{equ:LMI_cond_ct_2} now can be written as 
 \#\label{equ:LHS_sep_c} 
 &(A-BK')^\top P+P(A-BK')+C^\top C+K'^\top RK'+\gamma^{-2} P^\top DD^\top P\notag\\
 &\quad=\underbrace{[B(K-K')]^\top P+P[B(K-K')]+K'^\top R K'-K^\top R K}_{\circled{1}}\notag\\
 &\qquad+\underbrace{(A-BK)^\top P+P(A-BK)+C^\top C+K^\top RK+\gamma^{-2} P^\top DD^\top P}_{\circled{2}}. 
 \#
 We now need to show that there exists some $\alpha>0$ such that $\circled{1}+\circled{2}<0$.  
 By substituting in $K'$ from  \eqref{equ:K_prime_GN_ct}, we can write $\circled{1}$ as 
 \$
 \circled{1}&=-[B^\top P-RK ]^\top R^{-1}[B^\top P-RK]+PBR^{-1}B^\top P-P_KBR^{-1}B^\top P-PBR^{-1}B^\top P_K+P_KBR^{-1}B^\top P_K\notag\\
 &\leq  PBR^{-1}B^\top P-P_KBR^{-1}B^\top P-PBR^{-1}B^\top P_K+P_KBR^{-1}B^\top P_K\notag\\
 &\leq  (P-P_K)BR^{-1}B^\top (P-P_K)=\alpha^2\bar PBR^{-1}B^\top \bar P=o(\alpha). 
 \$
 Moreover,  
 $\circled{2}$ can be written as 
 \#\label{equ:circle_2_res_cont}
 \circled{2}&=(A-BK)^\top P_K+P_K(A-BK)+C^\top C+K^\top RK+\gamma^{-2} P_K^\top DD^\top P_K\notag\\
 &\quad+\gamma^{-2} P^\top DD^\top P-\gamma^{-2} P_K^\top DD^\top P_K+\alpha(A-BK)^\top \bar P+\alpha \bar P(A-BK)\notag\\
% &=\alpha\gamma^{-2} \bar P^\top DD^\top P_K+\alpha\gamma^{-2} P_K^\top DD^\top\bar P+\alpha^2\gamma^{-2} \bar P^\top DD^\top\bar P+\alpha(A-BK)^\top \bar P+\alpha \bar P(A-BK)\notag\\
 &=-\alpha I+\alpha^2\gamma^{-2} \bar P^\top DD^\top\bar P=-\alpha I+o(\alpha), 
 \# 
 where the second equation has used the Riccati  equation \eqref{equ:cont_riccati}, and the third one has used the definition of $\bar P$ in \eqref{equ:def_bar_P_ct}. 
 Thus, there exists a small enough $\alpha>0$, such that $\circled{1}+\circled{2}<0$, namely, there exists some $P>0$ such that  \eqref{equ:LMI_cond_ct_2} holds for $K'$ obtained from \eqref{eq:exact_gn_ct} with stepsize $\eta=1/2$. 
 On the other hand,  since  such a $P$ makes  $\circled{2}<0$, it also makes the LMI  \eqref{equ:LMI_cond_ct_2} hold for $K$, i.e., 
\#\label{equ:LMI_cond_ct_3}
  (A-BK)^\top P+P(A-BK)+C^\top C+K^\top RK+\gamma^{-2} P^\top DD^\top P<0. 
\# 
By linearly combining \eqref{equ:LMI_cond_ct_2} and \eqref{equ:LMI_cond_ct_3} and the convexity of quadratic functions,  the LMI \eqref{equ:LMI_cond_ct_3} also holds for $K_{\eta}=K+2\eta(K'-K)=(1-2\eta)K+2\eta K'$ for any $\eta\in[0,1/2]$. 
%\#\label{equ:LMI_cond_ct_4}
%0>&2\eta\cdot[(A-BK')^\top P+P(A-BK')+C^\top C+K'^\top RK'+\gamma^{-2} P^\top DD^\top P]\notag\\
%&\quad+(1-2\eta)\cdot[(A-BK)^\top P+P(A-BK)+C^\top C+K^\top RK+\gamma^{-2} P^\top DD^\top P]\notag\\
%\geq&(A-BK_{\eta})^\top P+P(A-BK_{\eta})+C^\top C+K_{\eta}^\top RK_{\eta}+\gamma^{-2} P^\top DD^\top P, 
%\#
%where for any $\eta\in[0,1/2]$, we define  $K_{\eta}=K+2\eta(K'-K)=(1-2\eta)K+2\eta K'$ to be the interpolation between $K$ and $K'$, and the second inequality follows by the convexity of $K^\top R K$. 
% \eqref{equ:LMI_cond_ct_4} thus shows that for any stepsize $\eta\in[0,1/2]$, $K_\eta$ that lies between $K$ and $K'$ satisfies the LMI \eqref{equ:LMI_cond_ct_3}.   

 Similar techniques are used  for the natural PG update \eqref{eq:exact_npg_ct}.  
 Recall   that 
 \#\label{equ:restate_exact_npg_ct}
 K'=K-2\eta (RK-B^\top P_K).  
 \#
 As before, we  choose $P=P_K+\alpha\bar P$ for some $\alpha>0$. Then,   the term $\circled{2}$ in \eqref{equ:LHS_sep_c} is still $-\alpha I+o(\alpha)$. The term $\circled{1}$ in \eqref{equ:LHS_sep_c} can be written as 
 \#\label{equ:natural_ct_trash_1} 
\circled{1}&=(K'-K)^\top R(K'-R^{-1}B^\top P)+(K-R^{-1}B^\top P)^\top R(K'-K)\notag\\
%&=-2\eta(RK-B^\top P_K)^\top R(K'-R^{-1}B^\top P)-2\eta(K-R^{-1}B^\top P)^\top R(RK-B^\top P_K)\notag\\
%&=-2\eta(RK-B^\top P_K)^\top (RK-B^\top P)-2\eta(RK-B^\top P)^\top (RK-B^\top P_K)\notag\\
%&\quad+4\eta^2 (RK-B^\top P_K)^\top R (RK-B^\top P_K)\notag\\
&=-4\eta(RK-B^\top P_K)^\top (RK-B^\top P_K)+4\eta^2 (RK-B^\top P_K)^\top R (RK-B^\top P_K)\notag\\
&\quad+2\alpha\eta(RK-B^\top P_K)^\top(B^\top \bar P)+2\alpha\eta(B^\top \bar P)^\top (RK-B^\top P_K)\notag\\
&\leq -4\eta(RK-B^\top P_K)^\top (RK-B^\top P_K)+4\eta^2 (RK-B^\top P_K)^\top R (RK-B^\top P_K)\notag\\
&\quad+2\eta(RK-B^\top P_K)^\top(RK-B^\top P_K)+2\alpha^2\eta (B^\top \bar P)^\top(B^\top \bar P),
\#
where we have used the definition of $P$, and the inequality is due to the fact that 
\$
&\alpha(RK-B^\top P_K)^\top(B^\top \bar P)+\alpha(B^\top \bar P)^\top (RK-B^\top P_K)\\
&\quad\leq (RK-B^\top P_K)^\top(RK-B^\top P_K)+\alpha^2 (B^\top \bar P)^\top(B^\top \bar P). 
\$
In addition, if the stepsize  $\eta\leq1/(2\|R\|)$, 
then $\circled{1}$ can be further  bounded from \eqref{equ:natural_ct_trash_1} that
 \$
\circled{1}
&\leq -2\eta(RK-B^\top P_K)^\top (RK-B^\top P_K)+2\eta(RK-B^\top P_K)^\top(RK-B^\top P_K)+2\alpha^2\eta (B^\top \bar P)^\top(B^\top \bar P)\notag\\
&=2\alpha^2\eta (B^\top \bar P)^\top(B^\top \bar P)=o(\alpha). 
\$
As a result, there exists small enough $\alpha>0$ (and thus $P$) such that $\circled{1}+\circled{2}<0$. 
%In words, there exists some $P>0$ such that  
Hence, \eqref{equ:LMI_cond_ct_2} holds for $K'$ obtained from \eqref{equ:restate_exact_npg_ct} with   $\eta\leq 1/(2\|R\|)$. 
By Lemma \ref{lemma:cont_bounded_real_lemma}, this proves the first argument that $\|\cT(K')\|_{\infty}<{\gamma}$. 
Such a $K'$ also ensures the existence of the stabilizing solution $P_{K'}\geq 0$ to the Riccati equation \eqref{equ:cont_riccati}.  
%$\|\cT(K')\|_{\infty}<\sqrt{1/\gamma}$, 
This completes the proof. 
\hfill$\QED$

%\issue{to add: continuous-time results' proof}

\subsection{Proof of Theorem \ref{theorem:global_exact_conv_ct}}\label{sec:proof_theorem:global_exact_conv_ct}

We first introduce
% several helper lemmas that are used in the ensuing analysis.
%We start with 
the continuous-time \emph{cost difference lemma} that establishes the relationship between  $P_{K'}-P_K$ and $K'-K$. 
 
\begin{lemma}[Continuous-Time Cost Difference Lemma]\label{lemma:cost_diff_ct}
Suppose that both $K,K'\in\cK$. 
%Thus, both matrices $A-BK'+\gamma^{-2} DD^\top  P_{K'}$ and $A-BK+\gamma^{-2} DD^\top  P_{K}$ are Hurwitz \cite{zhou1996robust}. 
Then,  we have the  following upper bound:
\#
P_{K'}-P_K&\le \int_{0}^\infty  \eb^{(A-BK'+\gamma^{-2} DD^\top P_{K'})^\top \tau}\cdot [(RK-B^\top P_K)^\top(K'-K) +(K'-K)^\top (RK-B^\top P_K)\notag\\&\quad+(K'-K)^\top R(K'-K)]\cdot \eb^{(A-BK'+\gamma^{-2} DD^\top P_{K'})\tau}d\tau. \label{eq:CDL_upper_ct}
\# 
%where recall $E_K$ is defined in \eqref{equ:def_mu_Ek}.
%have finite costs and thus $P_K$ and $P_{K'}$ exist. Also, suppose $W^{-1}>\gamma P_K$, $W^{-1}>\gamma P_{K'}$, and thus both $I-\gamma P_{K} W$ and {$I-\gamma P_{K'} W$ are  invertible}. 
If additionally the matrix   $A-BK'+\gamma^{-2} DD^\top P_{K}$ is Hurwitz,  
then we  have the   lower bound:
%\#\label{eq:CDL1}
%&P_{K'}-P_K \leq  \sum_{t=0} [(A-BK')^\top(I-\gamma P_{K'}W)^{-1}]^t\big[Q+(K')^\top RK' \notag\\
%&\qquad\qquad\qquad+ (A-BK')^\top \tP_K (A-BK')-P_K\big][(I-\gamma P_{K'}W)^{-\top}(A-BK')]^t. 
%\#
%In addition, we have
%\begin{align}
%\label{eq:CDL2}
%&Q+(K')^\top R K'+(A-BK')^\top \tP_K(A-BK')-P_K\notag\\
%&\quad=2(K'-K)^\top E_K+(K'-K)^\top (R+B^\top \tP_{K} B)(K'-K),
%\end{align}
%where recall $E_K$ is defined in \eqref{equ:def_mu_Ek}. 
%Combining the above two inequalities, we have 
\#
P_{K'}-P_K&\ge \int_{0}^\infty \eb^{(A-BK'+\gamma^{-2} DD^\top P_{K})^\top \tau}\cdot [(RK-B^\top P_K)^\top(K'-K) +(K'-K)^\top (RK-B^\top P_K)\notag\\&\quad+(K'-K)^\top R(K'-K)]\cdot \eb^{(A-BK'+\gamma^{-2} DD^\top P_{K})\tau}d\tau. \label{eq:CDL_lower_ct}
\# 
\end{lemma}

\begin{proof}
	First, by Lemma \ref{lemma:cont_bounded_real_lemma},  the matrices $A-BK'+\gamma^{-2} DD^\top P_{K'}$ and $A-BK+\gamma^{-2} DD^\top P_{K}$ are both Hurwitz. 
%	is also a Riccati equation. Hence, the matrices $A-BK'+\gamma^{-2} DD^\top P_{K'}$ and $A-BK+\gamma^{-2} DD^\top P_{K}$ are both closed-loop controllers, which are thus Hurwitz \cite{zhou1996robust}.
	 As a result, the integral on the RHS of \eqref{eq:CDL_upper_ct} is well defined. By subtracting two Riccati  equations \eqref{equ:cont_riccati}  corresponding to $K'$ and $K$, we have
\# 
&(A-BK'+\gamma^{-2} DD^\top P_{K'})^\top \Delta_P	+\Delta_P(A-BK'+\gamma^{-2} DD^\top P_{K'})+(RK-B^\top P_K)^\top(K'-K) \notag\\
&\quad+(K'-K)^\top (RK-B^\top P_K)+(K'-K)^\top R(K'-K)-\gamma^{-2} \Delta_P^\top  DD^\top\Delta_P \label{equ:cost_diff_pf_trash_1}\\
&=(A-BK'+\gamma^{-2} DD^\top P_K)^\top \Delta_P	+\Delta_P(A-BK'+\gamma^{-2} DD^\top P_K)+(RK-B^\top P_K)^\top(K'-K) \notag\\
&\quad+(K'-K)^\top (RK-B^\top P_K)+(K'-K)^\top R(K'-K)+\gamma^{-2} \Delta_P^\top  DD^\top\Delta_P=0, 
\label{equ:cost_diff_pf_trash_2}
\#
where we let $\Delta_P= P_{K'}-P_K$, and the   relationship follows from the facts below:
\$
&K'^\top R K'-K^\top RK= (K'-K)R(K'-K)+K^\top R(K'-K)+(K'-K)^\top RK\\
&\gamma^{-2} P_{K'}DD^\top P_{K'}-\gamma^{-2} P_{K}DD^\top P_{K}=\gamma^{-2} P_KDD^\top\Delta_P+\gamma^{-2} \Delta_P DD^\top P_K+\gamma^{-2} \Delta_P^\top  DD^\top \Delta_P\\ 
&\quad=\gamma^{-2} P_{K'}DD^\top\Delta_P+\gamma^{-2} \Delta_P DD^\top P_{K'}-\gamma^{-2} \Delta_P^\top  DD^\top\Delta_P. 
\$

From \eqref{equ:cost_diff_pf_trash_1}, we know that $\Delta_P\leq \hat{\Delta}_P$, where $\hat{\Delta}_P$ is the solution to the following Lyapunov equation 
\#\label{equ:cost_diff_pf_trash_3} 
&(A-BK'+\gamma^{-2} DD^\top P_{K'})^\top \hat{\Delta}_P	+\hat{\Delta}_P(A-BK'+\gamma^{-2} DD^\top P_{K'})+(RK-B^\top P_K)^\top(K'-K) \notag\\
&\quad+(K'-K)^\top (RK-B^\top P_K)+(K'-K)^\top R(K'-K)=0.
\#
This is due to the fact that \eqref{equ:cost_diff_pf_trash_1} subtracted from   \eqref{equ:cost_diff_pf_trash_3} yields
\#\label{equ:cost_diff_pf_trash_4}
&(A-BK'+\gamma^{-2} DD^\top P_{K'})^\top (\hat{\Delta}_P-\Delta_P)	+(\hat{\Delta}_P-\Delta_P)(A-BK'+\gamma^{-2} DD^\top P_{K'})+\gamma^{-2}  \Delta_P^\top  DD^\top \Delta_P=0.
\# 
Since $A-BK'+\gamma^{-2} DD^\top P_{K'}$ is stabilizing, $\hat{\Delta}_P-\Delta_P$ can be viewed as the unique solution to this Lyapunov equation \eqref{equ:cost_diff_pf_trash_4}. Moreover, since $\gamma^{-2} \Delta_P^\top  DD^\top\Delta_P\geq 0$, we obtain that $\hat{\Delta}_P\geq \Delta_P$. Note that the solution $\hat{\Delta}_P$ to \eqref{equ:cost_diff_pf_trash_3} has the form on the RHS of \eqref{eq:CDL_upper_ct}, which completes the proof for the upper bound. 

Similarly, if the matrix   $A-BK'+\gamma^{-2} DD^\top P_{K}$ is Hurwitz, then the RHS of \eqref{eq:CDL_lower_ct} is well defined, and so is the solution $\tilde \Delta_P$ to the following Lyapunov equation
\#\label{equ:cost_diff_pf_trash_5}
&(A-BK'+\gamma^{-2} DD^\top P_{K})^\top \tilde{\Delta}_P	+\tilde{\Delta}_P(A-BK'+\gamma^{-2} DD^\top P_{K})+(RK-B^\top P_K)^\top(K'-K) \notag\\
&\quad+(K'-K)^\top (RK-B^\top P_K)+(K'-K)^\top R(K'-K)=0.
\#
Note that $\tilde \Delta_P$ has the form of the RHS of \eqref{eq:CDL_lower_ct}. Subtracting \eqref{equ:cost_diff_pf_trash_5} from  \eqref{equ:cost_diff_pf_trash_2} yields the Lyapunov equation
\#\label{equ:cost_diff_pf_trash_6}
&(A-BK'+\gamma^{-2} DD^\top P_{K})^\top (\Delta_P-\tilde \Delta_P)	+(\Delta_P-\tilde \Delta_P)(A-BK'+\gamma^{-2} DD^\top P_{K})+\gamma^{-2} \Delta_P^\top  DD^\top \Delta_P=0.
\#
Hence, $\gamma^{-2} \Delta_P^\top  DD^\top \Delta_P\geq 0$ implies that the unique solution to \eqref{equ:cost_diff_pf_trash_6}, $\Delta_P-\tilde \Delta_P\geq 0$, which completes the proof of the lower bound. 
\end{proof}

As in the discrete-time setting, Lemma \ref{lemma:cost_diff_ct} also characterizes the ``Almost Smoothness'' of $P_K$ with respect to $	K$ \citep{fazel2018global}.  Now we are ready to analyze the updates \eqref{eq:exact_gn_ct} and \eqref{eq:exact_npg_ct}.  

We start by the following helper lemma that lower-bounds the solution to the Lyapunov equation,
\#\label{equ:lyapunov_DD}
(A-BK+\gamma^{-2} DD^\top P_K)\cM_K+\cM_K(A-BK+\gamma^{-2} DD^\top P_K)^\top +M =0, 
\#
 for any  matrix $M>0$. 

\begin{lemma}\label{lemma:lower_bnd_cont_lyap}
	{Suppose that $K\in\cK$, and there exists a constant $\cC_{K}>0$ such that $K^\top R K\leq \cC_{K}\cdot I$.} 
	Let $\cM_K>0$ be the unique solution to the Lyapunov equation  \eqref{equ:lyapunov_DD} 
	 with $M>0$. Then,  
	{
	\$
	\cM_K\geq \frac{\sigma_{\min}(M)}{4\omega_K}\cdot I,
	\$
	where  $\omega_K>0$ is defined as 
		\small
	\#\label{equ:def_omega_K}
	\omega_K:=\max\bigg\{\lambda_{\max}\big(-P_K^{-1/2}C^\top C P_K^{-1/2}+\gamma^{-2} P_K^{1/2}DD^\top P_K^{1/2})\big),-\lambda_{\min}\big(P_K^{-1/2}(-C^\top C-\cC_{K}\cdot I)P_K^{-1/2}\big)\bigg\}. 
	\#
	\normalsize
	}
\end{lemma}
\begin{proof}
	By \cite{shapiro1974lyapunov}, the solution $\cM_K$ to the Lyapunov equation satisfies 
	\#\label{equ:lower_bnd_cont_lyap_trash_1}
	\cM_K\geq \frac{\sigma_{\min}(M)}{2\|A-BK+\gamma^{-2} DD^\top P_K\|} \cdot I=\frac{\sigma_{\min}(M)}{2\big\|P_K^{1/2}(A-BK+\gamma^{-2} DD^\top P_K)P_K^{-1/2}\big\|} \cdot I.
	\#
	On the other hand,  multiplying $P^{-1/2}_{K}$  on both sides of   the Riccati  equation  \eqref{equ:cont_riccati} yields  
%	\small
	\$
	&P_K^{-1/2}(A-BK+\gamma^{-2} DD^\top P_K)^\top P_K^{1/2}+P_K^{1/2}(A-BK+\gamma^{-2} DD^\top  P_K)P_K^{-1/2}\\
	&\quad=-P_K^{-1/2}(C^\top C+K^\top RK-\gamma^{-2} P_KDD^\top  P_K)P_K^{-1/2},
	\$
	\normalsize
	which further implies that
	\#
	&\lambda_{\max}\big(P_K^{-1/2}(A-BK+\gamma^{-2} DD^\top P_K)^\top P_K^{1/2}+P_K^{1/2}(A-BK+\gamma^{-2} DD^\top P_K)P_K^{-1/2}\big)\label{equ:lower_bnd_cont_lyap_trash_2}\\
	&\quad \leq \lambda_{\max}\big(P_K^{-1/2}(-C^\top C+\gamma^{-2} P_KDD^\top P_K)P_K^{-1/2}\big)=\lambda_{\max}\big(-P_K^{-1/2}C^\top CP_K^{-1/2}+\gamma^{-2} P_K^{1/2}DD^\top P_K^{1/2})\big)\notag
%	\leq \lambda_{\max}\big(-P_K^{-1/2}QP_K^{-1/2}+I\big)
	\\
	&\lambda_{\min}\big(P_K^{-1/2}(A-BK+\gamma^{-2} DD^\top P_K)^\top P_K^{1/2}+P_K^{1/2}(A-BK+\gamma^{-2} DD^\top P_K)P_K^{-1/2}\big)\notag\\
	&\quad \geq \lambda_{\min}\big(P_K^{-1/2}(-C^\top C-\cC_{K}\cdot I)P_K^{-1/2}\big),\label{equ:lower_bnd_cont_lyap_trash_3}
	\#
	by taking the largest and smallest eigenvalues on both sides, respectively.  Note that in \eqref{equ:lower_bnd_cont_lyap_trash_2}, the term $K^\top R K$ is dropped, while in 
	\eqref{equ:lower_bnd_cont_lyap_trash_3}, $K^\top R K$ is replaced by $\cC_{K}\cdot I$, and $\gamma^{-2} P_KDD^\top P_K$ is dropped. 
Let $\omega(X)$ be the \emph{numerical radius} of a   matrix $X$   defined as
\$
\omega(X):=\max\{\lambda_{\max}(X+X^\top)/2,-\lambda_{\min}(X+X^\top)/2\}. 
\$
Then \eqref{equ:lower_bnd_cont_lyap_trash_2} and \eqref{equ:lower_bnd_cont_lyap_trash_3} together yield 
\#\label{equ:lower_bnd_cont_lyap_trash_4}
&\omega\big(P_K^{-1/2}(A-BK+\gamma^{-2} DD^\top P_K)^\top P_K^{1/2}+P_K^{1/2}(A-BK+\gamma^{-2} DD^\top P_K)P_K^{-1/2}\big)\\
&\quad\leq \max\bigg\{\lambda_{\max}\big(-P_K^{-1/2}C^\top CP_K^{-1/2}+\gamma^{-2} P_K^{1/2}DD^\top P_K^{1/2})\big),-\lambda_{\min}\big(P_K^{-1/2}(-C^\top C-\cC_{K}\cdot I)P_K^{-1/2}\big)\bigg\}. \notag
\#
By the relationship between operator norm and numerical radius \citep{shebrawi2009numerical}, we also have
	\#\label{equ:lower_bnd_cont_lyap_trash_5}
	&\big\|P_K^{1/2}(A-BK+\gamma^{-2} DD^\top P_K)P_K^{-1/2}\big\|\notag\\
	&\quad\leq 2\omega\big(P_K^{-1/2}(A-BK+\gamma^{-2} DD^\top P_K)^\top P_K^{1/2}+P_K^{1/2}(A-BK+\gamma^{-2} DD^\top P_K)P_K^{-1/2}\big).
	\# 
Combining\eqref{equ:lower_bnd_cont_lyap_trash_1},   \eqref{equ:lower_bnd_cont_lyap_trash_4}, and \eqref{equ:lower_bnd_cont_lyap_trash_5} proves the desired result.
\end{proof}

\vspace{10pt}
\noindent{\textbf{Gauss-Newton:}}
\vspace{4pt} 
   
Recall that  the Gauss-Newton update has the form  $K'=K-2\eta(K-R^{-1}B^\top P_K)$. 
By Theorem \ref{thm:stability_update_ct}, $K'$ also lies in $\cK$ if $\eta\leq 1/2$.  
Then, by the upper bound \eqref{eq:CDL_upper_ct}, for any $\eta\in[0,1/2]$, 
\begin{align}\label{equ:monotone_p_ct}
&P_{K'}-P_K\le(-4\eta+2\eta)\int_{0}^\infty \eb^{(A-BK'+\gamma^{-2} DD^\top  P_{K'})^\top \tau}\left[ (RK-B^\top P_K)^\top R^{-1}(RK-B^\top P_K) \right]\notag\\
&\qquad\qquad\qquad\qquad\qquad\qquad\cdot\eb^{(A-BK'+\gamma^{-2} DD^\top  P_{K'})\tau}d\tau\le 0, 
\end{align} 
which implies the monotonic decrease of $P_K$ (matrix-wise) along the update.  Since $P_K$ is lower-bounded, such a monotonic sequence of $\{P_{K_n}\}$ along the iterations must converge to some $P_{K_{\infty}}\in\cK$. Now we show this $P_{K_{\infty}}$ is indeed $P_{K^*}$. 
Multiplying by any $M> 0$ on both sides of \eqref{equ:monotone_p_ct} and 
taking the trace  further implies that 
\begin{align}\label{equ:p_trace_upper_bnd_gn_ct}
&\tr(P_{K'}M)-\tr(P_KM)\notag\\
&\quad\le-2\eta\tr\left[ (RK-B^\top P_K)^\top R^{-1}(RK-B^\top P_K) \int_{0}^\infty \eb^{(A-BK'+\gamma^{-2} DD^\top  P_{K'}) \tau}M\eb^{(A-BK'+\gamma^{-2} DD^\top  P_{K'})^\top\tau}d\tau\right]\notag\\
%&\qquad\qquad\qquad\qquad\qquad\qquad\cdot\eb^{(A-BK'+\gamma^{-2} DD^\top  P_{K'})\tau}d\tau\notag\\
%&\quad \leq -2\eta \tr\bigg\{\left[ E_K^\top(R+B^\top \tP_K B)^{-1}E_K \right]\cdot\notag\\
%&\qquad\qquad\qquad\qquad
%%\underbrace{
%\sum_{t\ge 0} [(I-\gamma P_{K'}W)^{-\top}(A-BK')]^t W[(A-BK')^\top(I-\gamma P_{K'}W)^{-1}]^t
%%}_{\cM_{K',K'}}
%\bigg\}\notag\\
&\quad\leq\frac{-2\eta\sigma_{\min}(M)}{4\omega_K}
\tr\big[(RK-B^\top P_K)^\top R^{-1}(RK-B^\top P_K)\big]\notag\\
%\leq-2\eta\sigma_{\min}(W)\tr\big[E_K^\top(R+B^\top \tP_K B)^{-1}E_K \big]
&\quad\leq \frac{-\eta\sigma_{\min}(M)}{2\omega_K\sigma_{\max}(R)}
\tr\big[(RK-B^\top P_K)^\top (RK-B^\top P_K)\big]
%\notag\\
%&\quad \leq \frac{-2\eta\sigma_{\min}(W)}{\sigma_{\max}(R+B^\top \tP_{K_0} B)}\tr(E_K^\top E_K)
,
\end{align}
where 
%$\cM_{K',K'}\geq 0$ is a nonnegative definite matrix depending on $K'$, and 
the second inequality follows by applying Lemma \ref{lemma:lower_bnd_cont_lyap}, and $\omega_K$ is as defined in \eqref{equ:def_omega_K}. 
{Since for any finite $N> 0$, there exists some constant $\cC_{K}^N>0$ such that  $K_{n}^\top R K_{n}\leq \cC_{K}^N\cdot I$ for all $K_{n}$ with $n\leq N-1$.
By definition,  
$\omega_K$ can be uniformly upper bounded along the iteration as}
\small
\#\label{equ:def_bar_omega}
\overline{\omega}_{\cK}:=\max\bigg\{\lambda_{\max}\big(-P_{K_0}^{-1/2}C^\top CP_{K_0}^{-1/2}+\gamma^{-2} P_{K_0}^{1/2}DD^\top P_{K_0}^{1/2})\big),~-\lambda_{\min}\big(P_{K_{\infty}}^{-1/2}(-C^\top C-\cC_{K}^N\cdot I)P_{K_{\infty}}^{-1/2}\big)\bigg\},
\#
\normalsize
which is due to  the facts that the first and second terms in the $\max$ operator are increasing and decreasing with respect to $P_K$, respectively, and $P_{K_0}\geq P_{K_n}\geq P_{K_{\infty}}$ holds for all $n\geq 0$ from \eqref{equ:monotone_p_ct}.

From iterations $n=0$ to $N-1$, replacing $\omega_K$ by $\overline{\omega}_{\cK}$ and $M$ by identify  matrix $I$,  summing over both sides of \eqref{equ:p_trace_upper_bnd_gn_ct} and dividing by $N$, we  have
\$
\frac{1}{N}\sum_{n=0}^{N-1}\tr\big[(RK_{n}-B^\top P_{K_{n}})^\top (RK_{n}-B^\top P_{K_{n}})\big]\leq \frac{2\overline{\omega}_{\cK}\sigma_{\max}(R)\cdot \big[\tr(P_{K_0})-\tr(P_{K_{\infty}})\big]}{\eta\cdot N}.
\$
This shows that the sequence $\{\|RK_n-B^\top P_{K_n}\|_F^2\}$ converges to zero, namely, the sequence $\{K_n\}$ converges to the stationary point $K$ such that  $RK-B^\top P_{K}=0$, with sublinear $O(1/N)$ rate. By Proposition \ref{coro:opt_control_form_cont}, this is in fact towards the global optimum $K^*$.   
%Additionally, if $(A-BK+\gamma^{-2}DD^\top P_{K},D)$ is controllable at $K$, then by Proposition  \ref{coro:opt_control_form_cont}, such  convergence is towards 
%stationary point $K$ is unique and is indeed  the  
%unique \emph{global} 
%the global 
%optimizer  $K^*$.  

%Since $RK-B^\top P_K=\bm{0}$ gives the unique optimal solution $K^*=R^{-1}B^\top P_{K^*}$, this shows that the sequence 
%$\{K_n\}$ converges to the optimal control gain $K^*$ with sublinear $O(1/N)$ rate in the sense of the convergence of $\tr[(RK-B^\top P_K)^\top (RK-B^\top P_K)]$. 

\vspace{10pt}
\noindent{\textbf{Natural  Policy Gradient:}}
\vspace{4pt} 
 
Recall that the natural PG update follows $K'=K-2\eta E_K$. 
By Theorem \ref{thm:stability_update_ct}, $K'$ also lies in $\cK$ if $\eta\leq 1/(2\|R\|)$.  
Then, by  the upper bound in \eqref{eq:CDL_upper_ct}, we also have the monotonic decrease of $P_K$ (matrix-wise) along the iterations as 
\small
\begin{align*}
&P_{K'}-P_K
%\le\sum_{t\ge 0} [(A-BK')^\top(I-\gamma P_{K'}W)^{-1}]^t\left[-4\eta E_K^\top E_K+4\eta^2E_K^\top(R+B^\top \tP_K B)^{-1}E_K \right]\notag\\
%&\qquad\qquad\qquad\qquad\qquad\qquad\cdot[(I-\gamma P_{K'}W)^{-\top}(A-BK')]^t\notag\\
\le(-4\eta+2\eta)\int_{0}^\infty \eb^{(A-BK'+\gamma^{-2} DD^\top  P_{K'})^\top \tau}\left[ (RK-B^\top P_K)^\top  (RK-B^\top P_K) \right] \eb^{(A-BK'+\gamma^{-2} DD^\top  P_{K'})\tau}d\tau\\
&\quad\le 0.   
%\notag\\
%&\quad\le 0, 
\end{align*} 
\normalsize
%which also implies the monotonic decrease of $P_{K}$ along the update.  
As in \eqref{equ:p_trace_upper_bnd_gn_ct}, taking the trace of both sides   yields
\begin{align}\label{equ:p_trace_upper_bnd_ng_ct}
&\tr(P_{K'}M)-\tr(P_KM)
%\le(-4\eta+4\eta^2)\tr\bigg\{\sum_{t\ge 0} [(A-BK')^\top(I-\gamma P_{K'}W)^{-1}]^t\left[ E_K^\top(R+B^\top \tP_K B)^{-1}E_K \right]\notag\\
%&\qquad\qquad\qquad\qquad\qquad\qquad\cdot[(I-\gamma P_{K'}W)^{-\top}(A-BK')]^t W\bigg\}\notag\\
%&\quad \leq -2\eta \tr\bigg\{\left[ E_K^\top(R+B^\top \tP_K B)^{-1}E_K \right]\cdot\notag\\
%&\qquad\qquad\qquad\qquad
%%\underbrace{
%\sum_{t\ge 0} [(I-\gamma P_{K'}W)^{-\top}(A-BK')]^t W[(A-BK')^\top(I-\gamma P_{K'}W)^{-1}]^t
%%}_{\cM_{K',K'}}
%\bigg\}\notag\\
\leq\frac{-\eta\sigma_{\min}(M)}{2\overline{\omega}_{\cK}}
\tr\big[(RK-B^\top P_K)^\top (RK-B^\top P_K)\big]
%\leq-2\eta\sigma_{\min}(W)\tr\big[E_K^\top(R+B^\top \tP_K B)^{-1}E_K \big]
%\leq {-2\eta\sigma_{\min}(W)}\tr(E_K^\top E_K)\notag\\
%&\quad \leq \frac{-2\eta\sigma_{\min}(W)}{\sigma_{\max}(R+B^\top \tP_{K_0} B)}\tr(E_K^\top E_K)
,
\end{align} 
for any $M>0$, 
where $\overline{\omega}_{\cK}$ is also defined as in \eqref{equ:def_bar_omega}.  Summing over both sides of \eqref{equ:p_trace_upper_bnd_ng_ct} from $n=0$ to $n=N-1$ gives 
\$
\frac{1}{N}\sum_{n=0}^{N-1}\tr\big[(RK_{n}-B^\top P_{K_{n}})^\top (RK_{n}-B^\top P_{K_{n}})\big]\leq \frac{2\overline{\omega}_{\cK}\cdot \big[\tr(P_{K_0})-\tr(P_{K_{\infty}})\big]}{\eta\cdot N},
\$
where we have replaced $M$ in \eqref{equ:p_trace_upper_bnd_ng_ct} by $I$. 
This   completes the proof. 
\hfill$\QED$

\subsection{Proof of Theorem \ref{theorem:local_exact_conv_ct}}\label{sec:proof_theorem:local_exact_conv_ct}

By the lower bound \eqref{eq:CDL_lower_ct} from Lemma \ref{lemma:cost_diff_ct} and completion of   squares, we have that 
\small
\#\label{equ:p_diff_lower_bnd_ct}
&P_{K'}-P_K\geq \int_{0}^\infty \eb^{(A-BK'+\gamma^{-2} DD^\top  P_{K})^\top \tau}\cdot [(RK-B^\top P_K)^\top(K'-K) +(K'-K)^\top (RK-B^\top P_K)\notag\\
&\qquad\qquad\qquad\qquad+(K'-K)^\top R(K'-K)]\cdot \eb^{(A-BK'+\gamma^{-2} DD^\top  P_{K})\tau}d\tau \notag\\
&\quad\geq \int_{0}^\infty \eb^{(A-BK'+\gamma^{-2} DD^\top  P_{K})^\top \tau}\cdot \big[-(RK-B^\top P_K)^\top R^{-1}(RK-B^\top P_K)\big]\cdot \eb^{(A-BK'+\gamma^{-2} DD^\top  P_{K})\tau}d\tau.
\#
\normalsize
%where the second inequality follows by completing   squares. 
%
%By completing the squares, we have
%\#\label{equ:p_diff_lower_bnd_trash}
%&-2(K-K')^\top E_K+(K'-K)^\top (R+B^\top \tP_{K} B)(K'-K)\notag\\
%&=[K'-K+(R+B^\top \tP_{K} B)^{-1}E_K]^\top (R+B^\top \tP_{K} B) [K'-K+(R+B^\top \tP_{K} B)^{-1}E_K]\notag\\
%&\qquad-E_K^\top (R+B^\top \tP_{K} B)^{-1}E_K\notag\\ 
%&\geq -E_K^\top (R+B^\top \tP_{K} B)^{-1}E_K,
%\#
%which can be plugged into 
%\eqref{equ:p_diff_lower_bnd} to yield
%\$
%&P_{K'}-P_K\geq \sum_{t\ge 0} [(A-BK')^\top(I-\gamma P_{K}W)^{-1}]^t\big[ -E_K^\top (R+B^\top \tP_{K} B)^{-1}E_K\big]\cdot[(I-\gamma P_{K}W)^{-\top}(A-BK')]^t.
%\$
Multiplying $DD^\top>0$ and 
taking traces on both sides of \eqref{equ:p_diff_lower_bnd_ct} with  $K'=K^*$  gives 
\#\label{equ:trace_diff_lower_bnd_2_ct}
\tr(P_{K}DD^\top)-\tr(P_{K^*}DD^\top)&\leq \tr\left[ (RK-B^\top P_K)^\top R^{-1}(RK-B^\top P_K) \right]\cdot\|\cM_{K,K^*}\|
%\leq \frac{\tr\left( E_K^\top E_K \right)}{\sigma_{\min}(R)}\cdot\|\cM_{K,K^*}\|,
\#
where by a slight abuse of notation, we define $\cM_{K,K^*}$ as 
\$
\cM_{K,K^*}:=\int_{0}^\infty \eb^{(A-BK^*+\gamma^{-2} DD^\top  P_{K}) \tau}\cdot DD^\top \cdot \eb^{(A-BK^*+\gamma^{-2} DD^\top  P_{K})^\top\tau}d\tau. 
\$
%i.e., the solution to the  Lyapunov equation
%\$
%(I-\gamma P_{K}W)^{-\top}(A-BK^*)\cM_{K,K^*}(A-BK^*)^\top(I-\gamma P_{K}W)^{-1}+W=\cM_{K,K^*},
%\$
Note that  $A-BK^*+\gamma^{-2} DD^\top  P_{K^*}$ is Hurwitz. 
Let $\epsilon:=-\max_{i\in[m]}[\Re\lambda_{i}(A-BK^*+\gamma^{-2} DD^\top  P_{K^*})]>0$ be the largest real part of the eigenvalues of $A-BK^*+\gamma^{-2} DD^\top  P_{K^*}$.  
By the continuity of $P_K$, 
% from  Lemma \ref{lemma:differentiability_policy_grad_ct}, 
 there exists a ball $\cB(K^*,r)\subseteq\cK$  centered at  $K^*$ with radius $r>0$,  such that  for any $K\in \cB(K^*,r)$,  
\#\label{equ:linear_rate_trash_0_ct}
\max_{i\in[m]}~[\Re\lambda_{i}(A-BK^*+\gamma^{-2} DD^\top  P_{K})]\leq -\epsilon/2<0. 
\#

\vspace{10pt}
\noindent{\textbf{Gauss-Newton:}}
\vspace{4pt} 

Under $DD^\top>0$, which implies the observability condition in Theorem \ref{theorem:global_exact_conv_ct}, 
 we know that $\{K_n\}$ approaches $K^*=R^{-1}B^\top P_{K^*}$. 
From \eqref{equ:p_trace_upper_bnd_gn_ct}, \eqref{equ:def_bar_omega}, and  \eqref{equ:trace_diff_lower_bnd_2_ct}, with $M$ replaced  by $DD^\top$, we obtain that if some $K_n$ is close enough to $K^*$  such that  $K_n\in\cB(K^*,r)$, then letting $K=K_n$ and $K'=K_{n+1}$, we have 
\begin{align*}
\tr(P_{K'}DD^\top)-\tr(P_KDD^\top)\leq \frac{-\eta\sigma_{\min}(DD^\top)\sigma_{\min}(R)}{2\overline{\omega}_{\cK}\sigma_{\max}(R)\|\cM_{K,K^*}\|}
[\tr(P_{K}DD^\top)-\tr(P_{K^*}DD^\top)],
\end{align*}
which further implies that
\small
\begin{align}\label{equ:linear_rate_trash_1_ct}
\tr(P_{K'}DD^\top)-\tr(P_{K^*}DD^\top)\leq \bigg(1-\frac{\eta\sigma_{\min}(DD^\top)\sigma_{\min}(R)}{2\overline{\omega}_{\cK}\sigma_{\max}(R)\|\cM_{K,K^*}\|}\bigg)\cdot[\tr(P_{K}DD^\top)-\tr(P_{K^*}DD^\top)]. 
\end{align}
\normalsize
By \eqref{equ:linear_rate_trash_1_ct}, the sequence $\{\tr(P_{K_{n+p}}DD^\top)\}_{p\geq 0}$ decreases to $\tr(P_{K^*}DD^\top)$ starting from some $K_n\in \cB(K^*,r)$. By continuity, there must exist a $K_{n+p}$  close enough to $K^*$, such that the lower-level set $\{K\given \tr(P_KDD^\top)\leq \tr(K_{n+p}DD^\top)\}\subseteq \cB(K^*,r)$. Hence, starting from $K_{n+p}$, the iterates  will never leave $\cB(K^*,r)$. Thus,  by \eqref{equ:linear_rate_trash_0_ct}, $\cM_{K,K^*}$, as the unique  solution to the Lyapunov equation
\$
(A-BK^*+\gamma^{-2} DD^\top  P_{K})\cM_{K,K^*}+\cM_{K,K^*}(A-BK^*+\gamma^{-2} DD^\top  P_{K})^\top +DD^\top=0,
\$
must be uniformly bounded by  some constant $\overline{\cM}_{r}>0$ over  $\cB(K^*,r)$. Replacing   $\|\cM_{K,K^*}\|$ in \eqref{equ:linear_rate_trash_1_ct} by $\overline{\cM}_{r}$ gives the uniform local linear contraction of $\{\tr(P_{K_n}DD^\top)\}$.

In addition,  by the upper bound \eqref{eq:CDL_upper_ct} in Lemma \ref{lemma:cost_diff_ct} and $RK^*=B^\top P_{K^*}$, we have
\#\label{equ:trash_1_cont_q_quad}
\tr(P_{K'}DD^\top)-\tr(P_{K^*}DD^\top)&\leq \tr\Big\{\int_{0}^\infty  \eb^{(A-BK'+\gamma^{-2} DD^\top P_{K'})^\top \tau}\cdot [(K'-K^*)^\top R(K'-K^*)]\notag\\&\qquad\cdot \eb^{(A-BK'+\gamma^{-2} DD^\top P_{K'})\tau}d\tau \cdot DD^\top\Big\}. 
\#
For $\eta=1/2$, suppose that some $K=K_n\in\cB(K^*,r)$, then $K'=K_{n+1}=R^{-1}B^\top P_K$, which yields that
\#\label{equ:trash_2_cont_q_quad}
 K'-K^*=R^{-1}B^\top (P_{K} -P_{K^*})\Longrightarrow \|K'-K^*\|_F\leq c\cdot\|P_{K}-P_{K^*}\|_F,
% =\\
% &\quad=[(R+B^\top \tilde P_{K} B)^{-1}-(R+B^\top \tilde P_{K^*} B)^{-1}]B^\top\tilde P_{K} A+[(R+B^\top \tilde P_{K^*} B)^{-1}B^\top(\tP_{K}-\tP_{K^*})A]\notag\\
% &\quad=(R+B^\top \tilde P_{K} B)^{-1}B^\top (\tP_{K}-\tP_{K^*}) B (R+B^\top \tilde P_{K^*} B)^{-1} B^\top\tilde P_{K} A+[(R+B^\top \tilde P_{K^*} B)^{-1}B^\top(\tP_{K}-\tP_{K^*})A]\notag. 
\#
for some constant $c>0$. 
Combining    \eqref{equ:trash_1_cont_q_quad} and \eqref{equ:trash_2_cont_q_quad} gives  
\$
\tr(P_{K'}DD^\top)-\tr(P_{K^*}DD^\top)&\leq c'\cdot [\tr(P_{K}DD^\top)-\tr(P_{K^*}DD^\top)]^2,
\$
for some constant $c'$. Note that from some $p\geq 0$ such that  $K_{n+p}$  onwards never leaves $\cB(K^*,r)$, the constant $c'$ is uniformly bounded, which completes the Q-quadratic convergence rate of $\{\tr(P_{K_n}DD^\top)\}$ around $K^*$.

\vspace{10pt}
\noindent{\textbf{Natural  Policy Gradient:}}
\vspace{4pt} 

Combining \eqref{equ:p_trace_upper_bnd_ng_ct} and \eqref{equ:trace_diff_lower_bnd_2_ct} yields that 
\begin{align*}
\tr(P_{K'}DD^\top)-\tr(P_{K^*}DD^\top)\leq \bigg(1-\frac{\eta\sigma_{\min}(DD^\top)\sigma_{\min}(R)}{2\overline{\omega}_{\cK}\|\cM_{K,K^*}\|}\bigg)\cdot[\tr(P_{K}DD^\top)-\tr(P_{K^*}DD^\top)]. 
\end{align*} 
Using similar arguments as above, one can establish the local linear rate of $\{\tr(P_{K_n}DD^\top)\}$ with a different contraction factor. 
This completes the proof. 
\hfill$\QED$

\clearpage

\section{Auxiliary Results}\label{sec:aux_res}
In this section, we prove several auxiliary results used before. 

\begin{lemma}[Integral of Gaussian Random Variables]\label{lemma:integral_Gaussian}
	Suppose that $z\sim \cN(\bar{z},Z)$. Then,  for any positive semidefinite  matrix $P$ and scalar $\beta$ satisfying $I-\beta P Z>0$, it follows that
	\$
	\frac{2}{\beta}\log\EE\exp\bigg(\frac{\beta}{2}z^\top Pz\bigg)=\bar{z}^\top \tP \bar{z}-\frac{1}{\beta}\log\det(I-\beta PZ),
	\$
	where $\tP= P+\beta P(Z^{-1}-\beta P)^{-1}P$.
\end{lemma}
\begin{proof}
	By definition, we have
	\$
	\EE\exp\bigg(\frac{\beta}{2}z^\top Pz\bigg)&=\frac{1}{(2\pi)^{n/2}(\det Z)^{1/2}}\int e^{\beta x^\top Px/2}e^{-(x-\bar{z})^\top Z^{-1} (x-\bar{z})/2} dx.
	\$
	By completing the squares in the exponent, and using the fact that 
	\$
	\frac{1}{(2\pi)^{n/2}(\det \Sigma)^{1/2}}\int  e^{-(x-\mu)^\top \Sigma^{-1} (x-\mu)/2} dx=1,
	\$
	for any Gaussian random variable following $\cN(\mu,\Sigma)$, 
	we obtain the desired result. 
\end{proof}

\begin{lemma}\label{lemma:optimal}
Let $\cJ^*$   be the minimum  of the limit in \eqref{eq:obj}, and suppose that the following  \emph{modified Riccati equation}  admits a stabilizing fixed-point solution  $P_{K^*}\geq  0$ such that $W^{-1}-\beta P_{K^*}>0$,  
\#
\left\{
                \begin{array}{ll}
P_{K^*}&=~~Q+(K^*)^\top RK^*+(A-BK^*)^\top \tP_{K^*}(A-BK^*)\\
\tP_{K^*}&=~~P_{K^*}+\beta  P_{K^*}(W^{-1}-\beta P_{K^*})^{-1}P_{K^*}\label{equ:def_tP}\\
K^*&=~~(R+B^\top\tP_{K^*} B)^{-1}B^\top \tP_{K^*} A 
                \end{array}
              \right..
\#
Then, we have $\cJ^*=- {\beta}^{-1}\log\det (I-\beta P_{K^*}W)$. 
Moreover, {among all controls that generate a well-defined objective}, 
%limit in \eqref{eq:obj}}, 
the optimal control  is LTI state-feedback  given by $\mu_t(x_{0:t},u_{0:t-1})=-K^*x_t$ for all $t\geq0$. 
\end{lemma}
\begin{proof}
Recall that the original LEQG problem is defined as
\#\label{eq:obj}
\min_{\{\mu_t\}_{t\geq 0}}\quad \lim_{T\to\infty}~~\frac{1}{T}\frac{2}{\beta}\log\EE\exp\bigg[\frac{\beta}{2} \sum_{t=0}^{T-1}c(x_t, u_t) \bigg]~~~\text{with}~~~u_t=\mu_t(x_{0:t},u_{0:t-1}),
\#
where $x_{0:t}$ and $u_{0:t-1}$ denote the history of states from time $0$ to $t$ and actions from time $0$ to $t-1$, respectively. 	
Let $h_t:=(x_{0:t},u_{0:t-1})$. 
% denote the history of states and actions obtained until time $t$.   
For notational convenience, we define 
\$
V_t(x)=x^\top P^t x, \qquad [\cF_{u_t}(f)](x)=c(x,u_t)+\frac{2}{\beta}\log\EE\big\{\exp \big[\frac{\beta}{2}f{(Ax+Bu_t+w)}\big]\big\}
\$
where $P^t\geq  0$ is any nonnegative definite matrix,  $u_t$ is the control at time $t$ that is adapted to the $\sigma$-algebra generated by $h_t$, and $f:\RR^m\to \RR$ can be any function of $x$.  Also, we note that the expectation in $\cF_{u_t}(f)$ is taken over the randomness of $w$  given $x$ and $u_t$.  
Obviously, $\cF_{u_t}(f)$ is a monotone operator, i.e., if some $g$ satisfies $g(x)\geq f(x)$ for any $x$, then we have $[\cF_{u_t}(g)](x)\geq [\cF_{u_t}(f)](x)$ for any $x$.

Consider a $T$-stage    control sequence $(u_0,u_1,\cdots,u_{T-1})$, 
%and define 
%\$
%\tilde{V}_{u_{T-1}}(x):=[\cF_{u_{T-1}}(V)](x)-\frac{2}{\beta}\log\EE\big\{\exp \big[\frac{\beta}{2}{V(Ax+Bu_{T-1}+w_{T-1})}\big]\big\}=c(x,u_{T-1}).
%\$
%\remind{Can we show that the second term is non-expansive, which thus gives us a lower bound, using $V$.}
and notice that 
\$
&[\cF_{u_{T-1}}({V}_{0})](x)=c(x,u_{T-1})+\frac{2}{\beta}\log\EE\big\{\exp \big[\frac{\beta}{2}V_0(Ax+Bu_{T-1}+w_{T-1})\big]\big\}\\
&\quad=\frac{2}{\beta}\log\EE\big(\exp \big\{\frac{\beta}{2}[c(x,u_{T-1})+V_0(Ax+Bu_{T-1}+w_{T-1})]\big\}\biggiven x,u_{T-1}\big). 
\$
Keeping imposing operators $\cF_{u_{T-1}},\cdots \cF_{u_{0}}$  yields 
%\$
%&[\cF_{u_{T-3}}\cF_{u_{T-2}}(\tilde{V}_{u_{T-1}})](x)=c(x,u_{T-3})+\frac{2}{\beta}\log\EE\big\{\exp \big[\frac{\beta}{2}[\cF_{u_{T-2}}(\tilde{V}_{u_{T-1}})](Ax+Bu_{T-3}+w_{T-3})\big]\big\}\\
%&\quad=\frac{2}{\beta}\log\EE\big(\exp \big\{\frac{\beta}{2}[c(x,u_{T-3})+c(x_{T-2},u_{T-2})+c(x_{T-1},u_{T-1})]\big\}\biggiven x,u_{T-1},u_{T-2},u_{T-3}\big),
%\$
%where $x_{T-2}=Ax+Bu_{T-3}+w_{T-3}$ and $x_{T-1}=Ax+Bu_{T-2}+w_{T-2}$. 
%Thus, by keeping imposing operators $\cF_{u_{T-2}},\cdots \cF_{u_{0}}$, we have
\#\label{equ:lemma_11_trash_1}
[\cF_{u_{0}}\cdots \cF_{u_{T-1}}({V}_{0})](x)=\frac{2}{\beta}\log\EE\bigg\{\exp\bigg[\frac{\beta}{2} \sum_{t=0}^{T-1}c(x_t, u_t) +\frac{\beta}{2}V_0(x_T)\bigg]\bigggiven x_0=x,u_{T-1},\cdots,u_{0}\bigg\}.
\#
Note that the RHS of \eqref{equ:lemma_11_trash_1} can be viewed as a $T+1$-stage undiscounted LEQG problem with the first $T$ stages having cost $c(x_t,u_t)$ and the last stage having cost $V_0(x_T)$.
 
On the other hand, for $t=0,\cdots,T-1$, letting $\lambda_t=- {\beta}^{-1}\log\det (I-\beta P^tW)$ and 
\#\label{equ:tP_t_update}
\tP^{t}=P^t+\beta P^t(W^{-1}-\beta P^t)^{-1}P^t,
\# 
we have from Lemma \ref{lemma:integral_Gaussian} that 
\$
&\frac{2}{\beta}\log\EE\big\{\exp \big[\frac{\beta}{2}(Ax+Bu_{T-t-1}+w)^\top P^{t}(Ax+Bu_{T-t-1}+w)\big]\big\}\\
&\quad=(Ax+Bu_{T-t-1})^\top \tP^{t} (Ax+Bu_{T-t-1})+\lambda_t,
\$
from which we obtain that 
\#\label{equ:quadratic_form_T}
[\cF_{u_{T-t-1}}(V_{t})](x)&=c(x,u_{T-t-1})+(Ax+Bu_{T-t-1})^\top \tP^t (Ax+Bu_{T-t-1})+\lambda_t. 
\#
Note that \eqref{equ:quadratic_form_T} is a quadratic function of $u_{T-t-1}$, which has the minimizer
\#\label{equ:K_t_min}
u_{T-t-1}^*=-(R+B^\top\tP^t B)^{-1}B^\top \tP^t Ax=:-K_{t+1}x.
\#
Substituting \eqref{equ:K_t_min} back to \eqref{equ:quadratic_form_T}, we require $V_{t+1}(x)$ to be updated as $
V_{t+1}(x)=[\cF_{u_{T-t-1}^*}(V_{t})](x)-\lambda_t$, 
i.e., we require 
\#\label{equ:P_t_update}
P^{t+1}=Q+(K_{t+1})^\top RK_{t+1}+(A-BK_{t+1})^\top \tP^t(A-BK_{t+1}). 
\#
In particular, \eqref{equ:tP_t_update}, \eqref{equ:K_t_min}, and \eqref{equ:P_t_update} constitute the recursion updates of the modified Riccati equation, which also gives the following relation
\#\label{equ:unnamed_trash}
\lambda_t+V_{t+1}(x)&=\min_{u_t}~~\underbrace{c(x,u_{T-t-1})+\frac{2}{\beta}\log\EE\big\{\exp \big[\frac{\beta}{2}(Ax+Bu_{T-t-1}+w)^\top P^t(Ax+Bu_{T-t-1}+w)\big]\big\}}_{[\cF_{u_{T-t-1}}(V_{t})](x)},
\#
i.e., 
\#\label{equ:ineq_T_prop}
[\cF_{u_{T-t-1}}(V_t)](x)\geq \lambda_t+V_{t+1}(x). 
\#
 In addition, $\cF_{u_{T-t-1}}(V_t)$ also has the property that
\#\label{equ:linear_T_prop}
[\cF_{u_{T-t-1}}(V_t+c)](x)=c+ [\cF_{u_{T-t-1}}(V_t)](x),
\# 
for any constant $c$, 
since $c$ can be taken out of the expectation in $\cF_{u_{T-t-1}}(V)$. 
Combining  \eqref{equ:ineq_T_prop},  \eqref{equ:linear_T_prop}, and the monotonicity of $\cF_{u_t}$ yields 
\#\label{equ:recursion_operator}
[\cF_{u_{0}}\cdots \cF_{u_{T-1}}({V}_{0})](x)&\geq [\cF_{u_{0}}\cdots \cF_{u_{T-2}}(\lambda_0+V_{1})](x)=\lambda_0+[\cF_{u_{0}}\cdots \cF_{u_{T-2}}(V_{1})](x)\notag\\
&\geq  \sum_{t=0}^{T-1}\lambda_t+V_T(x). 
\#
Taking expectation over $x$ and dividing both sides of \eqref{equ:recursion_operator} by $T$, we obtain from  
\eqref{equ:lemma_11_trash_1} that
\#\label{equ:lemma_11_trash_2}
\frac{1}{T}\frac{2}{\beta}\log\EE\bigg\{\exp\bigg[\frac{\beta}{2} \sum_{t=0}^{T-1}c(x_t, u_t) +\frac{\beta}{2}V_0(x_T)\bigg]\bigggiven u_{T-1},\cdots,u_{0}\bigg\}\geq \frac{1}{T}\sum_{t=0}^{T-1}\lambda_t+\frac{1}{T}\EE[V_T(x)]. 
\# 
By letting $V_0(x)=0$ and taking limit $T\to\infty$ on both sides of \eqref{equ:lemma_11_trash_2}, the LHS converges to the objective of LEQG defined in \eqref{eq:obj} (which is assumed to exist for the studied control sequence $(u_0,u_1,\cdots,)$); while the RHS  converges to the value of $\lambda^*:=- {\beta}^{-1}\log\det (I-\beta P_{K^*}W)$, where $P_{K^*}$ is the unique stabilizing solution to the modified Riccati equation given in \eqref{equ:def_tP}, due to Lemma \ref{lemma:recursion_MRE} in \S\ref{sec:aux_res}. 
% that the recursion of \eqref{equ:tP_t_update}, \eqref{equ:K_t_min}, and \eqref{equ:P_t_update} 
%, which  is always upper-bounded by the LHS of \eqref{equ:lemma_11_trash_2},
%   converges to $(P_{K^*},K^*)$, the unique stabilizing fixed-point solution to \eqref{equ:def_tP}. 
Thus, the sequence $\{\lambda_{t}\}$ converges to $\lambda^*$ as $t\to\infty$, and so does the sequence $\{\sum_{t=0}^{T-1}\lambda_t/T\}$ as $T\to\infty$; while  $\EE[V_T(x)]/T$ vanishes to zero since $P^T$ converges to $P_{K^*}$ as $T\to\infty$. 
Hence, we obtain from \eqref{equ:lemma_11_trash_2} that 
\$
\lambda^*\leq \lim_{T\to\infty}~~\frac{1}{T}\frac{2}{\beta}\log\EE\exp\bigg[\frac{\beta}{2} \sum_{t=0}^{T-1}c(x_t, u_t) \bigg]~~~\text{for~any}~~~u_t=\mu_t(x_{0:t},u_{0:t-1}),
\$
where the equality can be obtained when $u_{t}^*=-(R+B^\top\tP_{K^*} B)^{-1}B^\top \tP_{K^*} Ax=-K^*x
$ for all $t\geq 0$. 
In other words, among all controls that achieve the $\lim$ of \eqref{equ:def_obj},  
% a well-defined limit of the general objective, 
 the optimal objective is $\cJ^*=\lambda^*$, and can be  achieved by the \emph{stationary linear state-feedback} control $K^*$ obtained from \eqref{equ:def_tP}, which completes the proof. 
%   of Lemma    \ref{lemma:optimal}. 
\end{proof}

% \subsection{Proof of Lemma  \ref{lemma:LEQR_obj_form_for_K}}\label{sec:proof_lemma_optimal}
%\begin{proof}
%\issue{STILL, WE NEED TO SHOW THAT THE SOLUTION TO THE RICCATI EQUATION EXISTS, FROM FINITE TO INFINITE. I RELEGATE THIS TO ANOTHER LEMMA IN THE APPENDIX, i.e., Lemma \ref{lemma:recursion_MRE}. } 

\begin{lemma}[Restatement of Lemma \ref{lemma:LEQR_obj_form_for_K}]\label{lemma:LEQR_obj_form_for_K_restate}
	For any LTI state-feedback controllers $u_t=-Kx_t$, such that  the Riccati  equation \eqref{equ:def_mod_Bellman_ori}  admits a solution $P_K\geq 0$ that: i) is stabilizing, i.e., $\rho\big((A-BK)^\top(I-\beta P_KW)^{-1}\big)<1$, and ii) satisfies $W^{-1}-\beta P_K>0$,   
%	that induces a finite objective value, suppose the \emph{modified Bellman equations} defined as
%	\#
%P_K&=Q+K^\top RK+(A-BK)^\top \tP_K(A-BK)\label{equ:def_PK}\\
%\tP_K&=P_K+\beta P_K(W^{-1}-\beta P_K)^{-1}P_K,\label{equ:def_tPK}
%\#	
%admits a stabilizing fixed-point solution such that: i) $P_K\geq 0$; ii) $W^{-1}-\beta P_K>0$; iii) $(A-BK)^\top(I-\beta P_KW)^{-1}$ is stable. 
 	 $\cJ(K)$ has the form of 
%\#\label{equ:obj_logdet_form}
\$
\cJ(K)=-\frac{1}{\beta}\log\det (I-\beta P_KW).
\$
%\#
%where $P_K$ is the solution to the modified Bellman equation that satisfies the second condition in Proposition \ref{prop:discrete_equiva_set_cK}. 
\end{lemma}

\begin{proof}
%Moreover, we n
Note that   by  \cite[Proposition $1$]{ionescu1992computing}, the stabilizing solution of the Riccati equation \eqref{equ:def_mod_Bellman_ori}, if it exists, is unique. 
	The proof then proceeds in the similar vein as that for Lemma \ref{lemma:optimal}.  
	In particular,  
	we keep the updates \eqref{equ:tP_t_update} and \eqref{equ:P_t_update}, and replace \eqref{equ:K_t_min} by $u_{T-t-1}^*=-Kx$. These updates constitute the recursion of the Riccati equation  \eqref{equ:def_mod_Bellman_ori}. 
	Then, following the derivation from \eqref{equ:unnamed_trash}-\eqref{equ:recursion_operator}, but replacing the inequality by equality, we obtain that
	\#\label{equ:LEQR_obj_form_trash_1}
	\frac{1}{T}\frac{2}{\beta}\log\EE\bigg\{\exp\bigg[\frac{\beta}{2} \sum_{t=0}^{T-1}c(x_t, u_t) +\frac{\beta}{2}V_0(x_T)\bigg]\bigggiven u_{T-1},\cdots,u_{0}\bigg\}= \frac{1}{T}\sum_{t=0}^{T-1}\lambda_t+\frac{1}{T}\EE[V_T(x)].
	\#
	Taking $T\to\infty$ on both sides of \eqref{equ:LEQR_obj_form_trash_1}, the limit on the LHS reduces  to the definition of $\cJ(K)$, and the limit on the RHS converges to $-\beta^{-1}\log\det(I-\beta P_KW)$. This follows by: i) Lemma \ref{lemma:recursion_MBE} in \S\ref{sec:aux_res}, showing that  the recursion of  \eqref{equ:def_mod_Bellman_ori} starting from $P^0=0$ converges to   $P_K$, which further leads to the convergence of 
	\$
	\lim_{t\to\infty}~~\lambda_t=-\beta^{-1}\log\det(I-\beta P_KW),
	\$
	and ii) $\lim_{T\to\infty}\EE[V_T(x)]/T={0}$ due to the boundedness of $P_K$.  As a result, we obtain the desired form in \eqref{equ:obj_logdet_form}, which completes the proof  of Lemma  \ref{lemma:LEQR_obj_form_for_K}. 
%\end{proof}
%\hfill$\QED$
%\end{proof} 
\end{proof}

%\begin{lemma}[Monotonicity of $\tP_K$]\label{lemma:monotone^tP_from_P}
%	Suppose both $K,K'\in\cK$, if $P_{K'}\geq P_{K}$, then $\tP_{K'}\geq \tP_K$. 
%\end{lemma}
%\begin{proof}
%	Since $K,K'\in\cK$, both $\tP_{K'}$ and $\tP_K$ exist. By definition, we have
%	\$
%	sXX,
%	\$
%	which completes the proof. 
%\end{proof}

%\newlytyped{
\begin{lemma}[Recursion of Discrete-Time Modified Riccati Equation \eqref{equ:def_tP}]\label{lemma:recursion_MRE}
%	Suppose there exists a fixed-point solution $(P_{K^*},K^*)$ for the modified Riccati equation \eqref{equ:def^tP} such that $K^*\in\cK$.  
	Suppose the   modified Riccati equation \eqref{equ:def_tP} admits a stabilizing fixed-point solution $P_{K^*}$ such that $W^{-1}-\beta P_{K^*}>0$. 
	Let $P^0\geq 0$ be some  nonnegative definite matrix satisfying $P^0\leq Q$ and $W^{-1}-\beta P^0>0$. Then, starting from $P^0$, the iterate sequence $\{P^t\}$  from \eqref{equ:recursion_Riccati} below converges  to $P_{K^*}$, 
%	, with $P_{K^*}\geq 0$ being this stabilizing  solution of \eqref{equ:def_tP}, 
	which is  unique.  
%	the modified Riccati equation \eqref{equ:def^tP} admits a fixed-point solution,   
	\#
\left\{
                \begin{array}{ll}
\tP^{t}&=~~P^t+\beta P^t(W^{-1}-\beta P^t)^{-1}P^t\label{equ:recursion_Riccati}\\
K^{t+1}&=~~(R+B^\top\tP^t B)^{-1}B^\top \tP^t A \\
P^{t+1}&=~~Q+(K^{t+1})^\top RK^{t+1}+(A-BK^{t+1})^\top \tP^t(A-BK^{t+1})
                \end{array}
              \right..
\#
%Moreover, 	 
\end{lemma}
\begin{proof}
First note that if $P_K>0$, then  
 $\tP_K$ in \eqref{equ:def_tPK} can also be written as 
\#\label{equ:reform_tPK}
	\tP_K&=(I-\beta P_KW)^{-1}P_K=(P^{-1}_K-\beta W)^{-1},
\#
which further implies the monotonicity of $\tP_K$,  i.e.,  if $P_{K'}\geq P_K$ for some $K,K'$, then $\tP_{K'}\geq \tP_K$. 
If $P^t>0$, then the update in \eqref{equ:recursion_Riccati} can be written as 
\$
&P^{t+1}=\cF(P^t):=Q+A^\top \tP^t A-A^\top (\tP^t)^\top B(R+B^\top\tP^t B)^{-1}B^\top \tP^t A\\
&\quad=Q+A^\top \big((\tP^t)^{-1}+BR^{-1}B^\top\big)^{-1} A=Q+A^\top \big((P^t)^{-1}-\beta W+BR^{-1}B^\top\big)^{-1} A,
\$
where the second equation is basic substitution, the third one uses matrix inversion lemma, and the last one uses  \eqref{equ:reform_tPK}. Obviously, $\cF(P^t)$ is matrix-wise monotonically increasing with respect to $P^t$. 
Note that as a fixed point of $\cF(P)$, $P_{K^*}=\cF(P_{K^*})\geq Q$, while $P^0\leq Q$. Hence, by monotonicity, if $P^0>0$, then $P_{K^*}\geq \cF(P^0)=P^1$. By induction, $P_{K^*}\geq P^t$, for all $t\geq 0$. If $P^0\geq 0$, then \$
	\tP^{0}=P^0+\beta P^0(W^{-1}-\beta P^0)^{-1}P^0\geq 0,\quad P^1=Q+(K^1)^\top RK^1+(A-BK^1)^\top \tP^{0}(A-BK^1) \geq Q>0,
	\$ 
	and we can thus apply the argument above starting from $t\geq 1$; else if $P^0>0$, we can apply it directly from $t \geq 0$. 
In addition, 
since from \eqref{equ:recursion_Riccati},  $P^1\geq Q\geq P^0$, we have $P^2=\cF(P^1)\geq P^1=\cF(P^0)$. By induction, $P^{t+1}\geq P^{t}$ for all $t\geq 0$. 
Note that $W^{-1}-\beta P^t$ is well defined for all $t$, due to that $W^{-1}-\beta P_{K^*}>0$ by assumption.  
Since the sequence $\{P^t\}$ is monotone and upper-bounded, we conclude that  the recursion of \eqref{equ:recursion_Riccati} must converge to some $P^{\infty}$, which constitutes a fixed-point solution to the modified Riccati equation \eqref{equ:def_tP}.

In addition, by Lemma $3.1$ in \cite{stoorvogel1994discrete}, we know that the stabilizing fixed-point solution, i.e., the $P_{K^*}$ that makes $\rho\big((I-\beta P_{K^*}W)^{-\top}(A-BK^*)\big)<1$,  is unique. 
Also, by Lemma $3.8$ therein, any other fixed-point solution $P$ that makes $A-B(R+B^\top \tP B)^{-1}B^\top \tP A$ stable, which by Lemma $3.4$ therein is a necessary condition for $P$ to be a stabilizing solution,  
satisfies that $P\geq P_{K^*}$. Notice that any fixed-point solution is greater than $Q$, and thus greater than $P^0$. Hence, starting from $P^0$  converges to the minimal  fixed point, i.e., the unique stabilizing one, which completes the proof. 
\end{proof}
%}

\begin{lemma}[Recursion of Riccati Equation \eqref{equ:def_mod_Bellman_ori}]\label{lemma:recursion_MBE}
Suppose that control gain $K$  yields a stabilizing solution $P_K$, i.e., $\rho\big((A-BK)^\top(I-\beta P_KW)^{-1}\big)<1$, to the  Riccati equation \eqref{equ:def_mod_Bellman_ori}, such that $W^{-1}-\beta P_K>0$.  
 Then, such a solution is \emph{unique}, and \emph{minimal} among all Hermitian  solutions. In addition, 
%stabilizing  control gain $K$, suppose  there exists  fixed-point solutions  to the modified Bellman equation \eqref{equ:def_PK}-\eqref{equ:def^tPK} that satisfy 
if $P_K\geq 0$,  
%Then the stabilizing solution $P_K$  that makes  $(A-BK)^\top(I-\beta P_KW)^{-1}$  stable is unique. In addition, 
%that satisfies $W^{-1}-\beta P_K>0$ and $(A-BK)^\top(I-\beta P_KW)^{-1}$ is stable 
%where $P_K\geq 0$ is the fixed-point solution to the Modified Bellman equation \eqref{equ:def_PK}-\eqref{equ:def^tPK}, 
the following  recursion   starting from any $P^0$ such that $0\leq P^0\leq Q$ and  $W^{-1}-\beta P^0>0$ converges to  $P_K$,   
	\#
\left\{
                \begin{array}{ll}
\tP^{t}&=~~P^t+\beta P^t(W^{-1}-\beta P^t)^{-1}P^t\label{equ:recursion_Bellman}\\
P^{t+1}&=~~Q+K^\top RK+(A-BK)^\top \tP^t(A-BK)
                \end{array} 
              \right..
\#
%The convergent point $(P^{\infty},K_{\infty})$ is the \issue{unique?} fixed-point solution $(P,K^*)$ of the modified Riccati equation \eqref{equ:def^tP}. 
\end{lemma} 
\begin{proof}
First, by \cite[Theorem $3.1$]{ran1988existence}, if one chooses $C=0$, $R<0$ therein, by reversing the sign in the theorem, and noticing that the pair $(A-BK,W^{1/2})$ is stabilizable since $W>0$, we know that there exists a maximal solution $-P^-\geq -P$ for any Hermitian solution $-P$ to the Riccati equation, which is also stabilizing. Hence, there exists a minimal solution $P^-$ that is also stabilizing. This conclusion can be directly applied to \eqref{equ:def_mod_Bellman_ori}, with $R=-1/\beta W^{-1}<0$ and since $1/\beta W^{-1}-P_K>0$. Also, by  \cite[Proposition $1$]{ionescu1992computing}, the stabilizing solution of the Riccati  equation \eqref{equ:def_mod_Bellman_ori}, if it exists, is unique. Thus, $P_K$,  as the unique stabilizing solution, is also minimal, i.e., $P_K\leq P$ for any solution to \eqref{equ:def_mod_Bellman_ori}. 

On the other hand, if $P^t>0$, we  rewrite the update \eqref{equ:recursion_Bellman} as
\$
P^{t+1}&=\cF_K(P^t):=Q+K^\top RK+(A-BK)^\top [(P^t)^{-1}-\beta W]^{-1}(A-BK),
\$
which is matrix-wise monotone with respect to $P^t$. Since $P_K \geq Q\geq P^0$, by induction, $\cF_K(P_K)=P_K\geq \cF_K(P^t)=P^{t+1}$ for all $t\geq 0$. If $P^0\geq 0$, then \$
	\tP^{0}=P^0+\beta P^0(W^{-1}-\beta P^0)^{-1}P^0\geq 0,\quad P^1=Q+(K^1)^\top RK^1+(A-BK^1)^\top \tP^{0}(A-BK^1) \geq Q>0,
	\$
	and the conclusion above also holds for $t\geq 1$; else it holds for $t \geq 0$. In addition,  since $P^1\geq Q\geq P^0$, by induction, $P^{t+1}\geq P^t$ for all $t\geq 0$. Note that the matrix $W^{-1}-\beta P^t$ is always invertible along the recursion, by the assumption that $W^{-1}-\beta P_K>0$. Therefore, by both the monotonicity and boundedness of the sequence $\{P^t\}$, the recursion \eqref{equ:recursion_Bellman} converges to some fixed-point solution to \eqref{equ:def_mod_Bellman_ori}. Since $P_K$ is the minimal one,  the sequence $\{P^t\}$ thus converges to $P_K$, which concludes the proof. 
\end{proof}

\clearpage

\section{Pseudocode}\label{sec:pseudo_code}
In this section, we provide the pseudocode  of the model-free algorithms mentioned in \S\ref{subsec:model_free}. 
Particularly, Algorithm \ref{alg:est_grad_corre} estimates the policy gradient $\nabla_K
  \cC(K,L)$ and the correlation matrix $\Sigma_{K,L}$ for any stabilizing  $(K,L)$; Algorithm \ref{alg:model_free_inner_NPG} finds an estimate of the maximizer $L(K)$ in \eqref{equ:minimizer_L_given_K} for a given $K$; Algorithm \ref{alg:model_free_outer_NPG} describes the updates of $K$ for finding an estimate of $K^*$.

\begin{algorithm}[!thpb]
	\caption{\textbf{Est($L$;$K$)}: Estimating  ${\nabla}_L
  \cC(K,L)$ and $ \Sigma_{K,L}$ at $L$ for given   $K$} 
	\label{alg:est_grad_corre}
	\begin{algorithmic}[1]
		\STATE Input: $K,L$, number of trajectories $m$, rollout length $\cR$,
                smooth parameter $r$, dimension $\tilde d=m_2 d$
%		\INPUT stuff
		\FOR{$i = 1, \cdots m$}
		\STATE Sample a policy $\widehat L_i = L+U_i$, with $U_i$ 
                drawn uniformly  over matrices with $\|U_i\|_F=r$
		\STATE Simulate $(K,\widehat L_i)$ for $\cR$ steps starting
                from $x_0\sim \cD$, and collect the empirical estimates $\widehat \cC_i$ and $\widehat \Sigma_i$ as:
\[
\widehat \cC_i = \sum_{t=1}^\cR c_t \, , \quad \widehat \Sigma_i = \sum_{t=1}^\cR x_t x_t^\top
\]
where $c_t$ and $x_t$ are the costs and states following  this trajectory
%		\STATE Update:
		\ENDFOR
		\STATE Return the estimates:
\[
\widehat{\nabla}_L
  \cC(K,L) = \frac{1}{m} \sum_{i=1}^m \frac{\tilde d}{r^2} \widehat \cC_i U_i
\, , \quad
\widehat \Sigma_{K,L} = \frac{1}{m} \sum_{i=1}^m \widehat \Sigma_i
\]
	\end{algorithmic}
\end{algorithm}

\begin{algorithm}[!hpb] 
	\caption{\textbf{Inner-NG($K$)}: Model-free updates  for estimating $L(K)$} 
	\label{alg:model_free_inner_NPG}
	\begin{algorithmic}[1]
		\STATE Input: $K$, number of iterations $\cT$,  initialization $L_0$ such that $(K,L_0)$ is stabilizing 
%		\INPUT stuff
		\FOR{$\tau = 0, \cdots, \cT-1$}
		\STATE Call \textbf{Est($L_\tau$;$K$)} to obtain the gradient  and the correlation matrix estimates:
		\$
		[\widehat{\nabla}_L \cC(K,L_\tau),\widehat \Sigma_{K,L_\tau}]=\textbf{Est}(L_\tau;K) 
		\$
		\STATE  Policy gradient update:
		$\quad\quad\quad\quad~
		L_{\tau+1} = L_\tau+\alpha \widehat{\nabla}_L \cC(K,L_\tau), 
		$\\
		or 
		natural PG update:
		$\qquad\qquad\quad
		L_{\tau+1} = L_\tau+\alpha \widehat{\nabla}_L \cC(K,L_\tau)\cdot \widehat \Sigma_{K,L_\tau}^{-1}
		$
		\ENDFOR
		\STATE Return the iterate  $L_{\cT}$
%\[ 
%\widehat{\nabla_K
%  C(K,L)} = \frac{1}{m} \sum_{i=1}^m \frac{d}{r^2} \widehat C_i U_i
%\, , \quad
%\widehat \Sigma_{K,L} = \frac{1}{m} \sum_{i=1}^m \widehat \Sigma_i .
%\]
	\end{algorithmic}
\end{algorithm}

\begin{algorithm}[!t]
	\caption{\textbf{Outer-NG}: Model-free  updates for estimating $K^*$} 
	\label{alg:model_free_outer_NPG}
	\begin{algorithmic}[1]
		\STATE Input: $K_0$, number of trajectories $m$, number of iterations $T$, rollout length $\cR$, 
                 parameter $r$, dimension $\tilde d=m_1d$
%		\INPUT stuff
		\FOR{$t = 0, \cdots, T-1$}
		\FOR{$i = 1, \cdots m$}
		\STATE Sample a policy $\widehat K_i = K_t+V_i$, with $V_i$ 
                drawn uniformly   over matrices with $\|V_i\|_F=r$
        \STATE Call \textbf{Inner-NG}($\widehat K_i$) to obtain the estimate of $L(\widehat K_i)$:
        \[
        \widehat{L(\widehat K_i)} = \textbf{Inner-NG}(\widehat K_i)
        \]
		\STATE Simulate $\big(\widehat K_i,\widehat{L(\widehat K_i)}\big)$ for $\cR$ steps starting
                from $x_0\sim \cD$, and collect the empirical estimates $\widehat C_i$ and $\widehat \Sigma_i$ as:
\[
\widehat \cC_i = \sum_{t=1}^\cR c_t \, , \quad \widehat \Sigma_i = \sum_{t=1}^\cR x_t x_t^\top
\] 
where $c_t$ and $x_t$ are the costs and states following  this trajectory 
%		\STATE Update:
		\ENDFOR
		\STATE Obtain the estimates of the gradient and the correlation  matrix:
\[
\widehat{\nabla}_K
   \cC(K_t,\widehat{L(K_t)}) = \frac{1}{m} \sum_{i=1}^m \frac{\tilde d}{r^2} \widehat \cC_i V_i
\, , \quad
\widehat \Sigma_{K_t,\widehat{L(K_t)}} = \frac{1}{m} \sum_{i=1}^m \widehat \Sigma_i 
\]
\STATE Policy gradient update:
		$\quad\quad\quad~
		K_{t+1} = K_t-\eta \widehat{\nabla}_K
   \cC(K_t,\widehat{L(K_t)}), 
		$\\
		or 
		natural PG update:
		$\qquad\qquad
		K_{t+1} = K_t-\eta \widehat{\nabla}_K
   \cC(K_t,\widehat{L(K_t)})\cdot \widehat \Sigma_{K_t,\widehat{L(K_t)}}^{-1}
		$
%  or natural NG 
%update $L_t$ by $$L_{t+1} = \PP_{\Omega}\bigg[L_t + \eta \widehat{\nabla_L
%  \tilde \cC(L_t)}\cdot \widehat \Sigma_{\widehat{K(L_t)},L_t}^{-1}\bigg]$$
		\ENDFOR
		\STATE Return the iterate $K_T$.	
		\end{algorithmic}
\end{algorithm}

\end{document}